\documentclass[11pt]{amsart}
\usepackage{amsmath,amsthm,amssymb,hyperref,fancyhdr,mathtools}
\usepackage{derivative}
\usepackage{color}
\usepackage{comment}
\usepackage{enumitem}
\usepackage{fullpage}
\usepackage{derivative}
\usepackage{multirow}
\usepackage{booktabs} 
\usepackage{float}

\theoremstyle{plain}
\newtheorem{theorem}{Theorem}
\newtheorem*{theorem*}{Theorem}
\newtheorem{lemma}[theorem]{Lemma}
\newtheorem{proposition}[theorem]{Proposition}
\newtheorem{corollary}[theorem]{Corollary}
\newtheorem{definition}[theorem]{Definition}
\newtheorem{example}[theorem]{Example}
\newtheorem{remark}[theorem]{Remark}
\numberwithin{theorem}{section}
\numberwithin{equation}{section}

\newcommand{\ip}[2]{\ensuremath{\left\langle#1,#2\right\rangle}}

\newcommand{\B}{\mathbb{B}}
\newcommand{\C}{\mathbb{C}}
\newcommand{\D}{\mathbb{D}}

\newcommand{\R}{\mathbb{R}}
\newcommand{\Z}{\mathbb{Z}}
\newcommand{\N}{\mathbb{N}}
\newcommand{\T}{\mathbb{T}}

\newcommand{\cA}{\mathcal{A}}
\newcommand{\cB}{\mathcal{B}}

\newcommand{\cF}{\mathcal{F}}
\newcommand{\cH}{\mathcal{H}}

\newcommand{\cP}{\mathcal{P}}
\newcommand{\cS}{\mathcal{S}}
\newcommand{\cX}{\mathcal{X}}

\DeclareMathOperator{\spann}{span}
\DeclareMathOperator{\diag}{diag}

\newcommand{\GL}{\mathrm{GL}}
\newcommand{\SO}{\mathrm{SO}}
\newcommand{\SL}{\mathrm{SL}}
\newcommand{\SU}{\mathrm{SU}}
\newcommand{\U}{\mathrm{U}}

\newcommand{\Sp}{\mathrm{Sp}}

\newcommand{\Tr}{\operatorname{Tr}}
\newcommand{\tr}{\operatorname{tr}}

\newcommand{\ga}{\mathfrak{a}}

\renewcommand{\gg}{\mathfrak{g}}
\newcommand{\gh}{\mathfrak{h}}

\newcommand{\gk}{\mathfrak{k}}

\newcommand{\gn}{\mathfrak{n}}

\newcommand{\gs}{\mathfrak{s}}
\newcommand{\gt}{\mathfrak{t}}

\newcommand{\bm}{\mathbf{m}}

\newcommand{\id}{\mathrm{Id}}

\newcommand{\End}{\mathrm{End}}

\newcommand{\bsd}{\mathcal{D}}
\newcommand{\Uvar}{Z}
\newcommand{\Zvar}{Z}
\newcommand{\domdpi}[2]{\mathrm{Dom}(\mathrm{d}#1(#2))}
\newcommand{\diffops}{\rm{Diff}}
\newcommand{\avar}{\lambda}

\newcommand{\Berg}{\cA^2_\nu(\bsd)} 

\newcommand{\cbuGL}{C_{b,u}^{(L)}(G)} 
\newcommand{\cbuGR}{C_{b,u}^{(R)}(G)} 

\newcommand{\cbufL}{C_{b,u}^{(L)}(\bsd)} 
\newcommand{\cbufR}{C_{b,u}^{(R)}(\bsd)} 

\newcommand{\bddf}{L^\infty(\bsd)} 

\newcommand{\bddfG}{L^\infty(G)} 
\newcommand{\Lone}{L^1(\bsd,d\mu_0)} 
\newcommand{\Lp}{L^p(\bsd,d\mu_0)} 
\newcommand{\Lq}{L^{p'}(\bsd,d\mu_0)} 

\newcommand{\LoneG}{L^1(\G)} 

\newcommand{\LpG}{L^p(G)}
\newcommand{\domop}{\mathcal{D}_{\mathrm{op}}^r}

\newcommand{\cbuopL}{C_{b,u}^{(L)}(\cA^2_\nu)}

\newcommand{\bdd}{\cB(\cA^2_\nu)}

\newcommand{\traceop}{\cS^1(\cA^2_\nu)}
\newcommand{\schop}{\cS^p(\cA^2_\nu)}

\newcommand{\schopq}[1]{\cS^{#1}(\cA^2_\nu)}

\newcommand{\toepalg}{\mathfrak{T}_\nu(L^\infty)}

\newcommand{\haar}[1]{\ d\mu_{G}(#1)}
\newcommand{\haarK}[1]{\ d\mu_K(#1)}
\newcommand{\inv}[1]{\ d\mu_0(#1)}

\newcommand{\dv}{d\mu_\nu}
\newcommand{\actopL}[2]{\lambda^\nu_{#1}({#2})}
\newcommand{\actfL}[2]{\ell_{#1}#2}

\newcommand{\actfR}[2]{r_{#1}#2}

\newcommand{\ActopL}[1]{\lambda^\nu_{#1}}

\newcommand{\act}[2]{\tau_{#1}(#2)}

\newcommand{\Last}{\ast}
\newcommand{\last}{\ast_\ell}

\newcommand{\G}{G}

\newcommand{\pinu}{\pi_\nu}

\newcommand{\ipnu}[2]{\ensuremath{\left\langle#1,#2\right\rangle}_\nu}
\newcommand{\normnu}[1]{\|#1\|_\nu}

\newcommand{\ddt}{\frac{\mathrm{d}}{\mathrm{d}t}}
\newcommand{\del}[1]{\frac{\partial}{\partial #1}}
\newcommand{\supp}{\operatorname{supp}}
\newcommand{\hc}{\textbf{c}}
\newcommand{\gen}{g}
\newcommand{\Ad}{\mathrm{Ad}}

\newcommand{\Rad}[1]{{#1}^\#}
\newcommand{\rad}[1]{{#1}^\#}

\newcommand{\Pna}{\Phi^\nu_\alpha}
\newcommand{\GO}{\Gamma_\Omega}
\newcommand{\su}{\mathfrak{su}}
\renewcommand{\sl}{\mathfrak{sl}}

\newcommand{\rU}{\mathrm{U}}
\newcommand{\rS}{\mathrm{S}}


\author{Matthew Dawson}
\address{CONAHCYT---CIMAT Unidad Mérida,
          Parque Científico y Tecnológico de Yucatán,
          KM 5.5 Carretera Sierra Papacal --- Chuburná Puerto,
          Sierra Papacal; Mérida, YUC 97302, México}
\email{matthew.dawson@cimat.mx}

\author{Vishwa Dewage}
\address{Faculty of Mathematics, University of Vienna, Oskar-Morgenstern-Platz 1, 1090 Wien, Austria}
\email{vishwa.dewage@univie.ac.at}

\author{Gestur \'Olafsson}
\address{Department of Mathematics, Louisiana State University, Baton Rouge, LA 70803, USA.} 
\email{olafsson@math.lsu.edu}


\thanks{Vishwa Dewage was supported by Austrian Science Fund (FWF) [10.55776/ESP2699225].}

\title[QHA on Bergman spaces of bounded symmetric domains]{Quantum harmonic analysis and the Differential operators on weighted Bergman spaces of bounded symmetric domains}

\begin{document}

\begin{abstract}
We develop quantum harmonic analysis (QHA) for operators on the parametrized Bergman space $\mathcal{A}^2_\nu(\Omega)$ of a bounded symmetric domain $\Omega$, extending the theory previously known for the unit ball. Building on this framework, we introduce a 
class of differential operators acting on operators on $\mathcal{A}^2_\nu(\Omega)$. Combining these tools, we advance the theory of Toeplitz operators on $\mathcal{A}^2_\nu(\Omega)$: we establish Toeplitz approximation results via the $\alpha$-Berezin transform, 
construct a differential operator $D_\alpha$ that shifts the $\alpha$-Berezin transform, and obtain several characterizations of the radial Toeplitz algebra -- including one via the uniformly continuous radial operators, and one via the domain of $D_\alpha$. \end{abstract}

\maketitle

\setcounter{tocdepth}{1}
\tableofcontents

\section{Introduction}\label{sec:intro}
The importance of weighted Bergman spaces $\Berg$ (where $\nu>\gen-1$, and $\gen$ is the genus of $\bsd$) in analysis on a bounded symmetric domain $\Omega=G/K$ of genus $\gen$ has been clear since the discovery, by Harish-Chandra, of the holomorphic discrete series representations of the Hermitian Lie group $G$ on $\Berg$ (see~\cite{H54}). The most well-studied example of a bounded symmetric domain is the unit ball $\B^n$ in the complex inner-product space $\C^n$. Backed by the rich geometry and representation theory of $\Omega$, the theory of Toeplitz operators on $\Berg$ has been developed extensively (see, for example, \cite{E99, FK90, U83} and subsequent works). 

One of our main goals is to revisit Toeplitz operators on $\Berg$ from the perspective of Werner's quantum harmonic analysis (QHA) \cite{W84}, and to further develop the theory. In addition, we introduce a notion of differential operators acting on operators on $\Berg$, another concept that naturally fits within the scope of QHA.

In QHA, natural concepts such as the convolution and the Fourier transform are defined and studied for operators in place of 
functions, providing an operator-theoretic analogue of harmonic analysis.  In terms of representation theory, in QHA one studies representations of groups on spaces of operators using terminology and constructions which are inspired by the theory of group representations that act on spaces of functions.  Following several pivotal works that discovered connections between time-frequency analysis and the theory of Toeplitz operators \cite{F19, LS18, LS20}, there has been a rapid growth in the application of QHA within these areas \cite{LS25, DLMS24,  LM24}. These developments highlight QHA as a unifying framework that naturally accommodates objects such as Toeplitz operators, Wigner transforms, and pseudo-differential operators. As observed in \cite{F19} for the Fock space, Toeplitz operators and their counterpart, the Berezin transform, can be expressed as QHA convolutions, making QHA a powerful tool in the theory \cite{BF20, BFT25, DM24a, F19, FH23, FR23}. The QHA framework was extended to the unweighted Bergman space $\cA^2(\B^n)$ of the unit ball in \cite{DDMO24}, and in \cite{DFO25}, the notion of Fourier transform was developed for radial operators on weighted Bergman spaces. Some QHA concepts were extended to the weighted Bergman spaces $\Berg$ over a bounded symmetric domain in \cite{D25} to study the Toeplitz algebra. In the first part of this article, we systematically develop the QHA framework for $\Berg$, extending the results of \cite{DDMO24} to bounded symmetric domains.

Su\'arez introduced the $\alpha$-Berezin transform of operators on the Bergman space $\cA^2(\D)$ of the unit disc, and investigated when the approximation $T_{B_\alpha(S)}\to S$ holds for an operator $S$ on $\cA^2(\D)$ \cite{S04, S05, S08}. There have been many generalizations; see, for example, \cite{BHV14, BHV14b} for the unit ball, and \cite{D17} for bounded symmetric domains. Following the QHA approach in \cite{DDMO24}, we systematically investigate the $\alpha$-Berezin transform of operators and and obtain several new results. We use it to characterize the Toeplitz operators as follows (see Theorem \ref{theo:Toeplitz criterion}) :

\begin{theorem*}
    Let $1<p\leq \infty$ and let $S\in \schop$. Then the following are equivalent:
    \begin{enumerate}
        \item $S=T^\nu_a$ for some $a\in \Lp$.
        \item There exists $C>0$ such that $\|B^\nu_\alpha(S)\|_p < C$ for all $\alpha\in\N_0$.
    \end{enumerate}
\end{theorem*}
\noindent Here, $\schop$ denotes the space of operators on $\Berg$ of Schatten class $p$ when $1<p<\infty$ and denotes the space $\cB(\Berg)$ of all bounded operators on $\Berg$ when $p=\infty$, and $\mu_0$ denotes the (unique up to constant multiple) $G$-invariant measure on $\Omega=G/K$.

Later, in Theorem \ref{theo:radial covolution convergence}, we specify necessary conditions for the convergence $T_{B_\alpha(S)}\stackrel{\alpha\rightarrow\infty}{\longrightarrow} S$ to hold in different topologies. Utilizing this, we show that the Toeplitz algebra $\toepalg$ is generated by Toeplitz operators with right-uniformly continuous symbols $\nu$ is sufficiently large (see Corollary \ref{coro:Toeplitz algebra}). 
Then, restricting to radial operators, we present several characterizations of the radial Toeplitz algebra, one of these shows that it coincides with the algebra of uniformly continuous radial operators (see Theorem \ref{theo: main theorem} for the full statement):

\begin{theorem*}
     We can characterize the radial Toeplitz algebra as follows:
    \begin{align*}
        \toepalg^{K}= \cbuopL^{K} 
        = \overline{\{T^\nu_a\mid a\in \cbufR^{K}\}}^{\|\cdot\|}.
    \end{align*}
\end{theorem*}

Differential operators have played a role in both QHA and the theory of Toeplitz operators. For example, the Laplacian of an operator on $\cA^2(\D)$, introduced by Su\'arez in \cite{S04} via the Berezin transform, was used to describe the commutative radial Toeplitz algebra \cite{S08}, which was later used in \cite{GMV13} to analyze its Gelfand theory. In the second part of the article, we set out to develop differential operators acting on operators on the Bergman space $\Berg$. We take two complementary approaches: the approach of Su\'arez, and a more representation theoretic viewpoint, taken for example in the abelian setting in \cite{L24, FL25}. Among the two, Su\'arez's approach provides a wider framework. 

Throughout the paper, we make extensive and explicit use of a certain differential operator $D_\alpha$ that turns out to be extremely useful for studying Toeplitz operators and QHA on bounded symmetric domains.  This operator can be found implicitly in the paper \cite{U94}, in the sense that its construction is an immediate consequence of the results in that paper, although the authors never mention it. 

In our first such application of this differential operator, we show that $D_\alpha$ shifts the $\alpha$-Berezin transform of a function (see Theorem \ref{theo:Berezin shift} for the complete statement).

\begin{theorem*}
   Let $\alpha>\gen-1$ and $1\leq p\leq \infty$.  For all $f\in \Lp$, 
   \[
      B_{\alpha+1}f = D_\alpha (B_\alpha f)
   \]
\end{theorem*}

Analogously, the differential operator $D_{\nu+\alpha}$ shifts the Berezin transform of operators on $\Berg$ (see Theorem \ref{theo:Berezin shift op} for the complete statement):
\begin{theorem*}
   Let $\nu,\alpha>\gen-1$ and $1\leq p\leq \infty$.  For all $S\in \schop$, 
   \[
      B^\nu_{\alpha+1}(S) = D_{\nu+\alpha} (B^\nu_{\alpha} (S)).
   \]
\end{theorem*}

Once we restrict to $\cA^2_\nu=\Berg$ considered over a classical bounded symmetric domain, the above leads to a discrete quantum analog of the heat equation: for $S\in \bdd$, where ,
\begin{align}\label{eq:QHA heat}
    \varphi_{\alpha+1}\ast S-\varphi_{\alpha}\ast S = (D_\alpha-\id)(\varphi_\alpha\ast S).
\end{align}

In the case where $\bsd$ is a domain of rank one (that is, when $\bsd$ is the unit ball $\B^n$ for some $n\in\N$), the operator $D_\alpha$ can be expressed in terms of the identity and the invariant Laplacian of the unit ball $\B^n$. These operators are precisely the generators of the algebra of all invariant differential operators acting on functions on $\B^n$. Naturally, in the higher rank case, $D_\alpha$ becomes a polynomial in the $r$ differential operators that generate the algebra of invariant differential operators on $\bsd$, where $r$ is the rank of the domain $\bsd$. Then, making use of \eqref{eq:QHA heat}, we obtain another characterization of the radial Toeplitz algebra via the common domain $\domop$ of the $r$ differential operators mentioned above (see Theorem \ref{theo:main theorem two}):

\begin{theorem*}
    Consider $\cA^2_\nu=\Berg$ where $\bsd$ is a classical bounded symmetric domain, and $\nu>\gen-1$. Then
    $$(\overline{\domop}^{\|\cdot\|})^K=\cbuopL^K=\toepalg^K.$$
\end{theorem*}

The above result agrees with what is known for the unit disc \cite{S08}. In particular, it was used in \cite{GMV13} to address Gelfand theory of the commutative Toeplitz algebra for the Bergman space over the unit disc.

The structure of the paper is as follows.  In Section~\ref{sec:prelim}, we review basic results on the structure of and analysis on bounded symmetric domains, and in Section~\ref{sec:QHA}, we review basic results on QHA on bounded symmetric domains. Next, in Section~\ref{sec:alpha_berezin_section}, we introduce the $\alpha$-Berezin transform for functions on a bounded symmetric domain $\bsd$ and develop its basic properties.  In Section~\ref{sec:invariant_differential_operators}, as an aid to the reader, as well as an opportunity to fix the relevant notation, we briefly review the theory of invariant differential operators on Riemannian symmetric spaces. Then, in Section~\ref{sec:Berezin_shifts}, we construct the differential operator $D_\alpha$ and show how to use it to shift the Berezin transform of an operator or function.  Next, in Section~\ref{sec:Diffops on ops}, we show how to apply invariant differential operators (such as $D_\alpha$) to certain bounded operators on the Bergman space $\Berg$, using tools from representation theory. Finally, in Section~\ref{sec:heateq}, we use these tools to show that the Berezin transform of a radial Toeplitz operator satisfies a discrete-time generalized heat equation.  This result will allow us to give a characterization of the $C^*$-algebra generated by radial Toeplitz operators that is analogous to the characterization given by Suárez in \cite{S04,S05,S08} for the unit ball using the operator Laplacian. We conclude with an appendix which reviews the definition and root data for the classical bounded symmetric domains.     


\section{Preliminaries}\label{sec:prelim}

\subsection{Bounded symmetric domains}
A domain $\bsd\subseteq\C^n$ is said to be symmetric if for each $z\in \bsd$, there is a biholomorphic map $\sigma_z:\bsd\to\bsd$ of period two for which $z$ is a fixed point. From now on, we assume that $\bsd\subset\C^n$ is a bounded symmetric domain and that $G$ is the connected component of the automorphism group that contains the identity.
Suppose that $G$ is of rank $r$ and has restricted root multiplicities $a$ and $b$, where we are following the notation used in \cite{FK90} (more precisely, this means that the roots for the restricted root system of $\gg$ all have multiplicity $1$, $a$, or $2b$). See, for instance, \cite{H78,H63,K69,L69,L77,S80}), for more information on the structure and classification of bounded symmetric domains.  For a brief overview of the classical (that is, non-exceptional) bounded symmetric domains, we refer the reader to the appendix.

%
%
Without loss of generality, we can suppose that $0\in \bsd$, and that $\bsd$ is a convex, circled, open, and bounded subset of $\C^n$.  We let $K$ denote the stabilizer
\[
    K=\{x\in G\mid x\cdot 0 =0\}.
\]
One shows that $K$ is a maximal compact subgroup of $G$.  Because $G$ acts transitively on $\bsd$, one has that $G/K$ and $\bsd$ are isomorphic as $G$-spaces, where the identification is made via the map $xK\mapsto x\cdot 0$. Using the Harish-Chandra embedding, we can also assume, without loss of generality, that $K$ acts on $\bsd$ by the restriction to $\bsd$ of unitary operators on $\C^n$.

We denote by $\theta$ the Cartan involution of $G$ fixing $K$.  Then $\theta$ factorizes to a holomorphic involution 
\begin{align*}
    \sigma: \bsd \rightarrow \bsd ,\
    x\cdot 0 \mapsto \theta(x) \cdot 0.
\end{align*}
for all $x\in G$ with derivative at $0$ given by $(d\bsd)_0 = -\id \in \End(T_0)$, where $T_0$ is the tangent space to $\bsd$ at $0$.  (This involution is well-defined because $\theta(k)=k\in K$ for all $k\in K$.)  

Furthermore, for each $z\in \bsd$, we denote by $\sigma_z : \bsd \rightarrow \bsd$ the holomorphic involution fixing $z$.  Then $\sigma_z$ can be written in terms of $\theta$ as follows:
\begin{align*}
    \sigma_z : \bsd  \rightarrow \bsd, \
                x\cdot 0  \mapsto x_0 \theta(x_0^{-1}x) \cdot 0,
\end{align*}
where $x_0 \in G$ is such that $z = x_0\cdot 0$.

Given $z\in \bsd$, we define the biholomorphism $\tau_z$ by setting $\tau_z=\sigma_{z_0}$, where $z_0$ is the midpoint of the geodesic connecting $0$ and $z$. Then $\tau_z$ is an involution that interchanges $z$ and $0$.

\begin{lemma}\label{lem:Aut}
   Let $x\in G$. Then $x=\tau_{x\cdot 0}k_x$ for some $k_x\in K$.
\end{lemma}
\begin{proof}
   Note that $(\tau_{x\cdot 0}x)(0)=\act{x\cdot 0}{x\cdot 0}=0$. Hence  $\tau_{x\cdot 0}x=k_x$ for some $k_x\in K$. As $\tau_x^{-1}=\tau_x$ it follows that $x=\tau_{x\cdot 0}k_x$.
\end{proof}

Now let $\gen:=(r-1)a+b+2$ denote the genus of $\bsd$.  The complex dimension of $\bsd$ is given by $n=r(\frac{r-1}{2}a+b+1)$, and the complex dimension of the largest tube-type subdomain is given by $n_1=r(\frac{r-1}{2}a+1)$.  

We let $\overline{\bsd}$ denote the K\"ahler manifold with the same underlying real-analytic manifold as $\bsd$ but the opposite complex structure.  That is, as sets, one has that $\bsd = \overline{\bsd}$, the only difference being that the holomorphic functions on $\overline{\bsd}$ are precisely the antiholomorphic functions on $\bsd$.  

The Jordan triple determinant $h:\bsd\times\overline{\bsd}\rightarrow \C$ associated with $\bsd$ is a polynomial that is holomorphic in the first entry and an antiholomorphic in the second. This encodes many of the important properties of the associated Bergman space. Recall that $\bsd$ is equipped with $r$ mutually orthogonal vectors $\{e_1, \dots,e_r\}$, called a Jordan frame, such that for every $w\in \bsd$, there exist $k_w\in K$ and $0\leq t_1(w),\ldots,t_r(w)<1$ such that
\begin{equation}\label{eq:jframe}
    w=k_w\Big(\sum_{j=1}^r t_j(w) e_j\Big).
\end{equation}
One can show that
\begin{equation}
   \label{eq:formula_for_Jordan_triple_system}
   h(w,w) = \prod_{j=1}^r (1-t_j(w)^2)
\end{equation}
for all $w\in\bsd$.

We denote by $j:G\times \bsd \rightarrow \bsd$ the complex Jacobian of the action of $G$ on $\bsd$.  Then $j$ satisfies the cocycle relation:
\[
    j(xy, z) = j(x, y\cdot z) j(y,\cdot z)
\]
for all $gx,y\in G$ and $z\in\bsd$.  Furthermore, one can show that $j(x,\cdot)$ admits a holomorphic $g$th root $j(x,\cdot)^{1/g}$ for all $x\in G$.  

Below, we collect some of the important properties of the Jordan triple determinant:

\begin{lemma}
   \label{lem:triple_determinant_properties}
   The polynomial $h$ satisfies the following: for all $z,w, \xi\in \bsd$, $x\in G$, and $k\in K$, one has that
   \begin{enumerate}
      \item $h(z,w)\neq 0$
      \item $h(z,0)=1$
      \item $h(z,w)=\overline{h(w,z)}$
      \item $h(x\cdot z, x\cdot w) = j(x,z)^{-1/\gen} h(z,w) \overline{j(x,w)^{-1/\gen}}$
      \item $h(k\cdot z, k\cdot w) = h(z,w)$
      \item $h(x\cdot 0, x\cdot w) = j(x,0)^{-1/\gen} \overline{j(x,w)^{-1/\gen}}$
      \item $h(x\cdot 0, x\cdot 0) = |j(x,0)|^{-2/\gen}$
      \item $h(\act{\xi}{z},\act{\xi}{w})=\frac{h(\xi,\xi)h(z,w)}{h(z,\xi)h(\xi,w)}$.
   \end{enumerate}
\end{lemma}
\begin{proof}
   Statements (1) through (4) are standard in the literature (see, for instance, \cite{FK90}).  Statement (5) is a particular case of (4), using the fact that $j(k,w)=j(k,z)\in S^1$.  Statement  (6) is an immediate consequence of (4) and (2), and (7) is a particular case of (6).
\end{proof}

From this point on, we will identify the Jordan triple determinant with a function on $G\times G$, which we also denote by $h$, and write $h(x,y):=h(x\cdot 0, y\cdot 0)$ for $x,y\in G$. Then by statement (8) in Lemma \ref{lem:triple_determinant_properties} we have the following:
\begin{lemma}\label{lem:JTdet}
    For $x,y_1,y_2\in G$, the following holds:
    \begin{align*}
   h(xy_1,xy_2)=\frac{h(x,x)h(y_1,y_2)}{h(y_1,x)h(x,y_2)}.
\end{align*}
\end{lemma}

\begin{lemma}
   \label{lem:complex_jacobian_properties}
   For all $k\in K$, one has that $|j(k,0)|=1$.  Furthermore, the map
   \begin{align*}
      \chi: K  \rightarrow S^1, \ 
      k  \mapsto j(k,0)
   \end{align*}
   is a central character of $K$.  Finally, if $k\in K$, then $j(k,w) = j(k,0)$ for all $w\in\bsd$.  
\end{lemma}
\begin{proof}
   Let $k\in K$.  By statement (6) of Lemma~\ref{lem:triple_determinant_properties}, we have that
   \[
      1 = h(0,0) = h(k\cdot 0, k\cdot 0) = |j(k,0)|^{-2/\gen}.
   \]
   Now let $k_1,k_2\in K$.  From the cocycle property, we have that
   \[
      j(k_1k_2, 0) = j(k_1, k_2\cdot 0) j(k_2, 0) = j(k_1, 0) j(k_2,0).
   \]

   Finally, let $k\in K$ and $w\in \bsd$.  From statement (5) of Lemma~\ref{lem:triple_determinant_properties}, we have that 
   \[
      1 = h(k\cdot 0, w) = j(k,0)^{-1/\gen} \overline{j(k,w)^{-1/\gen}}.
   \]
   Since $|j(k,0)|=1$, it follows that $j(k,w) = j(k,0)$.
\end{proof}

\subsection{Functions on $\bsd$ as functions on the group}
Recall that we make the identification $G/K=\bsd$ via the map $xK\mapsto x\cdot 0$. We define a $G$-invariant measure $\mu_0$ on $\bsd$ by
\[
   \inv{z}=h(z,z)^{-\gen}dz,
\]
where $g$ is the genus of $\bsd$.
A function on the domain $f:\bsd\to \C$ can be identified with the function $\Tilde{f}:\G\to \C$ on $G$ by setting
\[
   \Tilde{f}(x)=f(x\cdot 0),\quad x\in G.
\]
Since the $G$-invariant measures on a homogeneous space are unique up to constant multiple, we can choose a suitable normalization of the Haar measure on $G$ such that, for $f\in \Lone$,
\begin{equation}\label{eq:integral}
    \int_{G} \Tilde{f}(x) \haar{x}=\int_{\bsd} f(z) \inv{z}.
\end{equation}

\subsection{The Bergman space}
Fix $\nu>\gen-1$ and define the measure
\[
   d\mu_\nu(z)=c_\nu h(z,z)^{\nu-\gen} dz = c_\nu h(z,z)^\nu \inv{z}
\]
where $dz$ denotes the normalized Lebesgue measure on $\bsd$, and $c_\nu>0$ is chosen such that $c_\nu\int_\bsd \dv=1$. Note that we have
\begin{equation}\label{eq:integral_nu}
    c_\nu\int_{G} \Tilde{f}(x) h(x,x)^{\nu} \haar{x}=\int_{\bsd} f(z) d\mu_\nu(z).
\end{equation}

It is well known that
\begin{equation}\label{eq:cnu}
     c_\nu = \frac{\Gamma_\Omega(\nu)\Gamma_\Omega(\gen-n/r)}{\Gamma_\Omega(\nu-n/r)\Gamma_\Omega(\gen)},
\end{equation}
where, as before, $n$ is the complex dimension of $\bsd$ and $r$ is its rank.  Here, $\Gamma_\bsd$ is the celebrated Gindinkin Gamma function \cite{G64}), and it is given by
\[
   \Gamma_\Omega(\alpha) = (2\pi)^{n_1-r}\prod_{j=1}^r \Gamma(\alpha_j-(j-1)\frac{a}{2}).
\]
for all $(\alpha_1,\ldots,\alpha_j)\in\C^r$ for which all the Gamma functions in the product are defined\footnote{Note that when $\alpha \in \C$, we define
\[
    \Gamma_\Omega(\alpha) := \Gamma_\Omega((\alpha,\ldots,\alpha))
\]}.  
It follows that $c_\nu$ is a rational function of $\nu$.

The Bergman space $\Berg$ consists of all holomorphic functions that are square integrable with respect to $\dv$. The inner product and norm in $\Berg$ are denoted by $\ipnu{\cdot}{\cdot}$ and $\normnu{\cdot}$ respectively. The Bergman space $\Berg$ is a reproducing kernel Hilbert space with reproducing kernel $K^\nu$ given by
\[
   K^\nu(w,z)=K_z^\nu(w)=h(z,w)^{-\nu},\ \ z,w\in \bsd.
\]
We have the reproducing formula
\[
   f(z)=\ipnu{f}{K^\nu_z}
\]
for all $f\in \Berg$ and $z\in \bsd$.
The normalized reproducing kernels $k_z^\nu$ are given by
\begin{equation}
   \label{eq:definition_of_normalized_reproducing_kernel}
   k_z^\nu(w)=\frac{K_z^\nu(w)}{\normnu{K_z^\nu}}=\frac{h(z,z)^{\nu/2}}{h(z,w)^{\nu}}
\end{equation}
for all $z, w\in\bsd$.
The Bergman space $\Berg$ is a closed subspace of the space of all square-integrable functions $L^2(\bsd,\dv)$. The Bergman projection $P^\nu:L^2(\bsd,\dv)\to \Berg$ is given by
\[
    (P^\nu f)(z)=\ipnu{f}{K_z^\nu},\ \ z\in \bsd,\ \ f\in L^2(\bsd,\dv).
\]

The algebra of bounded operators on $\Berg$ is denoted by $\bdd$. For $1\leq p< \infty$, let $\schop$ denote the space of Schatten $p$-class operators. For $p=\infty$, we make the identification $\cS^\infty(\cA^2_\nu)=\bdd$.
Given a function $a:\bsd\to \C$ the Toeplitz operator $T^\nu_a$ with symbol $a$ (which we will write as $T^\nu_a$ when there is no confusion) is defined by 
\begin{equation*}
    T^\nu_af:=P^\nu(af),\ \ f\in \Berg.
\end{equation*}
If $1\leq p\leq \infty$ and $a\in \Lp$, then $T^\nu_a\in \schop$ with $\|T^\nu_a\|_p\leq \|a\|_p$.
The \textit{Toeplitz algebra} $\toepalg$ 
is the $C^*$-subalgebra of $\bdd$  generated by the Toeplitz operators with symbols in $\bddf$.

The Berezin transform of $S\in \bdd$ is the function $B^\nu(S):\bsd\to \C$ defined by
\begin{equation}\label{eq:operator Berezin transform}
    B^\nu(S)=\ipnu{Sk_z^\nu}{k^\nu_z},\quad z\in\bsd.
\end{equation}

Moreover, the parametrized Berezin transform $B_\nu(a):\bsd\to \C$ of a function  $a:\bsd\to \C$, is given by
\begin{equation}\label{eq:function Berezin transform}
    B_\nu(a)=\ipnu{ak_z^\nu}{k^\nu_z},\quad z\in\bsd.
\end{equation}
The Toeplitz quantization map $a\mapsto T_a$ from $\Lp$ to $\bdd$, as well as its counterpart, the Berezin transform $B^\nu:\bdd\to \bddf$, are both injective maps.

The theory of Toeplitz operators on bounded symmetric domains is well established; see, for example, \cite{E99, DOQ15, DOQ18, D25, D17, H17, H19, U83, X15} for some of the literature.


\subsection{Polynomials on $\bsd$}

For each $1\leq j \leq r$, we define the conical polynomial $\Delta_j : \bsd \rightarrow \C$ by:
\[
   \Delta_j(w) = \prod_{l=j}^l t_l
\]
where $w=k\cdot (\sum_{l=1}^r t_l e_l)$ with $k\in K$ and $0\leq t_1,\ldots, t_r \leq 1$.  That this function is a polynomial follows from a theorem of Chevalley \cite[Cor.\ II.5.12]{H84}.  We similarly define, for each $\textbf{s} = (s_1,\ldots,s_r)\in \C^r$, the function $\Delta^{\textbf{s}}: \bsd\rightarrow \C$ by
\[
   \Delta^{\textbf{s}}:= \Delta_1^{s_1-s_2} \Delta_2^{s_2-s_3} \cdots \Delta_{r-1}^{s_{r_1}-s_r} \Delta_r^{s_r}.
\]
It is clear that $\Delta^{\textbf{m}}$ is a polynomial whenever $\textbf{m}\geq 0$, that is, when $\textbf{m} \in \Z^r$ with $m_1 \geq \cdots \geq m_r \geq 0$.  

Given a function $f:G\to \C$, the left translation of $f$ by $x\in G$, denoted $\actfL{x}{f}$ is given by
\begin{equation}\label{eq:translation}
    \actfL{x}{f}(y)=f(x^{-1}y), \quad y\in G.
\end{equation}

Let $\cP$ denote the space of holomorphic polynomials on $\bsd$.  Then $\cP$ is $K$-invariant under the action \eqref{eq:translation}.
This is due to the fact that the action of $K$ on $\bsd\subseteq \C^n$ is a restriction of the linear action of $K$ on $\C^n$. 

For each $\textbf{m}\geq 0$, we define $\cP_{\textbf{m}}$ to be the $K$-invariant space of polynomials generated by $\Delta^{\textbf{m}}$.  That is,
\[
   \cP_{\textbf{m}} := \spann\{ \actfL{k}{\Delta^{\textbf{m}}} \mid k\in K\}. 
\]

\begin{theorem}{\cite[Theorem 2.1]{FK90}}\label{theo:FK radial decomposition}
   Let $\nu > \gen-1$.  The space of $K$-finite vectors for $\cA^2_\nu$ coincides with the space $\cP$ of polynomials on $\bsd$.  Furthermore,
   \[
      \cP = \bigoplus_{\textbf{m}\geq 0} \cP_{\textbf{m}},
   \]
   where each space $\cP_{\textbf{m}}$ carries an irreducible representation of $K$.  Moreover, the representations on 
   $\cP_{\textbf{m}}$ and $\cP_{\textbf{n}}$ are inequivalent whenever 
   $\textbf{m}\neq \textbf{n}$.
\end{theorem}

We  also need the following, see \cite[Thm.\ 2.6 and Thm.\ 2.7]{U83} as well as \cite[(3.13)]{U21}

\begin{theorem}[Upmeier 1983]\label{thm:dimPm}
Let $\bm=(m_1,\ldots , m_r)\in \N_0^r$, $m_1\ge m_2\ge \cdots \ge m_r\ge 0 $. Let $u_r=1+ a(r-1)/2$ and
$v_r = 1+ a(r-1)/2 +b = u_r+b$. Then with
\[C(r,a,b) =\frac{\GO \left(u_r\right)}{\GO\left(v_r\right)} 
\frac{1}{\prod_{1\le i<j\le r} \frac{a(j-1)}{2}(1+\frac{a(j-i-1)}{2})_{a-1}}\] 
we have 
\[\dim \cP_\bm = C(r,a,b)   \frac{\GO (v_r +\bm)}{\GO (u_r + \bm)}  \prod_{1\le i< j\le r} \left(m_i-m_j +\frac{a(j-i)}{2}\right)
\frac{\Gamma \left(m_i-m_j +\frac{a(j-i+1)}{2}\right)}{\Gamma \left(m_j -m_j +1+ \frac{a(j-i -1)}{2}\right) }.\]
\end{theorem}

The Fischer inner product $\ip{\cdot}{\cdot}_F$ on $\cP$ is defined by:
\[
   \langle p, q \rangle_F := \left. p\left(\frac{\partial}{\partial z}\right)\overline{q}(z)\right|_{z=0}, \quad p,q\in \cP.
\]
    In fact, one can show that the Fischer inner product conincides with the Fock-space inner product for $\C^n$:
\[
   \langle p, q \rangle_F = \frac{1}{\pi^n} \int_{\C^n} p(z)\overline{q(z)} e^{-||z||^2} dz.
\]
(see \cite[p.\ 73]{FK90}.)  It is not difficult to see that this inner product is $K$-invariant.

Now let $\textbf{m}\geq 0$.  Since the representation of $K$ on each space $\cP_{\textbf{m}}$ is irreducible, it follows (by Schur's Lemma) that every $K$-invariant inner product on $\cP_{\textbf{m}}$ is a positive scalar multiple of the Fischer inner product. 

We note that, since $K^\nu(k\cdot z, k\cdot w) = K^\nu(z,w)$ for all $k\in K$ and $z,w\in \bsd$, the inner product on $\cA^2_\nu$ is $K$-invariant with respect to the natural left-translation action of $K$ on functions on $\bsd$.  Thus, for each value of $\textbf{m}\geq 0$ and $\nu>\gen-1$, the inner product of $\cA^2_\nu$ and the Fischer inner product must be the same, up to a positive scalar multiple, when restricted to $\cP_\textbf{m}$. In fact, it was shown in \cite{FK90} that the Fischer inner product and the inner product $\ipnu{\cdot}{\cdot}$ on $\cA^2_\nu$ satisfy: \begin{equation}
   \label{eq:relationship_between_Fischer_and_Bergman_inner_products}
   \langle f,g\rangle_\nu = \frac{\Gamma_\Omega(\nu)}{\Gamma_\Omega(\nu+\textbf{m})}  \langle f,g \rangle_F
\end{equation}
for all $f,g\in \cP_\textbf{m}$.

This property is very useful whenever it is necessary to make comparisons between vectors and operators defined for Bergman spaces $\cA^2_\nu$ with different values of $\nu$, as we will wish to do several times in the sequel.

\subsection{A projective unitary representation of $G$ on the Bergman space}
As before, we fix $\nu>\gen-1$.  If $\nu\in\N$, then we define the \textbf{(scalar-type) holomorphic discrete series representation} $\sigma_\nu$ by:
\begin{equation}
   \label{eq:holomorphic_discrete_series_definition}
   \sigma_\nu(x) f(z) = j(x^{-1},z)^{-\nu/\gen} f(x^{-1} \cdot z).
\end{equation}
for all $x\in G$, $f\in \Berg$ and $z\in\bsd$.
This representation is unitary and irreducible.
If $\nu\notin\N$, then $j(x,\cdot)^{-\nu/\gen}:\bsd \rightarrow \C$ does not have a holomorphic branch for all $x\in G$.  However, one can lift $j^{-\nu/\gen}$ to a smooth function on $\widetilde{G}\times\bsd$ that is holomorphic in the second variable (here, $\widetilde{G}$ is the universal covering group of $G$), since $\widetilde{G}\times\bsd$ is simply connected.  

This can be done as follows: let $q:\widetilde{G}\rightarrow G$ be the covering homomorphism that exhibits $\widetilde{G}$ as the universal covering of $G$.  One defines a transitive, holomorphic action of $\widetilde{G}$ on $\bsd$ by $x\cdot z := q(x)\cdot z$ for all $x\in \widetilde{G}$ and all $\in \bsd$).  Using the basepoint $0\in\bsd$ as before, this gives the identification $\bsd\cong\widetilde{G}/\widetilde{K}$, where $\widetilde{K}=q^{-1}(K)$.  Then the complex Jacobian of the action of $\widetilde{G}$ is given by the function
\begin{align*}
   \widetilde{j}:\widetilde{G}\times\bsd  \rightarrow\C, \
                     (x,z)  \mapsto j(q(x),z). 
\end{align*}
Lemma~\ref{lem:triple_determinant_properties} then holds without change when one uses $\widetilde{G}$ and $\widetilde{j}$ in place of $G$ and $j$, respectively.

Since $\widetilde{G}\times\bsd$ is simply connected, one defines $\widetilde{j}^{\nu/\gen}$ to be the unique holomorphic branch that is positive at $(e,0)$. One then defines, for all $\nu > \gen-1$, the unitary representation:
\[
   \widetilde{\sigma}_\nu(x) f(z) = \widetilde{j}(x^{-1},z)^{-\nu/\gen} f(x^{-1} \cdot z).
\]
for all $g\in\widetilde{G}$, $f\in\Berg$, and $z\in\bsd$.
It is well known that this defines an irreducible unitary representation $\widetilde{\sigma}_\nu$ of $\widetilde{G}$ on $\Berg$, also referred to as a (scalar-type) holomorphic discrete series representation of $\widetilde{G}$.

Nevertheless, it is often more convenient to work directly with the group $G$ rather than its universal covering $\widetilde{G}$.  In the case of a bounded symmetric domain, $\widetilde{G}$ is a semsimple Lie group with infinite center, which has several undesirable properties (for instance, the subgroup $\widetilde{K}$ of $\widetilde{G}$ that plays the role of the maximal compact subgroup in the Cartan and Iwasawa decompositions of $\widetilde{G}$ is not compact).  For values of $\nu$ that are not necessarily integers, one can consider the corresponding holomorphic discrete series representation as a projective unitary representation of $G$ rather than as a unitary representation of $\widetilde{G}$.  This is essentially the perspective that we will adopt in this paper.

First, for each $x\in \widetilde{G}$, we define an operator $\widetilde{\pi}_\nu(x)$ on $\Berg$ by
\begin{equation}
   \widetilde{\pi}_\nu(x)f(z):= k_{x\cdot 0}(z)f(x^{-1}\cdot z) 
\end{equation}
for all $f\in\Berg$ and all $z\in \bsd$.  We then have the following:

\begin{lemma}
   \label{lem:pi_tilde_is_a_projective_rep}
   For all $x\in\widetilde{G}$, 
   \[
   \widetilde{\pi}_\nu(x)=J_\nu(x,0) \widetilde{\sigma}_\nu(x),
   \]
   where $J_\nu:\widetilde{G}\times\bsd\rightarrow S^1$ is the cocycle defined by:
   \[
      J_\nu(x,z) := \left(\frac{\widetilde{j}(x,z)}{|\widetilde{j}(x,z)|}\right)^{-\nu/\gen}
   \]
   for all $x\in\widetilde{G}$ and $z\in\bsd$. Since $J_\nu(x,0)\in S^1$ for all $x\in\widetilde{G}$, we see that $\widetilde{\pi}_\nu(x)$ is a unitary operator for all $x\in\widetilde{G}$ and that, in fact, $\widetilde{\pi}_\nu$ is an irreducible projective unitary representation of $\widetilde{G}$ on $\Berg$ that is equivalent to $\widetilde{\sigma}_\nu$.
\end{lemma}

\begin{proof}
   Recalling the definition in (\ref{eq:definition_of_normalized_reproducing_kernel}) of the normalized reproducing kernel, we simply note that:
   \begin{align*}
      \widetilde{\pi}_\nu(x)f(z) & = h(x\cdot 0, x\cdot 0)^{\nu/2}h(z,x\cdot 0)^{-\nu} f(x^{-1}\cdot z) \\
                   & = |\widetilde{j}(x,0)|^{-\nu/\gen} h(z,x\cdot 0)^{-\nu} f(x^{-1}\cdot z) \\
                   & = |\widetilde{j}(x,0)|^{-\nu/\gen} \overline{\widetilde{j}(x,0)^{\nu/\gen}} \widetilde{j}(x, x^{-1} \cdot z)^{\nu/p} h(x^{-1}\cdot z,0)^{-\nu} f(x^{-1}\cdot x) \\
                   & = \left(\frac{\widetilde{j}(x,0)}{|\widetilde{j}(x,0)|}\right)^{-\nu/\gen}\widetilde{j}(x^{-1}, z)^{-\nu/\gen} f(x^{-1}\cdot z)
   \end{align*}
   for all $f\in\Berg$, $z\in\bsd$, and $x\in \widetilde{G}$, where we have used Lemma~\ref{lem:triple_determinant_properties} several times. 
\end{proof}

The projective representation $\pinu$ of $G$ that we define below is the one relevant for the rest of the article. We use it to formulate QHA on $\Berg$, as in \cite{D25}.

\begin{corollary}
   For each $x\in G$, we define the operator $\pinu(x)$ on $\Berg$ by:
   \begin{align}\label{eq:rep}
       \pinu(x)f(z) = k_{x\cdot 0}^\nu(z) f(x^{-1}\cdot z)
   \end{align}
   for all $f\in\Berg$ and $z\in \bsd$.  Then $\pinu:G \rightarrow \U(\Berg)$ is an irreducible projective unitary representation of $G$ satisfying $\widetilde{\pi}_\nu(x) = \pi_\nu(q(x))$ for all $x\in\widetilde{G}$. 
\end{corollary}
\begin{proof}
   The fact that $\widetilde{\pi}_\nu = \pinu \circ q$ is obvious from their respective definitions.  The rest of the Lemma is then an immediate consequence of the fact that, by the Lemma~\ref{lem:pi_tilde_is_a_projective_rep}, $\widetilde{\pi}_\nu$ is an irreducible projective unitary representation.
\end{proof} 

We now calculate the multiplier associated to this projective unitary representation:
\begin{proposition}
   \label{prop:multiplier}
   For all $x,y\in G$, one has that
   \[
     \pinu(x)\pinu(y)=m_\nu(x,y)\pinu(xy),
   \]
   where $m_\nu(x,y)=\frac{|h(x,y)|^\nu}{h(x,y)^\nu}=\frac{K^\nu(x,y)}{|K^\nu(x,y)|}\in \T$. 
\end{proposition}

\begin{proof}
   To verify that $\pinu$ is a projective representation, let $x, y \in G$ and $f\in \Berg$. Then
   \begin{align*}
      (\pinu (xy)f)(w)=k^\nu_{xy\cdot 0} (w)f((xy)^{-1}\cdot w)=\frac{h(xy,xy)^{\nu/2}}{h(xy\cdot 0, w)^\nu}f((xy)^{-1}\cdot w).
   \end{align*}
   and
   \begin{align*}
      (\pinu(x)\pinu(y)f)(w)&=k^\nu_{x\cdot 0}(w)k^\nu_{y\cdot 0}(x^{-1}w) f((xy)^{-1}\cdot w)\\
      &=\frac{h(x,x)^{\nu/2}}{h(x\cdot 0, w)^\nu}\frac{h(y,y)^{\nu/2}}{h(y\cdot 0, x^{-1}w)^\nu} f((xy)^{-1}\cdot w)\\
      &=\frac{h(xy,xy)^{\nu/2}|h(x,y)|^\nu}{h(x\cdot 0, w)^\nu h(y\cdot 0, x^{-1}w)^\nu} f((xy)^{-1}\cdot w)\quad \text{(by   lemma \ref{lem:JTdet})}.
   \end{align*}
   Note that $x=\tau_{x\cdot 0}\circ k_x$ for some $k_x\in K$ by Lemma \ref{lem:Aut}. Then $x^{-1}=k_x^{-1}\tau_{x\cdot 0}$ and we have
   \begin{align*}
      h(y\cdot 0, x^{-1}w)&=h(x^{-1}xy\cdot 0, x^{-1}w)\\
      &=h(\act{x\cdot 0}{xy\cdot 0}, \act{x\cdot 0}{w})\quad \text{(by the $K$-invariance of $h$)}\\
      &= \frac{h(x,x)h(xy\cdot 0,w)}{h(xy,x)h(x\cdot 0, w)}\\
      &=\frac{h(x,y)h(xy\cdot 0,w)}{h(x\cdot 0, w)}.
   \end{align*}
   Hence
   \begin{align*}
      (\pinu(x)\pinu(y)f)(w)= \frac{h(xy,xy)^{\nu/2}|h(x,y)|^\nu}{h(xy\cdot 0, w)^\nu h(x,y)^\nu} f((xy)^{-1}\cdot w)=m_\nu(x,y)(\pinu (xy)f)(w). 
   \end{align*}
\end{proof}


\section{Quantum harmonic analysis on the Bergman space} \label{sec:QHA}

In this section, we discuss quantum harmonic analysis (QHA) on the Bergman space $\Berg$, fixing the parameter $\nu>\gen-1$.
The main concepts are the convolution between a function and an operator, as well as the convolution between two operators. Halvdansson extended some of the main ideas in Werner's QHA \cite{W84}, originally introduced in the abelian setting of $\R^{2n}$,  to general locally compact groups. For unimodular groups, the arguments simplify considerably. Here we present an overview of the theory, omitting the proofs, and refer to \cite{H23} and \cite[Appendix A]{DDMO24} for details. We then proceed to discuss the additional features of QHA that are applicable in the case of $G$ acting on the Bergman space $\Berg$, generalizing the results in \cite[Section 3]{DDMO24} to bounded symmetric domains. Some of the QHA theory for $\Berg$ was also discussed in \cite{D25}. 

\subsection{Convolutions of functions}
Recall that the left translation of a function \eqref{eq:translation} $a:G\to \C$ by $x\in G$, is given by
\[
   (\actfL{x}{a})(z)=a(x^{-1}\cdot z), \quad z\in \bsd.
\]
Given two functions $\psi:G\to \C$ and $a:\bsd\to \C$, the convolution $\psi\ast a:\bsd\to \C$ is defined formally by
\[
   (\psi\ast a)(z)=\int_{G} \psi(x)(\actfL{x}{a})(z)\haar{x}=\int_{G} \psi(x)a(x^{-1}\cdot z)\haar{x}
\]
for all $z\in G$ for which the integral converges.
Let $p,p',r\in[1,\infty]$ be such that $\frac{1}{p}+\frac{1}{p'}=1+\frac{1}{r}$. Then, by the Young's inequality, if $\psi\in L^p(G)$ and $a\in L^{p'},(\bsd,d\lambda)$ $\psi\ast a\in \bddf$ with
\[
   \|\psi\ast a\|_r\leq \|\psi\|_{p}\|a\|_{p'}.
\]

We say that a function $\psi:\bsd\to \C$ is \textbf{radial} if $\actfL{k}{\psi}=\psi$ for all $k\in K$. Radial functions on $\bsd$ are precisely the $K$-bi-invariant functions on $G$. The radialization of a function $\psi:\bsd\to \C$, is given by
$$\rad{\psi}=\int_K \actfL{k}{\psi} \haarK{k},$$
where $\mu_K$ is the normalized Haar measure on $K$. Then a function $\psi:\bsd\to \C$ is radial if and only only if $\rad{\psi}=\psi$. Given a space of functions $V$, we denote the space of all radial (i.e., $K$-invariant) elements in $V$ by $V^K$. It is well known that the convolution of radial functions is commutative and hence $\Lone^K$ is a commutative Banach algebra, i.e. $\bsd$ is a commutative space. The following lemma supports this fact.

\begin{lemma}\label{lem:radial_functions}
   Let $\psi:\bsd\to \C$ be radial. Then, for all $x\in G$,
   \begin{enumerate}
      \item $(\actfL{x}{\psi})(z)=\psi(\act{x\cdot 0}{z}),\quad z\in \bsd$.
      \item $\Tilde{\psi}(x)=\Tilde{\psi}(x^{-1})$.
   \end{enumerate}
\end{lemma}

For $\alpha>\gen-1$, we define the function $\varphi_\alpha:\bsd\to \C$ by
\begin{equation}\label{eq:phinu}
    \varphi_\alpha(z)=c_\alpha h(z,z)^\alpha
\end{equation}
for all $z\in\bsd$.  These functions $\varphi_\alpha$ are radial and play an important role in what follows. In particular, the parametrized Berezin transform of a function can be written as a convolution with $\varphi_\alpha$:

\begin{lemma}\label{lem:function Berezin transform convol}
    Let $1\leq p\leq \infty$ and let $a\in \Lp$. Then for $\alpha>\gen-1$
    $$ B_\alpha(a)=a\ast \varphi_\alpha.$$
\end{lemma}
\begin{proof}
    Let $x\in G$ and let $z=x\cdot 0$. Note that
    \begin{align*}
    B_\alpha(a)(z)&= \ipnu{ak^\alpha_z}{k^\alpha_z}\\
    &=\int_\bsd a(w)\frac{h(z,z)^\alpha}{|h(z,w)|{2\alpha}} \mu_\alpha(w)\\
    &= c_\alpha \int_\bsd a(w)\frac{h(z,z)^\alpha h(w,w)^{\alpha}}{|h(z,w)|^{2\alpha}} \inv{w}.
\end{align*}
    Also, by Lemma \ref{lem:JTdet}, we have 
    $$\actfL{y}{\varphi_\alpha}(z)=c_\alpha \frac{h(z,z)^\alpha h(y\cdot 0,y\cdot 0)^{\alpha}}{|h(z,y\cdot 0)|^{2\alpha}}, \quad y\in G.$$
    Therefore,
    \begin{align*}
        B_\alpha(a)(z)
        &= \int_G a(y\cdot 0)\actfL{y}{\varphi_\alpha}(z) \haar{y}\quad \text{(by \eqref{eq:integral})}\\
        &= (a\ast \varphi_\alpha)(z). \qedhere
    \end{align*}
\end{proof}

\subsection{Translations and convolutions of operators}
The projective unitary representation $\pi_\nu$ gives rise to a genuine unitary representation on $\bdd$ that acts by conjugation, which we refer to as the translation of operators.
For $x\in G$, we define the operator $\ActopL{x}:\bdd\to \bdd$ by
\begin{equation}\label{eq:translationops}
    \actopL{x}{S} := \pi_\nu(x)S\pi_\nu(x)^*.
\end{equation}
It is not difficult to see that $\ActopL{x}$ is an isometry with respect to the operator norm on $\bdd$.  

Furthermore, although $\pi_\nu:G\rightarrow\U(\cA^2_\nu)$ is, in general, merely a projective unitary representation of $G$, it turns out that $\ActopL{x}:G\rightarrow\GL(\bdd)$ is a genuine representation of $G$:  

\begin{lemma}
   If $x,y\in G$, then
   \[
   \ActopL{xy} = \ActopL{x}\circ \ActopL{y}.
   \]
\end{lemma}
\begin{proof}
   Using Proposition~\ref{prop:multiplier}, we see that for all $S\in\bdd$ and $x,y\in G$, we have that
   \begin{align*}
       \actopL{xy}{S} &= \pi_\nu(xy)S\pi_\nu(xy)^{-1}\\
       &=m_\nu(x,y)\pi_\nu(xy) S (m_\nu(x,y)^{-1} \pi_\nu(xy)^{-1})\\
       &=\pi_\nu(x)\pi_\nu(y)S\pi_\nu(x)^*\pi_\nu(y)^*\\
       &= \actopL{x}{\actopL{y}{S}}
   \end{align*}
   where we are using that $m_\nu(x,y)\in \T$. 
\end{proof}

One also shows that, for all $1\leq p < \infty$, the Schatten $p$-class $\cS^p(\cA^2_\nu)$ of operators is invariant under the action of $\ActopL{x}$.  Furthermore, the representation
\[
   \ActopL{x} : G \rightarrow \GL(\cS^p(\cA^2_\nu))
\]
is strongly continuous, where $\cS^p(\cA^2_\nu)$ is considered with Schatten $p$-norm.  When $p=\infty$, one has that $\cS^\infty(\cA^2_\nu)=\bdd$, and the representation
\[
   \ActopL{x} : G\rightarrow \GL(\bdd)
\]
is \textit{not} strongly continuous when $\bdd$ is considered with respect to the operator-norm topology, although it is strongly continuous with respect to the strong operator topology on $\bdd$.

In the case $p=2$, we see that $\cS^2(\cA^2_\nu)$ is the space of Hilbert-Schmidt operators on $\cA^2_\nu$.  Furthermore, the representation $\ActopL{x}$ on $\cS^2(\cA^2_\nu)$ is unitary, and coincides with the tensor-product representation $\pi_\nu\otimes\overline{\pi_\nu}$.

In quantum harmonic analysis, one studies spaces of operators by treating them as noncommutative versions of spaces of functions. In particular, one says that the operator $\ActopL{x}(x)S$ is the \textbf{left translation of the operator $S$ by $x\in G$}. 

Continuing the analogy with function spaces, we now we define the convolution of a function and an operator using the integrated representation of $\ActopL{x}$. In particular, given a function $\psi:G\to \C$ and and an operator $S\in \bdd$, we define $\psi\ast S$ formally by the weak integral
\[
   \psi\ast S:= \ActopL{x}(\psi)S =\int_G\psi(x) \actopL{x}{S} \haar{x}.
\]

The following QHA-Young's inequalities hold: for $1\leq p,q,r\leq\infty$ such that $\frac{1}{p}+\frac{1}{q}=1+\frac{1}{r}$,
\begin{equation}\label{eq:QHAYoung}
    \|\psi\ast S\|_r\leq \Big(\frac{1}{c_\nu}\Big)^{1-q/r}\|\psi\|_p\|S\|_q.
\end{equation}
where the constant $c_\nu$ is as in \eqref{eq:cnu}.

\begin{remark}
    {\rm In \cite{H23} and \cite{DDMO24}, the inequality \eqref{eq:QHAYoung} wasn't stated in full generality. However, this follows by Riesz-Thorin interpolation, and Calder\'on's interpolation theorem for Schaten classes (see, \cite{BerghLofstrom76}). To achieve this, one uses the end point inequalities:
    $$\|\psi \ast S\|_q\leq \|\psi\|_1\|S\|_q,$$
    and
    $$\|\psi\ast S\|\leq \frac{1}{c_\nu} \|\psi\|_q\|S\|_{q'},$$
    where $q'$ satiesfies $\frac{1}{q}+\frac{1}{q'}$.}
\end{remark}

One also defines the convolution of two operators. Recall that the duality between $\schop$ and $\schopq{p'}$ is given by the trace, which will now play the role of the integral. Given two operators $S\in \bdd$ and $A\in \traceop$, the convolution $S\ast A:G \to \C$ is the function  given by
\[
   (S\ast A)(x):=\tr(S(\actopL{x}{A}))=\tr(S\pi_\nu(x)A\pi_\nu(x)^{-1})
\]
for all $x\in G$. The above convolution is not commutative in general. Rather, one has that $(S\ast A)(x)=(A\ast S)(x^{-1})$ for all $x\in G$.
The convolution of two operators also satisfy a QHA Young's inequality: for $1\leq p,q,r\leq\infty$ such that $\frac{1}{p}+\frac{1}{q}=1+\frac{1}{r}$,
\begin{equation}\label{eq:QHAYoung ops}
    \|S\ast A\|_r\leq \Big(\frac{1}{c_\nu}\Big )^{\frac{1}{r}} \|S\|_p\|A\|_q.
\end{equation}

\subsection{Duality}
Let $1\leq p<\infty$, and let $p'\in (0,\infty]$ such that $\frac{1}{p}+\frac{1}{p'}=1$.
The Banach space dual to $\LpG$ is $L^{p'}(G)$, with the duality pairing given by
$$\ip{f}{g}_{L^1}:=\int_G f(x)g(x)\haar{x},\quad f\in\LpG, g\in L^{p'}(G).$$
Similarly, the dual of Schatten-$p$-class of operators $\schop$ is the space $\cS^{p'}(\cH)$, with duality pairing given by
$$\ip{S}{T}_{\tr}:=\tr(ST), \quad S\in \schop, T\in\schopq{p'}.$$

Let $1\leq p<\infty$, and let $p'\in (0,\infty]$ such that $\frac{1}{p}+\frac{1}{p'}=1$. Then for $\psi\in \LoneG$, $a\in \bddfG$, $S\in \cS^p(\cH)$, and $B\in \cS^{p'}(\cH)$, we have
\begin{equation}\label{eq:duality}
    \ip{B\ast S}{\psi}_{L^1}=\ip{B}{\psi\ast S}_{\tr}\ \ \forall B\in \cS^{p^\prime}(\cH).
\end{equation}

Furthermore, for $S\in \traceop$,
\begin{equation}
    1\ast S= \frac{1}{c_\nu}\tr(S) \id.
\end{equation}
Also, for $S,T\in \traceop$, we have
\begin{equation}\label{eq:int of convol}
    \int_G (S\ast T)(x)\haar{x}=\frac{1}{c_\nu}\tr(S)\tr(T).
\end{equation}

The translations of convolutions behave as follows:
\begin{lemma}\label{lem:translation_convol}
    Let $\psi:G\to \C$ and $S,T\in \bdd$ be s. t. the following convolutions are defined. Then for $x\in G$,
    \begin{enumerate}
        \item $\actopL{x}{(\psi\ast S)}=(\actfL{x}{\psi})\ast S$
        \item $\actfL{x}{(S\ast T)}=(\actopL{x}{S})\ast T$
    \end{enumerate}
\end{lemma}

Moreover, the QHA convolutions are associative: that is, for all $\psi,\psi_1,\psi_2:G\to \C$ and $S,T\in \bdd$ such that the following convolutions are defined, we have:
\begin{enumerate}
    \item $\psi_1\ast (\psi_2\ast S)=(\psi_1\ast \psi_2)\ast S$
    \item $\psi\ast (S\ast T)=(\psi\ast S)\ast T$
\end{enumerate}
whenever the integrals and the trace defining these convolutions exist.

As a consequence of \cite[Theorem A.16]{DDMO24}, we also have the following convergence properties:
\begin{theorem}\label{theo:convergence}
    Let $\{S_k\}_{k\in\N}$ be a sequence of bounded operators converging to $S$ in the strong operator topology (SOT). Then the following holds for $\psi \in\LoneG$ and $A,B\in \traceop$:
    \begin{enumerate}
        \item $\psi\ast S_k \stackrel{k\rightarrow \infty}{\longrightarrow} \psi\ast S$.
        \item $S_k\ast A \stackrel{k\rightarrow \infty}{\longrightarrow}  S\ast A$ pointwise.  
        \item 
            \(
                (S_k\ast A)\ast B\stackrel{k\rightarrow \infty}{\longrightarrow}  (S\ast A)\ast B
            \)
            in the weak operator topology (WOT). 
    \end{enumerate}
\end{theorem}

\subsection{Toeplitz operators}
Now we observe that Toeplitz operators and Berezin transforms can be written as QHA convolutions. Let $\Phi^\nu$ be the rank one operator $\Phi^\nu=1\otimes 1$. 

The following Lemma is \cite[Lemma 3.1]{D25}. We include its short proof for completeness.
\begin{lemma}\label{lem:Toeplitz}
   Let $1\leq p\leq \infty$ and let $a\in \Lp$. Then the Toeplitz operator $T^\nu_a$ can be written as
   \[
      T^\nu_a=c_\nu (a\ast \Phi^\nu)= a\ast (c_\nu \Phi^\nu).
   \]
\end{lemma}

\begin{proof}
   First note that 
   \[
      \actopL{x}{\Phi^\nu}=\pinu(x)1\otimes \overline{\pinu(x)1}= h(x,x)^\nu (K_{x\cdot 0}^\nu\otimes \overline{K_{x\cdot 0}^\nu)}.
   \]
   Note that we consider $a$ as a function on $G$ as usual, by writing $a(x)=a(x\cdot 0)$, for $x\in G$. Then for $f_1,f_2\in \Berg$,
   \begin{align*}
      \ipnu{(a\ast \Phi^\nu)f_1}{f_2} & = \int_G a(x\cdot 0) \ipnu{\actopL{x}{\Phi^\nu}f_1}{f_2} \haar{x}\\
        & = \int_G h(x,x)^\nu a(x\cdot 0) \ipnu{(K_{x\cdot 0}^\nu\otimes \overline{K_{x\cdot 0}^\nu})f_1}{f_2} \haar{x}\\
        & = \frac{1}{c_\nu}\int_{\bsd} a(z) \ipnu{(K_z^\nu\otimes \overline{K_z^\nu})f_1}{f_2} d\mu_\nu(z)\ \ \text{(by \ref{eq:integral_nu})}\\
        & =\frac{1}{c_\nu}\int_{\bsd} a(z) f_1(z)\overline{f_2(z)} d\mu_\nu(z)\\
        & =\frac{1}{c_\nu} \ipnu{af_1}{f_2}\\
        & =\frac{1}{c_\nu} \ipnu{T^\nu_af_1}{f_2}.\qedhere
   \end{align*}
\end{proof}

Now we remark that the Berezin transform of an operator can be written as a QHA convolution as well.
\begin{lemma}
   Let $S\in \bdd$. Then 
   \(
      B^\nu(S)=S\ast \Phi^\nu.
   \)
\end{lemma}

\begin{proof}
   Note that
   \begin{align*}
      (S \ast \Phi^\nu)(x)&=\Tr(S\actopL{x}{\Phi^\nu})\\
         & =\Tr(S (k^\nu_{x\cdot0}\otimes \overline{k^\nu_{x\cdot 0}}))\\
         & = \ip{Sk^\nu_{x\cdot0}}{ k^\nu_{x\cdot 0}}\\
         & = B^\nu(S)(x\cdot 0).\qedhere
   \end{align*}
\end{proof}

We also compute $\Phi^\nu \ast \Phi^\nu$ in terms of the function $\varphi_\nu$ \eqref{eq:phinu}.
\begin{lemma}\label{lem:convolution of Phinu}
   We have that
   \(
      \Phi^\nu \ast \Phi^\nu=\frac{1}{c_\nu}\varphi_\nu.
   \)
\end{lemma}

\begin{proof}
   Note that
   \begin{align*}
      (\Phi^\nu \ast \Phi^\nu)(z)&= \ipnu{\Phi^\nu k^\nu_{x\cdot0}}{ k^\nu_{x\cdot 0}}\\
         & =\ipnu{k^\nu_{x\cdot 0}}{1}\ipnu{1}{k^\nu_{x\cdot 0}}\\
         & =h(z,z)^\nu\ipnu{K^\nu_{x\cdot 0}}{1}\ipnu{1}{K^\nu_{x\cdot 0}}\\
         & =h(z,z)^\nu. \qedhere
   \end{align*}
\end{proof}

\subsection{Radial operators and their convolutions}

Recall that an operator $S\in \bdd$ is said to be $K$-invariant (radial) if $\actopL{k}{S}=S$ for all $k\in K$. Utilizing the fact that $\bsd$ is a commutative space, we prove additional commutativity properties satisfied by the QHA convolutions of $ K$-invariant operators. As Toeplitz operators are QHA convolutions with the rank one radial operator $\Phi^\nu$, these results play a central role in the theory of Toeplitz operators on $\Berg$. The results we proved here were proved for the unweighted Bergman space of the unit ball $\B^n$ in $\C^n$ in \cite{DDMO24}.

Note that for $k\in K$ and $f\in \Berg$, we have $\pinu(k)f=k^\nu_{k\cdot 0}\actfL{k}{f}=\actfL{k}{f}$ \eqref{eq:translation}. Therefore, by Theorem \ref{theo:FK radial decomposition} and the Schur's lemma, a radial operator $S\in \bdd$ is diagonalizable and acts by constant multiples of the identity on the polynomial spaces $\cP_{\textbf{m}}$.

Recall that the radialization (the $K$-average) of an operator $S\in \bdd$ is given by
\[
    \Rad{S}=\int_{K} \actopL{k}{S} \haarK{k} 
\]
where $d\mu_K$ is the normalized Haar measure on $K$. Then $\|\Rad{S}\|\leq \|S\|$, and moreover, an operator $S\in \bdd$ is radial if and only if $\Rad{S}=S$. In particular, for the projections $P_\textbf{m}:\Berg\to \cP_\textbf{m}$, we have that $\Rad{P_{\textbf{m}}}=P_\textbf{m}$.

\begin{lemma}\label{lem:convolution_rad}
    Let $\psi:G\to \C$ be a function and let $S\in \bdd$ such that the convolution $\psi\Last S$ is defined. Then 
    \begin{enumerate}
        \item If $\psi$ is right $K$-invariant (i.e. $\psi$ is a function on $\bsd$), then $\psi \Last S=\psi \Last \Rad{S}.$
        \item If $\psi$ is left $K$-invariant, then $\psi\ast S$ is $K$-invariant.
    \end{enumerate}
\end{lemma}

\begin{proof}
    Note that for $k\in K$, $\psi\ast \actopL{k}{S}=(\actfR{k^{-1}}{\psi})\ast S= \psi\ast S$. Then (1) follows by averaging over $K$. See \cite[Lemma 3.2]{DDMO24} for a complete similar proof.

    Similarly, $\actopL{k}{\psi\ast S}= (\actfL{k}{\psi})\ast S=\psi\ast S$ for $k\in K$. Hence $\psi\ast S$ is $K$-invariant.
\end{proof}

The proof of the following proposition is similar to the proof of \cite[Lemma 3.11]{DDMO24}. However, we include it due to its importance.

\begin{theorem}\label{theo:radial_convol}
   Let $S\in \bdd$ and let $A,B\in \traceop^K$. Then
   \begin{enumerate}
      \item $S\ast A$ is a function on $\bsd$.
      \item If $S$ is $K$-invariant $S\ast A=A\ast S$.
      \item  $(S\ast A)\ast B=(S\ast B)\ast A$.
      \item If $S$ is also $K$-invariant, $(S\ast A)\ast B=(A\ast B)\ast S$.
   \end{enumerate}
\end{theorem}

\begin{proof}
    (1) follows from $\actfR{k}{(S\ast A)}=S\ast \actopL{k}{A}=S\ast A$ for $k\in K$. And (2) follows from the fact that $\actfL{k}{(S\ast A)}=(\actopL{k}{S})\ast A=S\ast A$.

    To prove (3), first we note that
    \begin{align*}
        (\Phi^\nu\ast B)\ast (A\ast \Phi^\nu) = (\Phi^\nu\ast A)\ast (B\ast \Phi^\nu)
    \end{align*}
    by (1) and (2), along with the commutativity of convolutions of radial functions. Then by the injectivity of the Berezin transform, we obtain 
    \begin{equation}\label{eq:radial assoc}
        (\Phi^\nu\ast B)\ast A= (\Phi^\nu\ast A)\ast B.
    \end{equation}
    Now by \eqref{eq:duality}, for any $\psi\in L^1(G)$, 
    \begin{align*}
        \langle (S\ast A)\ast (B\ast \Phi^\nu),\psi\rangle_{L^1}&= \langle(S\ast A)\ast B, \psi\ast \Phi^\nu\rangle_{\operatorname{Tr}}\\
        &= \langle S\ast A, \psi\ast (\Phi^\nu\ast B)\rangle_{L^1}\\
        &= \langle S,\psi\ast (\Phi^\nu\ast B)\ast A\rangle_{\operatorname{Tr}}\\
        &= \langle S, \psi\ast (\Phi^\nu\ast A)\ast B\rangle_{\operatorname{Tr}}\quad \text{(by \ref{eq:radial assoc})}\\
        &=\langle (S\ast B)\ast (A\ast \Phi^\nu), \psi\rangle_{L^1}.
    \end{align*}
    Then it follows that $(S\ast A)\ast (B\ast \Phi^\nu)=(S\ast B)\ast (A\ast \Phi^\nu)$. Again by the injectivity of the Berezin transform, we have $(S\ast A)\ast B=(S\ast B)\ast A$.

    For the proof of (4), note that  $(S\Last A)\Last B= (A\Last S)\Last B=(A\Last B)\Last S$.
\end{proof}

\subsection{Uniform continuity and QHA}
The notion of uniform continuity for functions can be defined using the group $G$. However, since $G$, is noncommutative, the definition depends on the type of translation chosen. 

The right-translation of a function $f$ on $\bsd$ by $g\in \G$, denoted $\actfR{x}{f}$, is  the function on $\G$ given by
\[
    (\actfR{x}{f})(y)=f(yx\cdot 0), \ \ \forall y\in \G.
\]
The right translation $\actfR{x}{f}$ is generally not a right-$K$-invariant function on $G$ and hence
does not define a function on $\bsd$. 
We also define the right translation of $f$ by $z\in \bsd$ as follows:
\begin{equation}\label{eq:ztranslation}
    \actfR{z}{f}:=\actfR{\tau_{z}}{f},
\end{equation}

where $\tau_z$ is the involution interchanging $0$ and $z$.
Since each $x\in \G$ can be written as $x=\tau_{x\cdot 0}k_x$ for some $k_x\in K$, we have that  \( \actfR{x\cdot 0}{f}=\actfR{x}{f}.\)
\begin{definition}
    Let $f\in L^\infty(G)$. 
    \begin{enumerate}
        \item We say that $f$ is \textbf{left-$G$-uniformly continuous} if the map $G\to \bddfG$, $x\mapsto\actfL{x}{f}$, is continuous with respect to the norm topology on $L^\infty(G)$. 
        \item We say that $f$ is \textbf{right-$G$-uniformly continuous} if the map $G\to \bddfG$, $x\mapsto\actfR{x}{f}$, is continuous with respect to the norm topology on $L^\infty(G)$.
    \end{enumerate}
    The set of all bounded, left-$G$-uniformly continuous functions on $\bsd$ is denoted by $\cbufL$, and the set of all bounded, right-$G$-uniformly continuous functions on $\bsd$ is denoted by $\cbufR$ (recall here that we are identifying functions on $\bsd$ with right-$K$-invariant functions on $G$.
\end{definition}

The uniformly continuous functions $\cbufL$ and $\cbufR$ form $C^*$-subalgebras of $\bddf$. However, our primary focus is $\cbufR$. The following lemma is easy to prove:
Moreover, the translations of convolutions behave as follows:
\begin{lemma}\label{lem:rtranslation_convol}
    Let  $S,T\in\bdd$ s.t. $S\ast T$ is defined. Then for $x\in G$,
    $$\actfR{x}{(S\ast T)}=S\ast \actopL{x}{T}.$$
\end{lemma}
The proof is a direct calculation which we omit.

\begin{lemma}\label{lem:convolution_in_cbuR}
        Let $S\in \bdd$ and let $A\in \traceop$ be a radial operator. Then $S\Last A\in \cbufR$.
    \end{lemma}
    \begin{proof}
    By Theorem \ref{theo:radial_convol}, we have that $S\ast A$ is a function on $\bsd$. Additionally, $\actfR{x}{(S\ast A)}=S\ast \actopL{x}{A}$, and the map $x\mapsto \actopL{x}{A}$ is continuous from $G$ to $\traceop$. Then the result is true by QHA Young's inequality.
    \end{proof}

Recall the notion of Dirac sequences for $L^1(G/K)\simeq \Lone$.

\begin{definition}
    A sequence $\{\psi_\alpha\}_{\alpha\in \N}\subset \Lone$ is called a $G/K$-Dirac sequence if it satisfies the following:
    \begin{enumerate}
        \item $\psi_\alpha\geq 0$,
        \item $\|\psi_\alpha\|_1=1$,
        \item For each neighborhood $U$ of $0$, we have
        $\int_{\bsd\setminus U} \psi_\alpha(z) \inv{z}\to 0,$
        \item $\psi_\alpha$ is radial.
    \end{enumerate}
\end{definition}

Any $G/K$-Dirac sequence $\{\psi_\alpha\}_{\alpha\in \N}\subset \Lone$ forms a right approximate identity for $\Lone$. Moreover, we have the following :
\begin{proposition}\label{prop:cbuRf}
    Let $\{\psi_\alpha\}\subset \Lone$ be a $G/K$-Dirac sequence. Then the space $\cbufR$ has the following equivalent characterizations.
    \begin{align*}
    \cbufR&=\{a\in \bddf\mid g\mapsto \actfR{g}{a} \text{ is continuous from $\G$ to $\bddfG$} \}\\
    &=\{a\in \bddf\mid z\mapsto \actfR{z}{a}\   \text{ is continuous from $\bsd$ to $\bddfG$} \}\\
    &= \{a\in \bddf\mid a\ast \psi_\alpha\stackrel{\alpha\rightarrow\infty}{\longrightarrow} a \text{ in $\bddf$ with respect to  $\|\cdot\|_\infty$} \}\\
    &=\overline{\bddf\last \Lone}^{L^\infty}\\
    &=\cbufR\last \Lone.
\end{align*}
where $\tau_z$ is as in \eqref{eq:ztranslation}.
\end{proposition}
\begin{proof}
    This only uses standard ideas. For a similar proof, we point to \cite[Proposition 4.7]{DDMO24}.
\end{proof}

Now we define uniform continuity of operators.
\begin{definition}
    An operator $S\in \bdd$ is said to be left uniformly continuous if the translation map $x\mapsto \actopL{x}{S}$ is a continuous map from $G$ to $\bdd$, with respect to the norm topology. 
\end{definition}
The set of all left uniformly continuous operators is denoted by $\cbuopL$, which forms a $C^*$-subalgebra of $\bdd$. For details on what follows, we point to \cite[Section 4]{DDMO24}.

\begin{lemma} The following statements are true:
   \begin{enumerate}
        \item All compact operators are in $\cbuopL$.
        \item If $\psi\in \LoneG$ and $S\in \bdd$, then $\psi\Last S\in \cbuopL$.
        \item If $\psi\in \cbufL$ and $S\in \traceop$, then $\psi\Last S\in \cbuopL$.
    \end{enumerate}
\end{lemma}

\begin{proposition}\label{prop:cbu_characterizations}
    Let $\psi\in \LoneG$ be a cyclic vector for $\LoneG$ and let $\{\psi_\alpha\}_{\alpha\in\N}$ be a Dirac sequence in $\LoneG$. Then we have the following characterizations of $\cbuopL$:
    \begin{align*}
        \cbuopL&= \overline{\LoneG \Last \bdd}^{\cB} \\
        &= \{S\in \bdd \mid \psi_\alpha \Last S \to S \text{ in norm}\}\\
        &=\overline{\psi\Last \bdd}^{\cB}\\
        &=\LoneG \Last \cbuopL.
    \end{align*}
\end{proposition}

In addition, as the translation map $x\mapsto \lambda^\nu_x$ is strongly continuous on $\schop$ for $1\leq p<\infty$, we have the following:
\begin{lemma}\label{lem:schop_convergence}
    Let $1\leq p<\infty$, and let $\{\psi_\alpha\}$ is a Dirac sequence in $L^1(G)$. Then for $S\in \schop$ we have that
    $\psi_\alpha\ast S\to S$
    in $\schop$.
\end{lemma}

\subsection{Wiener's Tauberian theorem}
The classical Wiener's Tauberian theorem discusses when the translates of a function span a dense subset of $L^p$. The QHA analogue of the same addresses when the translates of a trace-class operator span a dense subset of the Schatten-$p$ class.

\begin{definition}
    A function $\psi\in \Lp$ (respectively $\psi\in \LpG$) is said to be $p$-cyclic, also referred to as $p$-regular, if the left translations $\actfL{x}{\psi}$, $x\in G$, span a dense subset of $\Lp$ ($\LpG$). Similarly, $S\in \schop$ is said to be $p$-cyclic is translates of $S$ span a dense subset of $\schop$.
\end{definition}

 Since Werner's original work \cite{W84}, several versions of the Wiener's Tauberian theorem have been developed: for $G=\mathbb{R}^n$, 
 see \cite{FLW26,LS18,W84}: for the abelian phase space, see \cite{FG23}; for general 
 locally compact groups, see \cite{H23}. Moreover, the version discussed in \cite{DDMO24} for the unit ball 
 generalizes naturally to bounded symmetric domains $\bsd=G/K$.

\begin{theorem}[Wiener's Tauberian theorem for $\Lp$]\label{theo:WienerT for GK}
    Let $1\leq p<\infty$, choose $p^\prime$  so that $\frac{1}{p}+\frac{1}{p^\prime}=1$ and let $\Psi\in \schop$ be a radial operator. Then the following are equivalent:
\begin{enumerate}
        \item $\Psi$ is $p$-cyclic 
        \item $\schop=\overline{\Lone \Last \Psi}^{\cS^p}$
        \item $S\mapsto S\Last \Psi$ is injective from $\schopq{p^\prime}  \to \bddf$.
        \item $\psi \mapsto \psi\Last \Psi$ is injective from $L^{p^\prime}(\G) \to \bdd$.
        \item $\Lp =\overline{\cS^1(\cA^2)\ast \Psi}^{L^p}$.
    \end{enumerate}
\end{theorem}

\begin{remark}
    In \cite[Theorem 5.2]{DDMO24}, statement (5) contains a typo: $\Lq$ was written where it should be $\Lp$.
\end{remark}

\begin{theorem} 
    Let the notation be as above, let $1\leq p < \infty$, and assume that $\Psi\in \traceop$ is a radial operator. Then the following are equivalent:
    \begin{enumerate}
        \item $\Psi$ is $p$-cyclic
        \item $\schop=\overline{\Lp \Last \Psi}^{\cS^p}$
        \item $S\mapsto S\Last \Psi$ is injective from $\schopq{p'}  \to \Lq$.
        \item $a\mapsto a\Last \Psi$ is injective from $L^{p'}(G) \to \schopq{p'}$
        \item $\Lp=\overline{\schop\Last \Psi}^{L^p}$.
    \end{enumerate}
\end{theorem}


\section{The \texorpdfstring{$\alpha$}{alpha}-Berezin transform} 
\label{sec:alpha_berezin_section}
In \cite{D17}, Dieudonne discusses the $\alpha$-Berezin transform of operators for Bergman spaces over bounded symmetric domains. This notion was first introduced by Su\'arez for the unit disc in \cite{S04,S05}, and then generalized to the unit ball in \cite{BHV14b,BHV14}. Here we study this notion systematically using the QHA framework, and obtain some new results.

\subsection{An approximate identity for \texorpdfstring{$\Lone$}{L1}}
Recall the function $\varphi_\alpha$ \eqref{eq:phinu} defined for $\alpha>\gen-1$. Then $\{\varphi_\alpha\}$ is a $G/K$-Dirac sequence, and therefore, a right-approximate identity for $\Lone$.

\begin{lemma}
   The family of functions $\{\varphi_\alpha\}_{\alpha>\gen-1}$ is a $G/K$-Dirac sequence.
\end{lemma}

\begin{proof}
   The function $\varphi_\alpha$ is radial, positive and has norm 1.
   
   Let $0<\delta<1$, and denote by $V_\delta$ the set:
   $$V_\delta=\{w\in\bsd \mid t_j(w)^2<\delta \text{ for all } 1\leq j \leq r\}$$
   where $t_j$ are as in \eqref{eq:jframe}.
   Then the collection $\{V_\delta\}_{\delta>0}$ is a neighborhood base of the origin.

   Now, for each $w\in \bsd\backslash V_\delta$, there exists $1\leq k \leq r$ such that $t_k(w)^2>\delta$ and hence $(1-t_k(w)^2)\leq (1-\delta)$. Also, $(1-t_j(w)^2) \leq 1$ for all $1\leq j\leq r$.  Recalling (\ref{eq:formula_for_Jordan_triple_system}), if follows that
   \[
      h(w,w)=  \prod_{j=1}^r (1-t_j(w)^2)
      \leq  (1-\delta)
   \]
   We therefore have that for $\alpha>\gen$,
   \begin{align*}
       \int_{\bsd\setminus V_\delta} \varphi_\alpha(w) \inv{w}\leq c_\alpha \int_{\bsd\setminus V_\delta} h(w,w)^{\alpha-\gen} dz
       \leq c_\alpha(1-\delta)^{\alpha-\gen}.
   \end{align*}
   Note that $c_\alpha (1-\delta)^{\alpha-\gen} \to 0$ as $\alpha\to \infty$, since $0<1-\delta<1$ and $c_\alpha$ is a polynomial in $\alpha$. This completes the proof.
\end{proof}

\subsection{The trace class operator $\Phi^\nu_\alpha$}
In order to define the $\alpha$-Berezin transform, we need to first define certain trace-class operators (in fact, it will turn out that they will have trace $1$).  Recall that by Lemma \ref{lem:convolution of Phinu}, we have that 
$B^\nu(\Phi^\nu)=\Phi^\nu\ast \Phi^\nu=\frac{1}{c_\nu}\varphi_\nu$. Extending this property, for 
each $\nu > \gen-1$ and $\alpha\geq 0$, we define an
operator $\Phi^\nu_\alpha$ whose Berezin transform is $\frac{1}{c_\nu}\varphi_{\nu+\alpha}$.
We will now calculate this operator and show that it is in fact of trace class (that is, $\Phi^\nu_\alpha \in \cS^1(\cA^2_\nu)$). Furthermore, if $\alpha\in\N_0$, we will see that it is in fact a finite-rank operator.

First we state the following well known lemma that will be used several
time in the following proof:
 
\begin{lemma}\label{lem:Gamma} Let $x,y\ge 0$, $x\not\in \N_0$ and let $c>0$.  Then there exists a constant $C>0$ such that 
\begin{equation}\label{eq:Gamma} \left| \frac{\Gamma (t -x)}{\Gamma (t+y)}\right| \le Ct^{-x-y},\quad \text{for all }\quad t\ge c>x .
\end{equation}
\end{lemma}
\begin{proof} By \cite[Ex.\ 76, p.\ 7]{Lebedev72}, it is known that 
\[
   \frac{\Gamma (t -x)}{\Gamma (t+y)} = t^{-x-y} \left[1 - \frac{(x+y)(y-x-1)}{2t} + O(t^{-2})\right]
\]
as $t\rightarrow \infty$.  Thus, there exists $R>0$ and a positive constant $C_1$ such the absolute value of the right hand side is bounded by $C_1t^{-x-y}$ for all $t>R$.  
Furthermore, the function
\[ [c,R] \ni t\mapsto t^{x+y}\left|\frac{\Gamma (t -x)}{\Gamma (t+y)}\right| \]
is bounded by some constant $C_2$ as $[c,R]$ is compact and the function is continuous. The claim now follows by taking $C=\max\{C_1,C_2\}$.
\end{proof}

\begin{theorem}\label{theo:Phi alpha}
   Let $\nu > \gen-1$ and $\alpha>0$.  For each $\textbf{m}\geq 0$, let $P^\nu_\textbf{m}$ denote the orthogonal projection of $\cA^2_\nu$ onto the polynomial space $\cP_\textbf{m}$.  Then the radial operator $\Phi_\alpha^\nu$ defined by the spectral decomposition:
   \[
       \Phi_\alpha^\nu = \frac{c_{\nu+\alpha}}{c_\nu}\sum_{\textbf{m}\geq 0} \frac{\Gamma_\bsd(-\alpha + \textbf{m})\Gamma_{\bsd}(\nu)}{\Gamma_{\bsd}(-\alpha)\Gamma_{\bsd}(\nu+\textbf{m})} P^\nu_{\textbf{m}} 
   \]
   is trace class (that is $\Phi_\alpha^\nu\in \cS^1(\cA^2_\nu)$) and
   \[
      B^\nu \Phi_\alpha^\nu = \frac{1}{c_\nu}\varphi_{\nu+\alpha} .
   \]
   Furthermore, 
   \[
      \Tr \Phi_\alpha^\nu = 1.
   \]
\end{theorem}
 
\begin{proof}
    Note that
    \begin{align*}
        B^\nu(\Phi^\nu_\alpha)&= \Phi^\nu_\alpha\ast \Phi^\nu\\
        &=  \frac{c_{\nu+\alpha}}{c_\nu} \sum_{\textbf{m}\geq 0} \frac{\Gamma_\bsd(-\alpha + \textbf{m})\Gamma_{\bsd}(\nu)}{\Gamma_{\bsd}(-\alpha)\Gamma_{\bsd}(\nu+\textbf{m})} (P^\nu_{\textbf{m}} \ast \Phi^\nu),
    \end{align*}
    where the convergence is considered pointwise (see Theorem \ref{theo:convergence}).
    Furthermore, we know from \cite[Theorem 3.8]{FK90} that 
   \[
      h(z,w)^\alpha = \sum_{\textbf{m} \geq 0 } \frac{\Gamma_{\bsd}(-\alpha+\textbf{m})}{\Gamma_\bsd(-\alpha)}K_{\textbf{m}}(z,w),
   \]
   where $K_{\textbf{m}}$ is the reproducing kernel for the finite-dimensional space $P_{\textbf{m}}$ with respect to the Fischer inner product $\ip{\cdot}{\cdot}_F$ on $\cP_{\textbf{m}}$.  
   Since $K_z^\nu=h(\cdot,z)^{-\nu}$, we have that $P^\nu_{\textbf{m}}K^\nu_z=\frac{\Gamma_{\bsd}(\nu+\textbf{m})}{\Gamma_\bsd(\nu)}K_{\textbf{m}}(\cdot,z)$, and hence,
   \begin{align*}
       (P^\nu_{\textbf{m}} \ast \Phi^\nu)(z)&= \ipnu{P^\nu_{\textbf{m}}k^\nu_z}{k^\nu_z}\\
       &=h(z,z)^\nu \ipnu{P^\nu_{\textbf{m}}K^\nu_z}{K^\nu_z}\\
       &= \frac{\Gamma_{\bsd}(\nu+\textbf{m})}{\Gamma_\bsd(\nu)} h(z,z)^\nu K_{\textbf{m}}(z,z).
   \end{align*}
   Therefore,
   \begin{align*}
       B^\nu(\Phi^\nu_\alpha)(z)&=   \frac{c_{\nu+\alpha}}{c_\nu} h(z,z)^\nu \sum_{\textbf{m}\geq 0} \frac{\Gamma_{\bsd}(-\alpha+\textbf{m})}{\Gamma_\bsd(-\alpha)} K_{\textbf{m}}(z,z)\\
       &=  \frac{c_{\nu+\alpha}}{c_\nu} h(z,z)^\nu  h(z,z)^\alpha\\
       &=  \frac{1}{c_\nu}\varphi_{\nu+\alpha}(z)
   \end{align*}
   as required. 

    If $\Phi^\nu_\alpha$ is indeed trace-class, we have that
   $$\frac{1}{c_\nu}\tr(\Phi^\nu_\alpha)\tr(\Phi^\nu)=\|\Phi^\nu_\alpha\ast \Phi^\nu\|_1=\frac{1}{c_\nu}\|\varphi_{\nu+\alpha}\|_1=\frac{1}{c_\nu}.$$
   Since $\tr(\Phi^\nu)=1$, we obtain that $\tr(\Phi^\nu_\alpha)=1$, in that case.
   
    Finally, we show that $\Pna$ is of trace class in several steps, dealing with each of the parts $\dim \cP_m$ and
  the quotient of $\Gamma$-factors separately.
  First, if $\alpha +(j-1)a/2 $ is a positive integer for all $j$, then the sum defining $\Pna$ is finite and hence $\Pna$ is of trace class
  in that case. 

  We will now assume for the moment that
  $\alpha+ (j-1)a/2$ is not an integer for all $1\le j\le r$. We can then write
   \[\Pna =  d (r,\nu,\alpha ) \sum_{\bm \ge 0} \frac{\GO (\bm - \alpha)}{\GO (\bm + \nu )} P_\bm .\]
  
  We will estimate $| \GO (\bm - \alpha)/\GO (\bm + \nu ) |$ 
and $\dim \cP_m$ to show that the sum converges absolutely. We do si by estimating each of the two parts separately, starting with the quotient of the $\Gamma$-functions. In the following, $C$ is a positive constant that may change as we move from line to line.
  \medskip
  
 \noindent
 {\bf Step 1:}
 Assume that $m_1\ge m_2\ge \cdots \ge m_1\ge 1$, so that we can use
 \eqref{eq:Gamma} to get 
\[
  \frac{\GO (\bm -\alpha)}{\GO (\bm +\nu)} =C\prod_{j=1}^r \frac{\Gamma (m_j - \alpha - (j-1)a/2)}{ \Gamma (m_j + \nu -(j-1)a/2)} = C  \prod_{j=1}^r m_j^{-\nu - \alpha}.\]
 \medskip 
 
 \noindent
 {\bf Step 2:} If $m_{k}\ge 1$ and $m_{k+1}=m_{k+2} =\ldots = m_r=0$ for some $1\le k \le r$ then 
   \[\prod_{i=k+1}^r \left|\frac{\Gamma (m_j - \alpha - (j-1)a/2)}{ \Gamma (m_j + \nu -(j-1)/2)}\right|=
   \prod_{i=k+1}^r \left|\frac{\Gamma (- \alpha - (j-1)a/2)}{ \Gamma (\nu -(j-1)a/2)}
   \right|   \le C\quad\text{is bounded}\]
   and the product for the other factors can be treated like above reducing it to the product
   \[\prod_{j=1}^{k} \left|\frac{\Gamma (m_j - \alpha - (j-1)a/2)}{ \Gamma (m_j + \nu -(j-1)/2)}\right| = C \prod_{j=1}^k x^{-\nu -\alpha} \]
as in the case $k=r$ above.
\medskip 
 
 \noindent
 {\bf Step 3:}  If $\alpha+(j-1)a/2\in\N_0$ for some $0\le j \le r$ then one has to treat the cases $m_r$ big enough such that 
 $\alpha+ (j-1)a <m_r$ where we can use the argument from above to estimate the quotient. If $m_j - \alpha - (j-1)a/2$ is
 a negative integer then then the corresponding part in the sum in a finite sum. For the other $\bm$'s  use the above
 estimate to deal with the sum. 
\medskip 
 
 \noindent
 {\bf Step 4:} 
We now estimate the dimension of $\cP_\bm$. The idea is the same as above so we do not discuss
all the possible different cases in so much details. Still working with the assumption that $m_1\ge m_2\ge \cdots \ge m_r\ge 1$ and that
$\alpha + (j-1)a/2\not\in \N_0$  we
get for the factor $\GO (v_r+\bm)/\GO (u_r+\bm)$ in the dimension formula, Theorem \ref{thm:dimPm} we have
   as above:
\[
\frac{ \GO (v_r+\bm)}{\GO (u_r+\bm)}=\prod_{j=1}^r \frac{\Gamma \left(m_j + v_r -\frac{a(j-1)}{2}\right)}{\Gamma \left(m_j + u_r -\frac{a(j-1)}{2}\right)}
\le C \prod_{j=1}^r m_j^b .\]
\medskip 
 
 \noindent
 {\bf Step 5:} 
For the last factor  
\begin{equation}\label{eq:Step5}
\prod_{1\le i< j\le r} \left(m_i-m_j +\frac{a(j-i)}{2}\right)
\frac{\Gamma \left(m_i-m_j +\frac{a(j-i+1)}{2}\right)}{\Gamma \left(m_j -m_j +1+ \frac{a(j-i -1)}{2}\right)}
\end{equation}
we assume first that $m_i-m_j> ar/2$. Using that $m_i-m_j \le m_i$ we get for those $\bm$'s using the same
argument as above:
\begin{align*}
\eqref{eq:Step5} &\le C_1
\prod_{1\le i< j\le r}\left(m_i-m_j +\frac{a(j-i)}{2}\right) (m_i-m_j)^{a-1}\\
& \le
C_1\prod_{1\le i< j\le r} (m_i-m_j)^{a}\\
&\le C_2\left( \prod_{j=1}^r m_j\right)^{a(r-1)} .
\end{align*}
 \medskip 
 
 \noindent
 {\bf Step 6:} 
Combining all of those estimates using the assumption that $\alpha + (j-1)/a\not\in \N_0$ we get
\begin{align*}
\sum_{\bm>0} \left| \frac{\GO (\bm - \alpha)}{\GO (\bm + \nu )}\right| \dim \cP_m
&\le
C\sum_{\bm>0 } \prod_{j=1}^r m_j^{-(\nu + \alpha -b - a(r-1))}\\
\le C \left(\sum_{j=1}^r \frac{1}{j^{(\nu + \alpha -b -a(r-1)))}}\right)
\end{align*} 
and the last sum is finite if and only if 
\[\nu + \alpha - b -a(r-1)> 1.\]
But, using the assumption $\alpha >0$ and $\nu > g-1=a(r-1) + b +1$, gives
\[\nu + \alpha - b - a(r-1)  > \alpha  +1 >1\]
which shows that the sum is finite.
\medskip 
 
 \noindent
 {\bf Step 7:} For the remaining cases we apply the same arguments as in step 2 and step 3 to derive an estimate and reduce the sum to
 $1\le k \le r-1$.
\end{proof}

\subsection{The $\alpha$-Berezin transform}
\begin{definition}
   For $\alpha>\gen-1$, and $S\in \bdd$, we define the $\alpha$-Berezin transform of $S$ to be the function $B^\nu_\alpha(S):\bsd\to \C$, by
   \[
      B^\nu_\alpha(S)=S\ast \Phi^\nu_\alpha.
   \]
\end{definition}

Clearly $B^\nu_\alpha(S)\in \bddf$ with $\|B^\nu_\alpha(S)\|_\infty\leq \|S\|\|\Phi^\nu_\alpha\|_1$.
 
By Theorem \ref{theo:Phi alpha}, we obtain that
\begin{equation}
    B^\nu_\alpha(S) =\frac{c_{\nu+\alpha}}{c_\nu}\sum_{\textbf{m}\geq 0} \frac{\Gamma_\bsd(-\alpha + \textbf{m})\Gamma_{\bsd}(\nu)}{\Gamma_{\bsd}(-\alpha)\Gamma_{\bsd}(\nu+\textbf{m})} (S\ast P^\nu_{\textbf{m}}).
\end{equation}

\begin{lemma}\label{lem:properties Berezin}
   Let $\alpha,\beta\in\N_0$, and suppose that $a\in \bddf$ and $S\in \bdd$. Then
   \begin{enumerate}
      \item $B^\nu_\alpha(S)\in \cbufR$.
      \item $B_\alpha^\nu(T^\nu_a) = {B_{\nu+\alpha}(a)}$
      \item $B_{\nu+\alpha}(B^\nu_\beta(S)) = B_{\nu+\beta}(B^\nu_\alpha(S))$.
   \end{enumerate}
\end{lemma}

\begin{proof}
   By Lemma \ref{lem:convolution_in_cbuR}, we obtain that $B^\nu_\alpha(S)\in \cbufR$.

   Note that
   \[
      B^\nu_\alpha(T^\nu_a)=c_\nu (a\ast \Phi^\nu)\ast \Phi^\nu_\alpha=c_\nu a\ast (\Phi^\nu\ast \Phi^\nu_\alpha)=a\ast \varphi_{\nu+\alpha}=B_{\nu+\alpha}(a). 
   \]
   Similarly,
   \begin{align*}
      B_{\nu+\alpha}(B^\nu_\beta(S)) & = B_{\nu+\alpha}(S * \Phi^\nu_\beta)  \\
      & = (S * \Phi^\nu_\beta) * \varphi_{\nu+\alpha} \\
      & = c_\nu(S * \Phi^\nu_\beta) * (\Phi^\nu_\alpha\ast \Phi^\nu) \\
      & = c_\nu((S * \Phi^\nu_\beta) * \Phi^\nu_\alpha)\ast \Phi^\nu) \\
       & = c_\nu((S * \Phi^\nu_\alpha) * \Phi^\nu_\beta)\ast \Phi^\nu) \quad \text{(by radiality of $\Phi^\nu_\alpha$ and Theorem \ref{theo:radial_convol})}\\
      & = B_{\nu+\beta}(B^\nu_\alpha(S)),
   \end{align*}
   where the last line followed by interchanging the roles of $\alpha$ and $\beta$ in the preceding calculation.
\end{proof}

\subsection{Toeplitz Criterion}
Here, we observe that the uniform boundedness of the $\alpha$-Berezin transform of $S\in \bdd$, forces $S$ to be a Toeplitz operator. And as the converse is also true, this provides a classification of Toeplitz operators.

\begin{theorem}\label{theo:Toeplitz criterion}
    Let $1<p\leq \infty$ and let $S\in \schop$. Then the following are equivalent:
    \begin{enumerate}
        \item $S=T^\nu_a$ for some $a\in \Lp$.
        \item There exists $C>0$ such that $\|B^\nu_\alpha(S)\|_p < C$ for all $\alpha\in\N$.
    \end{enumerate}
\end{theorem}

We omit the proof of the above theorem, and point to \cite[Theorem 6.6]{DDMO24} for a similar proof. The core idea of the theorem  relies on the fact that $\Lp$ admits a predual, which explains why the theorem does not hold when $p=1$.

\subsection{Approximating operators by Toeplitz operators}
As we noted in the previous subsection, the $\alpha$-Bereazin transform characterizes Toeplitz operators among general bounded operators on $\cA^2_\nu$. A natural question is whether it can also be used to characterize operators in the Toeplitz algebra  $\toepalg$. Specifically, one would like to determine whether one has that convergence $T^\nu_{B_\alpha^\nu(S)}\to S$ as $\alpha\rightarrow\infty$ for all $S\in \toepalg$. Suárez worked on this conjecture for the Bergman space over the unit disc $\cA^2(\D)$, in a series of papers \cite{S04,S05}. In particular, he proved that the convergence holds for operators in the radial Toeplitz algebra. His results were later extended to the setting of the unit ball in \cite{BHV14,BHV14b}, and further to bounded symmetric domains in \cite{D17}. This question was also studied for the unit ball using QHA techniques in \cite{DDMO24}, thereby broadening the known results. Here, we generalize these findings in \cite{DDMO24} to the case of arbitrary bounded symmetric domains. We also note that in a recent article by Mitkovski and Ouyang \cite{MO26}, a counterexample was obtained in the setting of the Bergman space over the unit disc $\cA^2(\D)$, disproving this conjecture for the full Toeplitz algebra.

We now include two natural QHA results (Lemma \ref{lem:Berezin_convol} and Proposition \ref{prop:convergence approx identity}) that will be key to proving \ref{theo:radial covolution convergence}.

\begin{lemma}\label{lem:Berezin_convol}
   Let $\alpha\in \N_0$. Then
   \begin{enumerate}
      \item $T^\nu_{B_\alpha^\nu(S)}=c_\nu (B^\nu(S)\ast \Phi^\nu_\alpha)$.
      \item If $S$ is radial, then $T^\nu_{B_\alpha^\nu(S)}=\varphi_{\nu+\alpha}\ast S$.
   \end{enumerate}
\end{lemma}

\begin{proof}
   Note that since $\Phi^\nu_\alpha$ and $\Phi^\nu$ are radial operators,
   \[
      T^\nu_{B_\alpha^\nu(S)}=c_\nu(S\ast\Phi^\nu_\alpha)\ast \Phi^\nu=c_\nu(S\ast\Phi^\nu)\ast \Phi^\nu_\alpha=c_\nu (B^\nu(S)\ast \Phi^\nu_\alpha)
   \]
   by Theorem \ref{theo:radial_convol}. The proof of the second statement is similar, and follows by Theorem \ref{theo:radial_convol}.
\end{proof}

The proof of the following theorem is similar to \cite[Proposition 6.9]{DDMO24} stated for the rank one case, and we avoid repeating the same arguments.
\begin{proposition}\label{prop:convergence approx identity}
    Let $\{\psi_\alpha\}_{\alpha\in\N_0}$ be a $G/K$-Dirac sequence, and let $S\in \bdd$ be a radial operator.  Then the following convergences hold:
    \begin{enumerate}
        \item $\psi_\alpha\ast S\to S$ in $\bdd$ with respect to the SOT.
        \item If $S\in \cbuopL$,  then $\psi_\alpha\ast S\to S$ with respect to the operator norm in $\bdd$.
        \item If $S\in \schop$, then $\psi_\alpha\ast S\to S$ with respect to the Schatten-$p$ norm in $\schop$.
    \end{enumerate}
\end{proposition}

Now we turn to the convergence of the sequence $\{T^\nu_{B^\nu_\alpha(S)}\}_{\alpha\in \N_0}$ to $S$.

\begin{theorem}\label{theo:radial covolution convergence}
    Let $1\leq p < \infty$.  Then the convergence $T^\nu_{B^\nu_\alpha(S)}\to S$ holds in each following cases with respect to the given topology:
    \begin{table}[H]
\centering
\begin{tabular}{l|p{10cm}} 
\toprule
Convergence & Sufficient conditions on $S$ \\
\midrule
Schatten $p$-norm &
\begin{enumerate}
    \item $S \in \schop^K$
    \item $S \in \Lone * \schop$
    \item $S \in \Lp * \mathcal{S}^1(\mathcal{A}^2_\nu)$
\end{enumerate} \\
\hline
&\\
operator norm & \begin{enumerate}
    \item $S \in \cbuopL^K$
    \item $S \in \Lone * C^{(L)}_{b,u}(\mathcal{A}^2_\nu)$
    \item $S \in L^q(\bsd, d\mu_0) * \mathcal{S}^1(\mathcal{A}^2_\nu)$ for all $1\leq q\leq \infty$
\end{enumerate} \\
\hline
&\\
SOT & \begin{enumerate}
    \item $S \in \bdd^K$
    \item $S \in \Lone * \bdd$
    \item All of the above cases
\end{enumerate} \\
\bottomrule
\end{tabular}
\end{table}
\end{theorem}

\begin{proof}
    We will consider only the operator norm convergence, as the other cases are similar. First, we assume that $S \in \cbuopL^K$. Then by Lemma \ref{lem:Berezin_convol} and Proposition \ref{prop:convergence approx identity}, we obtain 
    $T^\nu_{B^\nu_\alpha(S)}=\varphi_{\nu+\alpha}\ast S\to S$ in operator norm. 

    To prove (2), note that $S=\psi\ast T$ for some $\psi\in \Lone$ and $T\in \schop$. Then by lemma \ref{lem:convolution_rad}, we get that $S=\psi\ast \Rad{T}$. Also, $T_{B_\alpha(\Rad{T})}\to \Rad{T}$ in operator norm by (1). Then we get
    \begin{align*}
        T^\nu_{B^\nu_\alpha(S)}&= c_\nu((\psi\ast \Rad{T})\ast \Phi^\nu_\alpha)\ast \Phi^\nu\\
        &= c_\nu\psi\ast (\Rad{T}\ast \Phi^\nu_\alpha)\ast \Phi^\nu\\
        &=\psi\ast T^\nu_{B^\nu_\alpha(\Rad{T})}\\
        &\to \psi \ast \Rad{T}=S
    \end{align*}
    in the operator norm by QHA Young's inequality. The proof of (3) is similar.
\end{proof}

\begin{corollary}\label{coro:Toeplitz algebra}
    Let $1\leq p\leq \infty$ and let $a\in \Lp$. Then $T^\nu_{B_\alpha(a)}\to T^\nu_a$ in $\schop$, and the Toeplitz algebra is generated by Toeplitz operators with right-uniformly continuous symbols. 
    
    Moreover,  for $\nu>\gen-1+(r-1)a/2$ 
    $$\toepalg=\overline{\cbufR\ast \Phi^\nu}.$$
\end{corollary}

\begin{proof}
    The convergence is due to $T^\nu_{B_{\nu+\alpha}(a)}=T^\nu_{B^\nu_\alpha(T^\nu_a)}$ and $T^\nu_a=c_\nu(a\ast\Phi^\nu)$, and $B_\alpha(a)=a\ast \varphi_\alpha\in \cbufR$. Also, with the additional assumption on $\nu$, Toeplitz operators are dense in the Toeplitz algebra \cite[Theorem 5.9]{D25} (we caution the reader of the different parametrization on $\nu$ in \cite{D25}). Therefore, the result is true.
\end{proof}

\subsection{A description of the radial Toeplitz algebra}
Now we characterize the radial Toeplitz algebra $\toepalg^K$ as follows. 

\begin{theorem}\label{theo: main theorem}
    We can characterize the radial Toeplitz algebra as follows:
    \begin{align*}
        \toepalg^{K}&= \cbuopL^{K} \\
        &= \overline{\bddf^{K}\ast \traceop}^{\|\cdot\|}\\
        &= \{S\in \bdd^K\mid \varphi_\alpha\ast S \to S \text{ in } \|\cdot\| \}\\
        &= \{S\in \bdd^K\mid T^\nu_{B^\nu_\alpha(S)} \to S \text{ in } \|\cdot\| \}\\
        &= \overline{\{T^\nu_a\mid a\in \cbufR^{K}\}}^{\|\cdot\|}\\
        &= \Lone^{K}\ast \cbuopL^{K}.
    \end{align*}
\end{theorem}

\begin{proof}
    The proof follows by Proposition \ref{prop:cbu_characterizations}, Lemmas \ref{lem:properties Berezin} and \ref{lem:Berezin_convol}, and Theorem \ref{theo:radial covolution convergence}. We omit the proof, which consists mainly of standard inclusion arguments, and point to \cite[Theorem 6.14]{DDMO24} for a similar proof.
\end{proof}


\section{Invariant differential operators on Riemannian symmetric spaces}
\label{sec:invariant_differential_operators}
In the theory of operators on the Bergman spaces associated to the unit ball in $\C^n$, a number of recent papers have made fruitful use of the so-called ``operator Laplacian'', which refers to a technique that allows one to apply the invariant Laplace differential operator of the unit ball to bounded operators on $\cA^2_\nu(\B^n)$ instead of to functions on $\B^n$. 

In this section, we show that this notion can be naturally extended to all invariant differential operators and all bounded symmetric domains. We begin this section by briefly recalling the theory of invariant differential operators on Riemannian symmetric spaces.  For an excellent and thorough introduction to this theory, we recommend \cite{H84}. As an application, in Section \ref{sec:Berezin_shifts}, we construct a differential operator $D_\alpha$ which produces a recursive formula for the $\nu$-Berezin transform of a function.

Let $M$ be a smooth manifold. Because it will be convenient to consider differential operators and vector fields with complex coefficients, we will let $C^\infty(M)$ denote the space of $\C$-valued functions on $M$ (this is different from the usual convention in differential geometry, where only $\R$-valued functions are considered, but consistent with conventions of harmonic analysis).  A \textbf{differential operator on $M$} is a $\C$-linear map $D:C^\infty(M) \rightarrow C^\infty(M)$ that is local, in the sense that $\supp Df \subseteq \supp f$ for all $f\in C^\infty(M)$.  We let $\diffops(M)$ denote the associative algebra (over $\C$) of all differential operators on $M$.  A \textbf{vector field} on $M$ is then any derivation of the algebra $C^\infty(M)$ over $\C$.  We denote by $\cX(M)$ the algebra of all vector fields on $M$.

Suppose now that $G$ is a connected semisimple Lie group with finite center and that $K$ is a maximal compact subgroup.  For any $f:G\rightarrow\C$ and any $x\in G$, we let $\actfL{x}f:G\rightarrow\C$ denote the left-translated function, given by $\actfL{x}f(x) = f(x^{-1}y)$.  Similarly, for each $x\in G$, we let $\actfR{x}f$ denote the right-translated function, given by $\actfR{x}f(y) = f(yx)$, and we define the conjugation action $c(x) = \actfL{x}\actfR{x} = \actfR{x}\actfL{x}$, so that $c(x)f(y) = f(x^{-1}yx)$ for each $f:G\rightarrow\C$ and $y\in G$.    Then $G$ acts on $\diffops(G)$ as follows: for each $D\in \diffops(G)$, we define the (left) translated differential operator $x\cdot_L D$ by $x\cdot_L D = \actfL{x}D\actfL{x}^{-1}$.  We can similarly define the right translation of $D$ by $x\cdot_R D = \actfR{x}D\actfR{x}^{-1}$.  Finally, we define the adjoint action of $x$ on $\diffops(G)$ by $\Ad(x)D = c(x)Dc(x)^{-1}$ for each $x\in G$.

A differential operator is said to be \textbf{(left) invariant} if $g\cdot_L D = D$ for all $g\in G$ (we define right-invariant operators similarly). We denote the algebra of all left-invariant differential operators on $G$ by $\D(G)$.  A differential operator is said to be \textbf{bi-invariant} if it is both left invariant and right invariant.  It is not difficult to see that a differential operator on $G$ is bi-invariant if and only if it is both left-invariant and $\Ad$-invariant, and that the bi-invariant differential operators coincide with the center $Z(\D(G))$ of the algebra of left-invariant differential operators.

As usual, we denote the Lie algebras of $G$ and $K$ by $\gg$ and $\gk$, respectively.  Then the complexification $\gg^C$ and its universal enveloping algebra $\U(\gg^C)$ can be canonically identified with the left-invariant vector fields on $G$ and the left-invariant differential operators on $G$, respectively (remember that our differential operators and vector fields have complex coefficients).  This can be done as follows.  For each $X\in\gg$, we let $X^\# \in \cX(G)$ denote the vector field defined by:
\[
   X^\# f(x) = \ddt \Big|_{t=0}  f(x \exp (tX)).
\]
Then $X\mapsto X^\#$ is an injective Lie-algebra homomorphism over $\R$ from $\gg^\C$ to $\cX(G)$, which extends naturally to an injective Lie-algebra homomorphism over $\C$ from $\gg^\C$ to $\cX(G)$.  This then extends immediately to an associative-algebra isomorphism from $\U(\gg^\C)$ to to the algebra $\D(G)$ of all left-invariant differential operators on $G$.  

We denote by $S(\gg^\C)$ the symmetric algebra generated by $\gg^\C$.  The adjoint action $\Ad$ of $G$ on $\gg$ extends naturally to $S(\gg^\C)$, and there is a unique linear isomorphism:
\[
   \lambda: S(\gg^\C) \rightarrow \D(G) 
\]
that intertwines the adjoint action of $G$ on $S(\gg^\C)$ and $\D(G)$ and sends $X^m$ to $(X^\#)^m$ for each $X\in \gg$. This map is called the \textbf{symmetrization map} and, under the identification of $\D(G)$ with $\U(\gg^\C)$, coincides with the usual symmetrization map from $S(\gg^\C)$ to $\D(G)$ defined algebraically by considering $S(\gg^\C)$ to be the associated graded representation of $G$ corresponding to the filtered representation of $G$ on $\D(G)$.  One can show that, if $d=\dim \gg$ and $X_1,\ldots,X_d$ is a basis for $\gg$, then, for any polynomial $p\in\C[z_1,\ldots, z_d]$, we have that:
\[
   \lambda(p(X_1,\ldots,X_d))f(x) = p\left(\left.\del{t_1},\ldots,\del{t_d}\right)\right|_{t_1=\cdots=t_d=0} 
   f(x \exp(t_1 X_1 + \cdots +t_d X_d))
\]
for all $x\in G$ and $f\in C^\infty(G)$. 

In what follows, we will drop the distinction between $\D(G)$ and $\U(\gg^\C)$. 

Now let us consider the quotient symmetric space $\bsd = G/K$.
We similarly define, for each $f:G/K \rightarrow \C$ and each $g\in G$, the translated function $\actfL{x}{f}:G/K\rightarrow \C$ by $\actfL{x}{f}(y) = f(x^{-1}\cdot y)$. This then allows us to define an action of $G$ on $\diffops(G/K)$ by setting $x\cdot D:= \ell_xD\ell_{x^{-1}}$ for all $x\in G$ and $D\in \diffops(G/K)$, and we once again say that a differential operator $D\in \diffops(G/K)$ is \textbf{G-invariant} if $x\cdot D = D$ for all $x\in G$.   

We let $\kappa : G\mapsto G/K$ denote the canonical projection.  Then the injection
\begin{align*}
   C^\infty(G/K)  \hookrightarrow C^\infty(G),\
         f    \mapsto f\circ \kappa
\end{align*}
allows us identify the smooth functions on $G/K$ with the smooth right-$K$-invariant functions on $G$.  It is thus natural to expect that the algebra $\D(G/K)$ should have something to do with the algebra $\U(\gg)^K$ of differential operators on $G$ that are invariant under $G$ on the left and invariant under $K$ on the right. In fact, one shows that:
\begin{equation}
    \label{eq:themumap}
    \mu: \U(\gg^\C)^K  \rightarrow \D(G/K),\
       D  \mapsto \widetilde{D},
\end{equation}
where $\widetilde{D}$ is defined by $\widetilde{D}f(gK) = D(f\circ \kappa)(g)$ for all $f\in C^\infty(G/K)$ and $g\in G$, is well-defined and surjective, with kernel $\U(\gg^\C)^K \bigcap U(\gg^\C)\gk$.

Let $\theta$ be the Cartan involution fixing $\gk$. Then we have $\gg = \gk \oplus \gs$, where $\gs$ is the $-1$-eigenspace of $\theta$.  Then $\gs$ is a $K$-invariant subspace of $\gg$.  Thus, we see that $S(\gs^\C)$ is a $K$-invariant subspace of $S(\gg^\C)$
\[
    \U(\gg^\C)^K = (\U(\gg^\C)^K \cap \U(\gg^\C)^K \gk) \oplus \lambda(S(\gs^\C)^K),
\]
where $S(\gs^\C)^K$ denotes the $K$-invariant elements of $S(\gs^\C)$.  Thus, 
\[
    \mu\circ\lambda : S(\gs^\C)^K \rightarrow \D(G/K)
\]
is a linear isomorphism.

As usual, let $\ga$ be a maximal abelian subalgebra of $\gs$, and let $\Sigma:=\Sigma(\gg,\ga) \subseteq \ga^*$ denote the system of restricted roots.  We choose a positive subsystem $\Sigma^+\subseteq \Sigma$, so that $\Sigma = \Sigma^+ \bigcup (-\Sigma^+)$.  Defining $\gn = \bigoplus_{\avar\in\Sigma^+} \gg_\avar$, we have the infinitesimal Iwasawa decomposition
\[
   \gg = \gn \oplus \ga\oplus \gk.
\]
This gives rise (see \cite[Lemma II.5.14]{H84}) to the vector-space decomposition:
\begin{equation}
   \label{eq:projection_on_a}
    \U(\gg^\C) = \U(\ga^\C) \oplus (\gn^\C \U(\gg^\C) + \U(\gg^\C)\gk^\C).
\end{equation}
For each left-invariant operator $D\in \D(G) = \U(\gg^\C)$, we define $D_{\ga} \in \U(\ga^\C)$ to be the projection of $D$ onto the first term of (\ref{eq:projection_on_a}). 

Furthermore, this projection has the following interpretation:  Consider the global Iwasawa decomposition $G = NAK$, as well as the space $C^\infty(G)^{N,K}$ of all functions on $G$ that are invariant on the left by $N$ and on the right by $K$. Since $N\times A \times K \rightarrow G, (n,a,k) \mapsto nak$ is a diffeomorphism, we note that $C^\infty(G)^{N,K}$ can be canonically identified with the space $C^\infty(G/K)^N$ of $N$-invariant functions on $G/K$ and the space $C^\infty(A)$ of functions on $A$ using the linear isomorphisms
\begin{align*}
   C^\infty(G)^{N,K}  & \rightarrow C^\infty(G/K)^N \\
           f   & \mapsto f\circ \kappa
\end{align*}
and
\begin{align*}
   C^\infty(G)^{N,K} & \rightarrow C^\infty(A) \\
           f    & \mapsto f|_A.
\end{align*}
Then $D_\ga (f|_A) = (Df)|_A$ for all $f\in C^\infty(G)^{N,K}$ (see \cite[Lemma II.5.14]{H84}).

We will, from this point on, completely identify $C^\infty(G)^{N,K}$ and $C^\infty(G/K)^N$ and use the notation $C^\infty(N\backslash G/K)$ to refer to both.

Since $\ga$ is commutative, the symmetrization map is an algebra isomorphism, and we can canonically write $\U(\ga) = S(\ga)$.  Choosing any basis $H_1,\ldots, H_r$ for $\ga$, where $r = \dim \ga$ is the rank of the symmetric space, we have that for any polynomial $p\in\C[z_1,\ldots,z_r]$,
\[
   p(H_1,\ldots,H_r) f(a) = p\left(\left.\del{t_1},\ldots,\del{t_r}\right)\right|_{t_1 = \cdots = t_r = 0} f(a \exp(t_1H_1 + \cdots t_r H_r))
\]

Recall that the exponential map $\exp: \ga\rightarrow A$ is an isomorphism of Lie groups, with inverse $\log : A\rightarrow \ga$, where $\ga$ is a Lie group under addition.  If we use these isomorphisms to transfer left-invariant differential operators on $A$ onto the Lie algebra $\ga$, then we must obtain the left-invariant differential operators on $\ga$ (considered as a Lie group under addition). In fact, we obtain:
\[
   p(H_1,\ldots, H_r) f = p\left(\del{x_1},\ldots,\del{x_r}\right)f
\]
for all $f\in C^\infty(\ga)$, where $\del{x_i}f(H) := \left.\ddt\right|_{t=0}f(H+tH_i)$.

We let $\rho=\frac{1}{2} \sum_{\avar\in\Sigma^+} m_\avar \avar\in\ga^*$ be half the sum of the positive roots with multiplicity, as usual (here, $m_\avar = \dim \gg_\avar$. We then define the linear map $\gamma: \U(\gg^\C) \rightarrow \U(\ga^\C)$
by 
\[
   \gamma(D)f(a) = a^{-\rho} D_{\ga}((\cdot)^\rho f)(a)
\]
for all $f\in C^\infty(A)$ and $a\in A$, where we define $a^\lambda := e^{\lambda(\log a)}$ for all $a\in A$.

Note that for $H\in \ga^\C \subseteq S(\ga^\C) = \U(\ga^\C)$, we have that
\[
    \gamma(H)f(a) = (Hf)(a) + \rho(H)f(a) = (H+\rho(H)1)f(a)
\]
for all $a\in A$ and $f\in C^\infty(A)$.  Thus, for any $p\in\C[z_1,\ldots,z_r]$, we see that:
\[
    \gamma(p(H_1,\ldots, H_r)) = p(H_1 + \rho_1 1, \ldots, H_r + \rho_r 1),
\]
where $\rho_i = \rho(H_i)$ for all $1\leq i\leq r$.

This so-called ``$\rho$-shift'' has another interpretation: each functional $\lambda\in\ga^*$ extends canonically to a character of the associative algebra $S(\ga^\C) = \U(\ga^\C)$ by setting
\[
    \lambda(p(H_1,\ldots, H_r)) = p(\lambda(H_1),\ldots, \lambda(H_r))
\]
for all $p\in\C[z_1,\ldots,z_r]$, and every character of the characters of the algebra $S(\ga^\C) = \U(\ga^\C)$ arrives in this way.  We can consider elements of $S(\ga^\C)$ to be polynomial functions on $(\ga^\C)^*$ by setting:
\[
    D(\lambda):= \lambda(D)
\]
for each $D\in (\ga^\C)^*$. Using this identification of $S(\ga^\C)$ with the space of polynomials on $(\ga^\C)^*$, we can write $\gamma$ as:
\[
   \gamma(D)(\lambda) = D_{\ga}(\lambda + \rho).
\]

\begin{theorem}
   The linear map $\gamma$ defined above restricts to surjective algebra homomorphism
   \[
       \gamma : \U(\gg)^K \rightarrow \U(\ga)^W
   \]
   with kernel $\U(\gg)^K \cap \U(\gg^\C)\gk^\C$. Here, $W$ is the restricted Weyl group $W = W(\gg,\ga) := N_K(\ga)/Z_K(\ga)$, and $\U(\ga^\C)^W = S(\ga^\C)^W$ consists of all $W$-invariant elements of $U(\ga^\C) = S(\ga^\C)$.
\end{theorem}
\begin{corollary}
   The map
   \[
       \Gamma: \D(G/K) \rightarrow \U(\ga^\C)^W
   \]
   defined by $\Gamma(\mu(D)) = \gamma(D)$ for $D\in \U(\gg^\C)^K$ is an isomorphism of associative algebras.
\end{corollary}

For each $\avar\in(\ga^\C)^*$, one defines the function $e^\avar\in C(N\backslash G/K)^\infty$ by
\[
   e^\avar(x) := a(x)^\avar := e^{\avar(\log a(x))} 
\]
for all $x\in G/K$, where $a(x)\in A$ is chosen such that $x = na(x)K$ for some $n\in N$. Recalling the definition of the projection onto $\ga$ of operators in $\D(G/K)$, and using the fact that $\Delta$ is $N$-invariant on the left, it is not difficult to see that 
\[
    De^\avar = D_\ga(\avar) e^\avar
\]
for all $D\in\D(G/K)$, where we again consider $D_\ga\in S(\ga)$ as a polynomial on $\ga^*$.

Finally, for each $\avar\in(\ga^\C)^*$, we set:
\[
   \phi_\avar(x) = \int_K e^{i\avar + \rho}(k\cdot x)dx
\]
for each $x\in G/K$.  Then $\phi_\avar$ is a spherical function for all $\avar\in(\ga^\C)^*$ and $\phi_\avar = \phi_{\beta}$ if and only if there is $w\in W(\gg,\ga)$ s.t. $\beta=w\avar$. The ``$\rho$-shift'' is included in the integrand so that we have:
\[
    D\phi_\avar = D_\ga(i\avar + \rho)\phi_\avar = \Gamma(D)(i\avar)\phi_\avar,
\]
for any $D\in\D(G/K)$.

We can also begin with an arbitrary $W$-invariant element of $S(\ga^\C)$, say $Q\in S(\ga^\C)^W$.  Then $\Gamma^{-1}(Q) \in \D(G/K)$ is an invariant differential operator on $G/K$, and we have that
\[
    \Gamma^{-1}(Q)\phi_\avar = Q(i\avar) \phi_\avar,
\]
where we again interpret elements of $S(\ga^\C)$ as polynomial functions on $(\ga^\C)^*$.
Furthermore, we have that if $H_1,\ldots,H_r$ is a basis for $\ga$ and $q\in\C[z_1,\ldots, z_r]$ and $Q=q(H_1,\ldots,H_r)$, then
\[
   \Gamma^{-1}(Q)f(a) = q\left.\left(\del{t_1}-\rho_1 ,\ldots,\del{t_r}-\rho_r \right)\right|_{t_1=\cdots=t_r=0}f(a e^{t_1 H_1 + \cdots + t_r H_r}),
\]
for all $f\in C^\infty(N\backslash G /K)$ and $a\in A$.

Furthermore, it is well known that $\phi_\avar$ is positive-definite if $\avar\in\ga^* \subseteq (\ga^\C)^*$.  This permits one to construct a scalar Fourier transform for $K$-invariant functions on $G/K$, known as the spherical Fourier transform.  

First, we write 
\[
   \ga^+ := \{H\in\ga \mid \avar(H) > 0 \text{ for all } \avar\in\Sigma^+\}
\]
for the open Weyl chamber corresponding to $\Sigma^+$, and we note that the set $\ga'$ of regular elements of $\ga$ satisfies 
\[
   \ga' := \bigcup^{\cdot}_{w\in W} w\cdot \ga^+.
\]
Furthermore, we set
\[
   \ga^*_+ := \{\lambda\in \ga^* \mid \langle \lambda, \avar \rangle > 0 \text{ for all } \avar \in \Sigma^+ \},
\]
where $\ip{\cdot}{\cdot}$ denotes the bilinear form on $\gg^*$ which is dual to the Killing form on $\gg$ (this bilinear form on $\gg^*$ restricts to the unique (up to constant multiple) $K$-invariant real inner product on $\gs^*$).

On the other hand, every bi-invariant differential operator is clearly invariant on the left under $G$ and on the right under $K$, so we see that $Z(U(\gg^\C)) \subseteq \U(\gg^\C)^K$.  Hence the homomorphism $\mu$ in \eqref{eq:themumap} originally defined on $\U(\gg^\C)^K$, restricts to:
\begin{equation}\label{eq:restrictedmumap}
    \mu |_{Z(\U(\gg^\C))} : Z(\U(\gg^\C)) \rightarrow \D(G/K).
\end{equation}
It was shown by Helgason (see, for instance, \cite[Rem. 532]{H84}, ) that if $\gg$ is a direct sum of classical simple Lie algebras (that is, if it has no exceptional factors), then $\D(G/K) = \mu(Z(\U(\gg^\C)))$.
In particular, we have the following:
\begin{theorem}[\cite{H84}]\label{theo:center and difops}
    Let $\bsd=G/K$ be a classical bounded symmetric domain. Then $\D(G/K) = \mu(Z(\U(\gg^\C)))$.
\end{theorem}

\section{An invariant differential shift operator and the Berezin transform}\label{sec:Berezin_shifts}
In this section, we proceed to construct the differential operator $D_{\alpha}$ that shifts the parametrized Berezin transform.

For each $1\leq p\leq \infty$, we denote by $L^p(\bsd,d\mu)^K$ the space of all radial (i.e., $K$-invariant) functions in $L^p(\bsd,d\mu)$. Denote by $c(\lambda)$ the Harish-Charandra $c$-function. Recall that the spherical Fourier transform on $\bsd = G/K$ is the unitary operator
\begin{align*}
   \cF : L^2(\bsd,d\mu)^K \rightarrow L^2(\ga^*_+, |\hc(\lambda)|^{-2} d\lambda),
\end{align*}
defined by
\[
   \cF f(\lambda) := \int_\bsd f(z)\overline{\phi_{\lambda}(z)} d\mu_0(z) = 
   \int_\bsd f(z)\phi_{-\lambda}(z) d\mu_0(z)  
\]
for all $f\in L^1(\bsd,d\mu)^K \bigcap L^2(\bsd,d\mu)^K$ and $\lambda\in \ga^*_+$.

Unterberger and Upmeier calculated the spherical Fourier transform of $\varphi_\alpha$ for all $\alpha> \gen -1$ in \cite{U94} and arrived at the following result:  
\begin{theorem}[\cite{U94}]
   \[
      \cF(\varphi_\alpha)(\lambda) = \prod_{j=1}^r \frac{\Gamma(i\lambda_j + \alpha - \frac{\gen-1}{2}) \Gamma(-i\lambda_j + \alpha - \frac{\gen-1}{2})} {\Gamma(-\rho_j + \alpha - \frac{\gen-1}{2}) \Gamma(\rho_j + \alpha - \frac{\gen-1}{2})}
   \]
   for all $\lambda = (\lambda_1, \cdots, \lambda_r) \in \ga^*_+$, where
   \[
      \rho_j = \frac{a}{2}(j-1)+\frac{b+1}{2}
   \]
   for all $1\leq j \leq r$.
\end{theorem}

We observe that this theorem has a simple corollary:
\begin{corollary}
For all $\alpha>\gen-1$, one has that:
   \[
      \widehat{\varphi_{\alpha+1}}(\lambda) = \prod_{j=1}^r \frac{(i\lambda_j + \alpha - \frac{\gen-1}{2}) (-i\lambda_j + \alpha - \frac{\gen-1}{2})}  {(-\rho_j + \alpha - \frac{\gen-1}{2}) (\rho_j + \alpha -\frac{\gen-1}{2})} \widehat{\varphi_\alpha}(\lambda).
   \]
\end{corollary}

   This corollary can be rewritten as follows:
   \[
      \prod_{j=1}^r \left( (\alpha-\frac{\gen-1}{2} )^2 - \rho_j^2 \right) \widehat{\varphi_{\alpha+1}}(\lambda) = \prod_{j=1}^r \left( (\alpha - \frac{\gen-1}{2})^2 + \lambda_j^2 \right) \widehat{\varphi_\alpha}(\lambda)
   \]
We immediately notice that $\widehat{\varphi_{\alpha+1}}$ is an even symmetric polynomial times $\widehat{\varphi_\alpha}$.  In fact, we can explicitly expand the products on both sides of this equation as linear combinations of elementary symmetric polynomials, obtaining:
\begin{align*}
     \sum_{k=0}^r (-1)^k(\alpha - \frac{\gen-1}{2})^{2(r-k)}& \sum_{1\leq i_1 < \cdots < i_k \leq r} p_{j_1}^2 \cdots p_{j_k}^2 \widehat{\varphi_{\alpha+1}}(\lambda) \\
   &= 
   \sum_{k=0}^r (\alpha - \frac{\gen-1}{2})^{2(r-k)} \sum_{1\leq i_1 < \cdots < i_k \leq r} \lambda_{i_1}^2 \cdots \lambda_{i_k}^2 \widehat{\varphi_\alpha}(\lambda). 
\end{align*}
This can be written succinctly as:
\begin{align*}
     \widehat{\varphi_{\alpha+1}}(\lambda) =
   \frac{1}{d(\alpha)} \sum_{k=0}^r (\alpha - \frac{\gen-1}{2})^{2(r-k)} \delta_k(\lambda) \widehat{\varphi_\alpha}(\lambda),
\end{align*}
where 
\begin{equation}\label{eq:dnu}
    d(\alpha) = \prod_{j=1}^r \left( (\alpha-\frac{\gen-1}{2} )^2 - \rho_j^2 \right),
\end{equation}
and where, for $1\leq k \leq r$, $\delta_k$ is the fundamental even symmetric polynomial in $r$ variables of degree $k$, which is given by:
\[
   \delta_k(\lambda) = \sum_{1\leq i_1 < \cdots < i_k \leq r} \lambda_{i_1}^2 \cdots \lambda_{i_k}^2.
\]

In light of the above, we have the following result:

\begin{proposition}
   For each $\alpha>\gen-1$ there exists a $G$-invariant differential operator $D_\alpha$ on $\bsd$ such that $D_\alpha \varphi_\alpha = \varphi_{\alpha+1}$.  In particular, in light of the above, we can let $D_\alpha$ be the differential operator such that its $N$-radial part $\Gamma D_\alpha$ is given by the following differential operator on $\ga$:
   \[
      \Gamma D_\alpha = \frac{1}{d(\alpha)} \prod_{j=1}^r \left((\alpha - \frac{\gen-1}{2})^2 + (\frac{1}{i}\frac{\partial}{\partial t_j} - \rho_j)^2 \right).
   \]
\end{proposition}

Recall the Berezin transform $B_\alpha f$ of a function $f$ \eqref{eq:function Berezin transform}.

\begin{theorem}\label{theo:Berezin shift}
   Assume that $\bsd$ be a bounded symmetric domain and consider $\Berg=\cA^2_\nu$ with $\nu>\gen-1$. Let $\alpha>\gen-1$ and $1\leq p\leq \infty$.  For all $f\in \Lp$, 
   \[
      B_{\alpha+1}f = D_\alpha (B_\alpha f)
   \]
   Furthermore, if $\bsd$ is a classical domain, and if $f\in C^\infty \cap \Lp$ and $D_\alpha f\in \Lp$, then
   \[
      B_{\alpha+1}f = B_\alpha( D_\alpha f ).
   \]
\end{theorem}

\begin{proof}
   Let $f\in \Lp$. By Lemma \ref{lem:function Berezin transform convol}, we have
   \[
      B_\alpha f = f * \varphi_\alpha,
   \]
   with $\varphi_\alpha \in L^1(\bsd)$.  It follows that
   \begin{align*}
      B_{\alpha+1}f & = f * \varphi_{\alpha+1} \\
                 & = f * (D_\alpha \varphi_\alpha) \\
                 & = D_\alpha(f * \varphi_\alpha).
   \end{align*}
   If, in addition, we suppose that $\bsd$ is classical, by Theorem \ref{theo:center and difops}, we can obtain $Z\in Z(U(\gg^\C))$ s.t. $D_{\nu+\alpha}=\mu(Z)$. Hence, if $f\in C^\infty(\bsd)$ and $Df\in \Lp$ for all $G$-invariant differential operators $D\in \D(G/K)$, then we have
   \begin{align*}
      B_{\alpha +1}f & = (D_\alpha f) * \varphi_\alpha \\
                  & = B_\alpha (D_\alpha f). \qedhere
   \end{align*}
\end{proof}

By repeatedly applying this theorem, we immediately obtain the following corollary, which allows us to write the Berezin transform with parameter $\alpha+m$ as the composition of the Berezin transform with an invariant differential operator:
\begin{corollary}\label{coro:shifting function BT}
   Let $\alpha>\gen-1$, $m\in\N$, and $1\leq p\leq \infty$.  We define the invariant differential operator
   \[
      D_\alpha^m := D_{\alpha+k-1} D_{\alpha+k-2} \cdots D_{\alpha+1} D_\alpha \in \D(G/K).  
   \]
   Then 
   \[
      B_{\alpha+m}f = D_\alpha^m(B_\alpha f)
   \]
   for all $f\in \Lp$.  Furthermore, if $\bsd$ is a classical domain, and if $f\in C^\infty(G) \cap L^q(\bsd)$ and $D_\alpha^k f\in \Lp$ for all $1\leq k \leq m$, then
   \[
      B_{\alpha+m} f = B_\alpha (D_\alpha^m f).
   \]
\end{corollary}


\section{Invariant differential operators acting on bounded operators on \texorpdfstring{$\Berg$}{Bergman Space}}\label{sec:Diffops on ops}

The notion of differentiating operators has appeared in several settings. For instance, in \cite{S04} and \cite{S05},
the Laplacian of an operator on Bergman space over the unit disc was developed using the Berezin transform, and 
an analogous operator Laplacian on Fock spaces was investigated in \cite{DM24}. Moreover, smooth operators for the translations defined via the Weyl representation
were studied in \cite{L24} and \cite{FL25}, adopting a more representation-theoretic perspective. In this section, we 
generalize both of these approaches to operators on the Bergman space $\Berg$
 of a bounded symmetric domain $\bsd$, and we observe that Su\'arez's approach yields a broader framework.

\subsection{Smooth vectors and the derived representation}
We now recall that if $(\pi,V)$ strongly-continuous representation of a Lie group $G$ in a Banach space $V$, then there is a dense subspace $V^\infty \subseteq V$ of so-called smooth vectors; that is,
\[
   V^\infty := \{ v\in V\mid G\rightarrow V, g\mapsto \pi(g)v \text{ is smooth} \}.
\]
In fact, $V^\infty$ is a $G$-invariant subspace of $V$ which carries a natural Banach  space topology for which the action of $G$ is strongly continuous.  G{\aa}rding showed that $V^\infty$ is dense by using the following construction. 

Recall the definition of
\[
    \pi(f) := \int_G f(x)\pi(x)dx \in \cB(V)
\]
for all $f\in C_c(G)$, we obtain the so-called integrated representation (also denoted by $\pi$) of the convolution algebra $C_c(G)$ on $V$.  If $V$ is a reflexive Banach space, then this integrated representation extends to a representation of the Banach algebra $L^1(G)$, and if $V$ is a Hilbert space, then it gives a $*$-representation of $L^1(G)$ as a Banach $*$-algebra.

In \cite{Garding47}, G\aa{}rding  showed that 
\[
    \pi(C^\infty_c(G))V \subseteq V^\infty,
\]
and later in \cite[Theorem 3.3]{DixmierMalliavin78}, Dixmier and Malliavin  showed that this is an equality:
\begin{equation}\label{eq:DixMalli}
    \pi(C^\infty_c(G))V = V^\infty
\end{equation}

As usual, we defines the derived representation
\begin{equation}\label{eq:derivedrep}
    \mathrm{d}\pi(X)v = \left.\ddt\right|_{t=0} \pi(\exp(tX))v,
\end{equation}
for all $X\in\gg$ and $v\in V^\infty$.  It is notdifficult to show that $\mathrm{d}\pi$ defines a Lie-algebra representation of $\gg$ on $V^\infty$, which then canonically extends to an associative-algebra representation of $\U(\gg^\C)$.

For $X\in\gg$, we define $\domdpi{\pi}{X}$ to be the domain of the infinitesimal generator of the one-parameter group $t\mapsto \pi(\exp tX)$. Then for $\Uvar=X_1\cdots X_m\in \U(\gg^\C)$, the domain $\domdpi{\pi}{\Uvar}$ is defined inductively, and then to all $Z\in \U(\gg^\C)$ by linearity. Then by Theorem 1.1 in \cite{Goodman69} we have a characterization of $V^\infty$ by powers of
generators of one-parameter subgroups:
\begin{equation}\label{eq:smoothvectors}
    V^\infty=\bigcap_{j=1}^d \bigcap_{m=1}^\infty \domdpi{\pi}{X_j^m}.
\end{equation}

More can be said via the elliptic elements in $U(\gg^\C)$. In particular, the following theorem, which is due to Goodman, appears in \cite{Poulsen72}. It
is a Banach-space version of results of Nelson \cite{Nelson59} stated for the Hilbert space case.

\begin{theorem}\cite[Theorem 1.2]{Poulsen72}\label{theo:Nelson}
    Let $L\in U(\gg^\C)$ be elliptic and let $A=\overline{\mathrm{d}\pi(L)}$ (the closure of $\mathrm{d}\pi(L)$). Then
    $$V^\infty=\bigcap_{m=1}^\infty \mathrm{Dom}(A^m).$$
\end{theorem}

The above theorem is particularly useful. 
Now, as in section \ref{sec:invariant_differential_operators}, assume that $G$ is a connected semisimple Lie group with finite center and that $K$ is a maximal compact subgroup. Recall that a vector $v\in V$ is called $K$-invariant if $\pi(k)v=v$ for all $k\in K$.

Let $B$ denote the Killing form of $\gg$. With respect to the Cartan
decomposition $\gg=\mathfrak k\oplus\mathfrak s$, $B$ is negative
definite on $\mathfrak k$ and positive definite on $\mathfrak s$. Choose bases
$Y_1,\dots,Y_s$ of $\mathfrak s$ and $X_1,\dots,X_k$ of $\mathfrak k$ such that
$B(Y_i,Y_j)=\delta_{ij}$ and $B(X_i,X_j)=-\delta_{ij}$. Then the \emph{Casimir
element} is given by
\[
  \Omega_G=\sum_{i}Y_i^2-\sum_{j}X_j^2\in U(\gg^\C).
\]
We point to \cite[Theorem 3.11.1]{Varadarajan74}
It is independent of the choice of bases and belongs to
$Z(U(\gg^\C))$. Put $\Omega_K=\sum_j X_j^2\in
U(\mathfrak k^\C)$ and
\[
  \Delta=\sum_i Y_i^2+\sum_j X_j^2=\Omega_G+2\Omega_K.
\]
In particular, $\Delta$ is elliptic. 

\begin{proposition}\label{prop:Nelsons_radial}
    Let $v\in V$ be a $K$-invariant vector. Then if $v\in \domdpi{\pi}{\Omega_G^m}$ for all $m\in\N$, we have  that $v\in V^\infty$.
\end{proposition}

\begin{proof}
    The forward direction is trivial. To prove the other direction, we apply Theorem \ref{theo:Nelson} via the elliptic operator $\Delta$. First note that  $d\pi(X)v=0$ for all $X\in\mathfrak k$ as $\pi(k)v=v$ for $k\in K$. Therefore,  $d\pi(\Omega_K)v=0$ and 
    $$d\pi(\Delta)v=d\pi(\Omega_G)v$$
    for $v\in \domdpi{\pi}{\Omega_G}\subseteq \domdpi{\pi}{\Delta}$.
    Also, since $\Omega_G$ is central, $d\pi(\Omega_G)v$ is $K$-invariant. Then it follows that $$d\pi(\Delta)^m v=d\pi(\Omega)^m v, \quad m\in\N.$$  
    Therefore, $\domdpi{\pi}{\Omega_G^m}\subseteq \mathrm{Dom}(A^n)$ where $A$ be the closure of $d\pi(\Delta)$. Now the result follows by Theorem \ref{theo:Nelson}.
\end{proof}

\subsection{Derivatives of functions via regular representations}
Now we turn to the smooth vectors of the left and right regular representations $(\ell,L^p(G))$ and $(r,L^p(G))$ of $G$. Given a subspace $V$ of $L^p(G)$, we denote the smooth vectors in $V$ w.r.t. $\ell$ and $r$ by $V^{\infty,\ell}$ and $V^{\infty,r}$ respectively.

\begin{proposition}
    The smooth vectors  satisfy the inclusions:
    $$C_c^\infty(G) \subseteq L^p(G)^{\infty,\ell} \subseteq C^\infty(G),$$
    and
    $$C_c^\infty(G) \subseteq L^p(G)^{\infty, r} \subseteq C^\infty(G),$$
    where $1\leq p < \infty$.
\end{proposition}
\begin{proof}
   Note that by \eqref{eq:DixMalli}, we have that $C^\infty_c(G)\ast L^p(G) = L^p(G)^{\infty,\ell}$ which implies $L^p(G)^\infty \subseteq C^\infty(G)$. Also, the Theorem 3.1 in \cite{DixmierMalliavin78} states that (check) any $f\in C^\infty_c(G)$ can be written as a finite sum of convolutions:
   $$f=\sum_{j=1}^N g_i\ast h_i,\quad g_i,h_i\in C^\infty_c(G).$$
   Then by \eqref{eq:DixMalli}, we have that $f\in L^p(G)^{\infty,\ell}$. Therefore, $C_c^\infty(G) \subseteq L^p(G)^{\infty,\ell}$.
\end{proof}

 Similarly, for the left and right regular representations $(\ell,\cbuGL)$ and $(r,\cbuGR)$ on the spaces of left (respectively, right) uniformly continuous bounded functions, one has:
\begin{align*}
   C_c^\infty(G) & \subseteq \cbuGL^{\infty,\ell} \subseteq C^\infty(G), \\
   C_c^\infty(G) & \subseteq \cbuGR^{\infty,r} \subseteq C^\infty(G),
\end{align*}
Furthermore, it is clear that for all $X\in\gg$, we have that
\begin{align}
   \label{eq_derivative_of_left_translation}
      \mathrm{d}\ell(X)f(x) & = \left.\ddt\right|_{t=0} f(\exp^{-tX}x)\\
   \label{eq_derivative_of_right_translation}
      \mathrm{d}r(X)f(x) & = \left.\ddt\right|_{t=0} f(x\exp^{tX})\\
      \label{eq:dl_dr}
      \mathrm{d}r(X)f(x) & = \mathrm{d}\ell(\Ad(x^{-1})X)f(x)
\end{align}
for all $x\in G$ and $f\in L^p(G)^\infty$, $1\leq p <\infty$, as well as for $f$ in $\cbufL^\infty$ (in the case of \ref{eq_derivative_of_left_translation}) or $\cbufR^\infty$ (in the case of \ref{eq_derivative_of_right_translation}). The above identities extend naturally to $\U(\gg^\C)$. In fact, it is not difficult to show that $\mathrm{d}r$ induces the identification of $\U(\gg^\C)$ with $\D(G)$ that we made previously. However, $\mathrm{d}\ell(X)$ is not usually left-invariant.

\begin{lemma}\label{lem:diffops_convol_f}
   Let  $\Uvar\in U(\gg^\C)$ and let $1\leq p,q,r\leq \infty$ be such that $\frac{1}{p}+\frac{1}{q}-1=\frac{1}{r}$. Then
   \begin{enumerate}
       \item  If $f\in L^p(G)^{\infty,\ell}$ and $g\in L^q(G)^{\infty,\ell}$, then we have that $f\ast g\in L^r(G)^{\infty,\ell}$ and
       \begin{align*}
           \mathrm{d}\ell(\Uvar) (f\ast g) = (\mathrm{d}\ell(\Uvar)f)\ast g.
       \end{align*}
       \item  If $f\in L^p(G)^{\infty,r}$ and $g\in L^q(G)^{\infty,r}$, then we have that $f\ast g\in L^r(G)^{\infty,r}$ and
       \begin{align*}
           \mathrm{d}r(\Uvar) (f\ast g) = f\ast (\mathrm{d}r(\Uvar)g).
       \end{align*}
       In particular for $D\in \D(G)$, we have
       \begin{align*}
           D (f\ast g) = f\ast (Dg).
       \end{align*}
   \end{enumerate}
\end{lemma}

We omit the proof of the above lemma as it is well-known and because the idea of the proof is similar to the proof of Proposition \ref{prop:diffops_convol} that we address in detail.

Now we let
\begin{align*}
   ^\dagger : U(\gg^\C) & \rightarrow U(\gg^\C) \\ 
                   \Uvar & \mapsto \Uvar^\dagger 
\end{align*}
denote the unique involutive associative-algebra antihomomorphism on $U(\gg^\C)$ induced by the involutive Lie-algebra antihomomorphism $X\mapsto -X$ on $\gg^\C$.  

Similarly, it is not difficult to show that $\mathrm{d}\ell(X)$ extends to a unique differential operator in $\diffops(G)$.  However, this differential operator is \textit{right}-$G$-invariant but, in general, is not \text{left}-$G$-invariant.  Nevertheless, if $\Uvar\in Z(\U(\gg^\C))$, then $\mathrm{d}\ell(\Uvar)$ commutes with $\ell(x)$ for all $x\in G$ and is thus a bi-invariant (and in particular, left-invariant) operator on $G$. In fact, from (\ref{eq_derivative_of_left_translation}) and (\ref{eq_derivative_of_right_translation}) it is clear that $\mathrm{d}\ell(\Uvar^\dagger)f = \mathrm{d}r(\Uvar)f$ for all $\Uvar\in Z(\U(\gg^\C))$ and all $f\in C^\infty(G)$.

We now suppose that $\gg$ is classical (i.e., non-exceptional).  If $\gt$ is a Cartan subalgebra of $\gk$, then $\gh:= \gt \oplus \ga$ is a Cartan subalgebra of $\gg$.  One identifies $Z(U(\gg^\C))$ with $\D(G\times G/ \diag G)$, where we note that $G\times G/\diag G$ is a pseudo-Riemannian symmetric space.  Then $\D(G\times G)/ \diag G) = U(\gg^\C \oplus \gg^\C)^{\diag G}/(U(\gg^\C \oplus \gg^\C)^{\diag G} \cap U(\gg^\C \oplus \gg^\C) (\diag \gg^\C))$.  

One obtains the Harish-Chandra isomorphism
\[
    \Gamma_\gh :  Z(\U(\gg^\C)) \rightarrow S(\gh^\C)^{\widetilde{W}},
\]
where $\widetilde{W}$ is the Weyl group of the root system $\Delta(\gg^\C,\gh^\C)$.  Using the Killing form, we can embed $(\ga^\C)^*$ as a subspace of $(\gh^\C)*$.  Since $\ga$ and $\gt$ are orthogonal with respect to the Killing form, this embedding corresponds to extending weights on $\ga$ to $\gh$ by setting them to zero on $\gt$.  One shows that the restriction to $\ga^*$ of a $\widetilde{W}$-invariant polynomial on $(\gh^\C)^*$ gives a $W$-invariant polynomial on $\ga^*$.

Using the identification of $S(\gh^\C)$ with the algebra of polynomials on $(\gh^\C)^*$ and of $S(\ga^\C)$ with the algebra of polynomials on $(\ga^\C)^*$, one has a homomorphism
\[
   \Upsilon: S(\gh^\C)^{\widetilde{W}} \rightarrow S(\ga^C)^W.
\]
Helgason showed that, if $\gg$ is classical (that is, non-exceptional), then this homomorphism is surjective, and, furthermore, that $\Gamma = \Upsilon \circ \Gamma_\gh$.  

Since we have assumed that $\gg$ is a Hermitian Lie algebra, we know that $\Sigma(\gg,\ga)$ is of type $C_r$ or $BC_r$, depending on whether $G/K$ is of tube type or not.  Therefore, the Weyl group $W$ is the group generated by permutations of the elements $e_1,\ldots,e_r$ of the Jordan frame, together with the reflections $r_i$ (where $1\leq i\leq n$) that satisfy:
\[
   r_i(e_j) = \left\{ \begin{matrix}
                                -e_j, \text{ if } i = j \\
                                e_j \text{ if } i\neq j
                              \end{matrix} \right.
\]
It follows that all elements of $S(\ga^\C)^W$ are sums of terms of even order.  

Therefore, we have the following:
\begin{lemma}\label{lem:centralZ}
    Let $\bsd=G/K$ be a bounded symmetric domain and let  $\Zvar \in Z(U(\gg^\C))$. Then
    \[
   \mathrm{d}r(Z)f = \mathrm{d}r(Z^\dagger) f =\mathrm{d}\ell(Z)f
\]
for all $f\in \Lp^{\infty,\ell}$, where $1\leq p<\infty$ (and similarly for $\cbufL^{\infty,\ell}$ and $\cbufR^{\infty,r}$).
\end{lemma}

\subsection{Derivatives of operators via representation theory}
Now we consider operators on $\Berg$ for $\nu>\gen-1$.
The representation $x\mapsto \ActopL{x}$ \eqref{eq:translationops} is a representation of $G$ acting on the Banach space $\schop$ for $1\leq p\leq \infty$. This representation is strongly continuous for $1\leq p<\infty$, and for the case $p=\infty$, one has to consider $\cbuopL$ instead of $\bdd$ to obtain strong continuity.
 
 Now we turn to smooth vectors of the above representation.
 Even though the representation $\lambda^\nu$ is not strongly continuous on $\bdd$, we still consider its smooth vectors, denoted $\bdd^\infty$. More generally, given $V\subseteq \bdd$, we write $V^\infty$ for the smooth vectors in $V$ (in the topology of $V$) under the action $\lambda^\nu$.  
 Finally, $\Berg^\infty$ denotes the smooth vectors in $\Berg$ for the representation $\pinu$.

 \begin{lemma}\label{lem:smooth rank one}
     Let $f,g\in \Berg^\infty$. Then the rank one operator $f\otimes \overline{g}$ is in $ \traceop^\infty$.
 \end{lemma}
 \begin{proof}
    Note that for $x\in G$, 
     $$\actopL{x}{f\otimes \overline{g}}= (\pinu(x)f)\otimes \overline{(\pinu(x)g)}.$$
     Then $x\mapsto \actopL{x}{f\otimes \overline{g}}$ is smooth by bilinearity and because of the norm bound: $\|f\otimes \overline{g}\|_1=\|f\|_{\Berg}\|g\|_{\Berg}$. 
 \end{proof}

 \begin{example}
     Consider the rank one operator $\Phi^\nu=1\otimes\overline{1}$. Then $1\in \Berg^\infty$. This is because $z\mapsto k^\nu_z$ is real analytic. Hence, by Lemma \ref{lem:smooth rank one}, we get that $1\otimes \overline{1} \in \traceop^\infty$.
 \end{example}
 
 Furthermore, we have the following inclusions:
 \begin{align*}
    \text{span}\{\Berg^\infty\otimes \overline{\Berg^\infty}\}\subset \schop^\infty\subset \cbuopL^\infty \subseteq \bdd^\infty.
\end{align*}

\begin{remark}
    {\rm Let $1\leq p<\infty$. Then $\schop^\infty$ and $\cbuopL^\infty$ are dense subsets of $\schop$ and $\cbuopL$ respectively. However, we do not have that $\bdd^\infty$ is a dense subset of $\bdd$. This is in fact untrue by the results in the next section.}
\end{remark}

Next, we investigate how the derived representation $d\lambda^\nu$ of the representation $\lambda^\nu$ interacts with QHA convolutions.

\begin{proposition}\label{prop:diffops_convol}
   Let  $\Uvar\in U(\gg^\C)$ and let $1\leq p,q,r\leq \infty$ be such that $\frac{1}{p}+\frac{1}{q}=1+\frac{1}{r}$. Then 
   \begin{enumerate}
       \item  If $f\in L^p(G/K)^{\infty,\ell}$ and $S\in \schopq{q}$, then we have that $f\ast S\in \schopq{r}^\infty$ and
       \begin{align*}
           \mathrm{d}\lambda^\nu(\Uvar) (f\ast S) = (\mathrm{d}\ell(\Uvar)f)\ast S 
       \end{align*}
   \item  If $S\in \schop^\infty$ and $A\in \schopq{q}^K$, then we have that $S\ast A\in L^r(G/K)^{\infty,\ell}$ and
       \begin{align*}
            \mathrm{d}\ell(\Uvar)(S\ast A)= (\mathrm{d}\lambda^\nu(\Uvar)S)\ast A.
       \end{align*}
   \item  Assume that $\Uvar$ is central. Then if $f\in L^p(G/K)$  and $S\in \schopq{q}^\infty$, we have that $f\ast S\in \domdpi{\lambda^\nu}{\Uvar}\subset \schopq{r}$ and
       \begin{align*}
           \mathrm{d}\lambda^\nu(\Uvar) (f\ast S)= f\ast (\mathrm{d}\lambda^\nu(\Uvar)  S).
       \end{align*}
       \item  If $S\in \schop$ and $A\in (\schopq{q}^\infty)^K$, we have that $S\ast A\in L^r(G/K)^{\infty,r}$ and
       \begin{align*}
           \mathrm{d}r(\Uvar)(S\ast A)= S\ast (\mathrm{d}\lambda^\nu(\Uvar)A).
       \end{align*}
       Additionally, if $\Uvar$ is central, then $S\ast A\in \domdpi{\ell}{\Uvar}\subset L^r(G/K)$ and
       \begin{align*}
           \mathrm{d}\ell(\Uvar)(S\ast A)= S\ast (\mathrm{d}\lambda^\nu(\Uvar)A).
       \end{align*}
   \end{enumerate}
\end{proposition}
\begin{proof}
    To prove (1), note that for $X\in \gg$
    \begin{align*}
         \mathrm{d}\lambda^\nu(X)(f\ast S)&= \left.\ddt\right|_{t=0} \actopL{\exp(tX)}{f\ast S}\\
         &= \left.\ddt\right|_{t=0} ( \actfL{\exp(tX)}{f})\ast S\quad \text{(by Lemma \ref{lem:translation_convol})}\\
         &= \bigg(\left.\ddt\right|_{t=0} \actfL{\exp(tX)}{f}\bigg)\ast S\\
         &= ( \mathrm{d}\ell(X){f})\ast S,
    \end{align*}
    where the third equality follows from the continuity of the convolution imposed by the Young's inequality and the fact that $\mathrm{d}\ell(X)f\in \Lp$.
    
    The above identity naturally extends to products of elements in $\gg$ and we obtain
    $$\mathrm{d}\lambda^\nu(X_1\cdots X_m)(f\ast S)=( \mathrm{d}\ell(X_1\cdots X_m){f})\ast S.$$
    Then by linearity the identity holds for all $\Uvar\in U(\gg^\C)$ and $f\ast S\in \schopq{r}^\infty$ by \eqref{eq:smoothvectors}.

    The proof of (2) is similar and follows from the fact $\actfL{x}{(S\ast A)}=(\actopL{x}{S}\ast A)$. Also, we note that $S\ast A$ is a function on $\bsd$ as $A$ is $K$-invariant.

    To prove (3), note that $ \mathrm{d}\lambda^\nu(\Uvar)$ is an intertwining operator for the representation $\lambda^\nu$, as $\Uvar\in Z(U(\gg^\C))$. Therefore,
    \begin{align*}
       \mathrm{d}\lambda^\nu(\Uvar)(f\ast S)&= \int_G f(y)\mathrm{d}\lambda^\nu(\Uvar)\actopL{y}{S} \haar{y}\\
        &= \int_G f(y)\actopL{y}{\mathrm{d}\lambda^\nu(\Uvar)S} \haar{y}\\
        &=f\ast (\mathrm{d}\lambda^\nu(\Uvar)S).
    \end{align*}
   
    To prove (4), note that for $X\in \gg$
    \begin{align*}
        \mathrm{d} r(X)(S\ast A) &=\left.\ddt\right|_{t=0} \actfR{\exp(tX)}{(S\ast A)}\\
        &= \left.\ddt\right|_{t=0} (S\ast \actopL{\exp(tX)}{A})\quad \text{(by Lemma \ref{lem:rtranslation_convol})}\\
        &= S\ast \mathrm{d}\lambda^\nu(X)A.
    \end{align*}
    Then $S\ast A\in L^r(G/K)^{\infty,r}$ and $\mathrm{d} r(\Uvar)(S\ast A)=S\ast \mathrm{d}\lambda^\nu(\Uvar)A$ by the same argument as in (1).
    When $\Uvar$ is central the second statement follows by Lemma \ref{lem:centralZ}.    
\end{proof} 

As a consequence, we note that Berezin transform of an operator is also smooth:
\begin{corollary}
    Let $1\leq p\leq \infty$ and let $S\in \schop$. Then $B^\nu(S)\in L^p(G/K)^{\infty,r}$.
\end{corollary}

Since the representation $\lambda^\nu$ behaves much like a left translation of operators, and because we do not have a notion of right translation of operators at our disposal, the same cannot be said about Toeplitz operators. However, we have the following:

\begin{corollary}
    Let $1\leq p\leq \infty$ and let $a\in \Lp^K$. Then $T_a\in \schop^\infty$.
\end{corollary}
\begin{proof}
    Recall the Casimir element $\Omega_G\in Z(U(\gg^\C))$. Then by the above Proposition, we obtain $\Omega_GT_a=\Omega_G c_\nu (a\ast \Phi^\nu)=c_\nu (a\ast \Omega_G\Phi^\nu)$. In particular, as $\Phi^\nu\in \traceop^\infty$, we obtain that $T_a\in \domdpi{\lambda^\nu}{\Omega_G^m}$ for all $m\in\N$. Also, $T_a$ is a radial operator. Therefore, by Proposition \ref{prop:Nelsons_radial}, the result follows.
\end{proof}

\subsection{Invariant differential operators acting on operators}
Recall that given $D\in \D(G/K)$, there exists $\Uvar\in U(\gg^\C)^K$ s.t. $D=\mathrm{d}r(Z)$. However, this correspondence does not generally hold for the left-regular representation $\ell$, unless the domain $\bsd$ is a classical domain (see Theorem \ref{theo:center and difops}).
Given $Z\in U(\gg^\C)$ and a differential operator $D\in \D(G/K)$ such that $D f = \mathrm{d}\ell(Z)f$, we can define
$$DS=\mathrm{d}\lambda^\nu(Z)S, \quad S\in \bdd^\infty.$$
As $\bdd^\infty$ is the common domain for all $D\in \D(G/K)$, we could consider a larger domain $\mathrm{Dom}_{\mathrm{op}}(D)$, for a fixed differential operator $D$, as we have done below. This is indeed the approach of Su\'arez in \cite{S04}.

\begin{definition}
    For $1\leq p\leq \infty$ and $D\in \D(G/K)$, define $\mathrm{Dom}_{\mathrm{op}}^p(D)\subseteq \schop$ to be
    $$\mathrm{Dom}_{\mathrm{op}}^p(D)=\{S\in \schop\mid D(S\ast \Phi^\nu)=T\ast \Phi^\nu \text{ for some } T\in \schop.\}$$
    Then for $S\in \mathrm{Dom}_{\mathrm{op}}^p(D)$, we define
    $$DS:=T.$$
\end{definition}

Due to the injectivity of the Berezin transform (i.e. the regularity of the operator $\Phi^\nu$), the above definition is well-defined.

Note that for $p<q$, we have $\mathrm{Dom}_{\mathrm{op}}^p(D)\subseteq \mathrm{Dom}_{\mathrm{op}}^q(D)\subseteq  \mathrm{Dom}_{\mathrm{op}}^\infty(D)=: \mathrm{Dom}_{\mathrm{op}}(D)$. 

\begin{proposition}\label{prop:diffops_convol_dom}
    Let $D\in \D(G/K)$ and let  $1\leq p,q,r\leq \infty$ be such that $\frac{1}{p}+\frac{1}{q}=1+\frac{1}{r}$. Then
      \begin{enumerate}
                 \item  If $f\in L^p(G/K)$  and $S\in \mathrm{Dom}_{\mathrm{op}}^{q}(D)$, then we have that $f\ast S\in \mathrm{Dom}_{\mathrm{op}}^{r}(D)$ and
       \begin{align*}
           D (f\ast S)= f\ast (D S).
       \end{align*}
       \item  Assume that there exists $\Uvar\in U(\gg^\C)$ s.t. $D=\mathrm{d}\ell(Z)$. Then for $f\in L^p(G/K)^{\infty,\ell}$ and $S\in \schopq{q}$, we have that $f\ast S\in \mathrm{Dom}_{\mathrm{op}}^r(D)$ and
       \begin{align*}
          D(f\ast S) = (D f) \ast S.
       \end{align*}
       \item  If $S\in \schop$ and $A\in \mathrm{Dom}_{\mathrm{op}}^q(D)^K$, then we have that 
       \begin{align*}
           D(S\ast A)= S\ast (DA).
       \end{align*}
       \item Assume that there exists $\Uvar\in U(\gg^\C)$ s.t. $D=\mathrm{d}\ell(Z)$. Then for $S\in \mathrm{Dom}_{\mathrm{op}}^p(D)$ and $A\in \schopq{q}^K$, we have that 
       \begin{align*}
           D(S\ast A)= (DS)\ast A.
       \end{align*}
      \end{enumerate}
\end{proposition}
\begin{proof}

    To prove (1), 
    \begin{align*}
        D((f\ast S)\ast \Phi^\nu)&=D(f\ast (S\ast \Phi^\nu)) \quad \text{(by associativity)}\\
        &= f\ast D(S\ast \Phi^\nu) \quad \text{(by Lemma \ref{lem:diffops_convol_f})}\\
        &= f\ast (DS\ast \Phi^\nu) \quad \text{(by definition)}\\
        &= (f\ast DS)\ast \Phi^\nu
    \end{align*}
    Hence by the injectivity of the Toeplitz quantization, $D(f\ast S)=f\ast DS$, proving (1).

    To prove (2), note that 
    \begin{align*}
        D((f\ast S)\ast \Phi^\nu)&= \mathrm{d}\ell(Z)(f\ast (S\ast \Phi^\nu))\quad \text{(by associativity)}\\
        &=(\mathrm{d}\ell(Z)f)\ast (S\ast \Phi^\nu)\quad \text{(by Lemma \ref{lem:diffops_convol_f})}\\
        &=(Df)\ast (S\ast \Phi^\nu)\\
        &= ((Df)\ast S)\ast \Phi^\nu \quad \text{(by associativity)}.
    \end{align*}
    Hence $D(f\ast S)=(Df)\ast S$, as required.

    To prove (3), note that:
    \begin{align*}
        D(S\ast A) \ast \varphi_\nu 
        &= c_\nu ((S\ast A)\ast \Phi^\nu\ast \Phi^\nu) \quad \text{(by Lemma \ref{lem:convolution of Phinu})}\\
        &= c_\nu D((S\ast \Phi^\nu)\ast (A\ast \Phi^\nu))\quad \text{(as $A$ is radial by Theorem \ref{theo:radial_convol})}\\
        &= c_\nu (S\ast \Phi^\nu)\ast D(A\ast \Phi^\nu)\quad \text{(by Lemma \ref{lem:diffops_convol_f})}\\
        &= c_\nu (S\ast \Phi^\nu)\ast (DA\ast \Phi^\nu)\quad \text{(by definition)}\\
        &= (S\ast DA)\ast \varphi_\nu.
    \end{align*}  
    Now since $B_\nu(f)=f \ast \varphi_\nu$ (Lemma \ref{lem:function Berezin transform convol}), and by the injectivity of the Berezin transform, we obtain that $D(S\ast A) =S\ast DA$ as required.
    
    To prove (4), note that
    \begin{align*}
        D(S\ast A) \ast \varphi_\nu \qquad \qquad 
        &=c_\nu D ((S\ast A)\ast (\Phi^\nu\ast \Phi^\nu)) \quad \text{(by Lemma \ref{lem:convolution of Phinu})}\\
        &=c_\nu \mathrm{d}\ell(Z) ((S\ast \Phi^\nu)\ast (A\ast \Phi^\nu)) \quad \text{(as $A$ is radial)}\\
        &= c_\nu (\mathrm{d}\ell(Z)(S\ast \Phi^\nu))\ast (A\ast \Phi^\nu)\quad \text{(by Lemma \ref{lem:diffops_convol_f})}\\
        &=c_\nu (D(S\ast \Phi^\nu))\ast (A\ast \Phi^\nu)\\
        &= c_\nu (DS\ast \Phi^\nu)\ast (A\ast \Phi^\nu) \quad \text{(by definition)}.
    \end{align*}
    Then by the injectivity of the function Berezin transform, we obtain,
    $ D(S\ast A)=DS\ast A.$
\end{proof}

As we note below, $D$ acting on $\mathrm{Dom}_{\mathrm{op}}(D)$ coincides with $\mathrm{d}\lambda^\nu(Z)$ on $\bdd^\infty$.

\begin{proposition}\label{prop:closed_extension}
    Let $D\in \D(G/K)$ be such that $D = \mathrm{d}\ell(Z)$ for some $Z\in Z(U(\gg^\C))$. Then $D$ acting on $\mathrm{Dom}_{\mathrm{op}}(D)$ is a closed extension of $\mathrm{d}\lambda^\nu(Z)$.
\end{proposition}
\begin{proof}
    By Lemma \ref{lem:centralZ} and Proposition \ref{prop:diffops_convol}, for $S\in \bdd^\infty$, we have that
    $$(DS)\ast \Phi^\nu=D(S\ast \Phi^\nu)=(\mathrm{d}\lambda^\nu(Z)S)\ast \Phi^\nu.$$
    Then by the injectivity of the Berezin transform $DS=\mathrm{d}\lambda^\nu(Z)S$. 

    To show that $D$ is closed on $\mathrm{Dom}_{\mathrm{op}}(D)$, let $\{(S_k,DS_k)\}$ be a sequence in the graph of $D$ converging to $(S,T)$ in $\bdd\times\bdd$, where $\bdd$ is considered with the operator norm topology. Then 
    \begin{align*}
        \lim_{k\to \infty} D(S_k\ast \Phi^\nu)&= \lim_{k\to \infty} (DS_k)\ast \Phi^\nu\\
        &= T\ast \Phi^\nu
    \end{align*}
    in $\bddf$.  Also, since $\Phi^\nu$ is a smooth vector, since $Z$ is central, by Proposition \ref{prop:diffops_convol}, 
    \begin{align*}
        \lim_{k\to \infty} D(S_k\ast \Phi^\nu)&= \lim_{k\to \infty} S_k\ast (\mathrm{d}\lambda^\nu(Z)\Phi^\nu)\\
        &= S\ast (\mathrm{d}\lambda^\nu(Z)\Phi^\nu)\\
        &= D(S\ast\Phi^\nu)
    \end{align*}
    in $\bddf$. Then by the uniqueness of the limit, $D(S\ast\Phi^\nu)=T\ast \Phi^\nu$. Therefore, $D S=T$ by the injectivity of the Berezin transform, proving that $D$ is a closed operator on $\mathrm{Dom}_{\mathrm{op}}(D)$.
\end{proof}

\begin{remark}
    {\rm In the above proposition, show that $D$ is an extension of $\mathrm{d}\lambda^\nu(Z)$, we did not use the fact that $Z$ is central (equivalently, $D$ is bi-invariant). However, the centrality of $Z$ was essential in proving that the resulting operator is closed. We also note that all these considerations behave well on the classical bounded symmetric domains.}
\end{remark}

\begin{theorem}
     Assume that $1\leq p<\infty$. Let $D\in \D(G/K)$ be such that $D = \mathrm{d}\ell(Z)$ for some $Z\in Z(U(\gg^\C))$. Then the operator $D$ acting on $\mathrm{Dom}_{\mathrm{op}}^p(D)$ is the closure of $\mathrm{d}\lambda^\nu(Z)$.
\end{theorem}

\begin{proof}
    First, note that since $\bsd$ is classical, there is $Z\in Z(U(\gg^\C))$ s.t. $D= \mathrm{d}\ell(Z)$ on smooth vectors.
    Let $S\in \mathrm{Dom}_{\mathrm{op}}^p(D)$, and let $\{\psi_j\}_{j}\subset C_c^\infty(G)$ be a Dirac sequence in $L^1(G)$. Then by Lemma \ref{lem:schop_convergence}, we obtain that $\psi_j\ast S\to S$ in $\schop$.
    Also, by \eqref{eq:DixMalli}, we have that 
    $\psi_j\ast S\in \schop^\infty$, and 
    \begin{align*}
        \mathrm{d}\lambda^\nu(Z)(\psi_j\ast S)&= D(\psi_j\ast S) \quad \text{(by Proposition \ref{prop:closed_extension})}\\
        &= \psi_j\ast DS \quad \text{(by Proposition \ref{prop:diffops_convol_dom})}.
    \end{align*}
    Now since $DS\in \schop$,  we have that $\psi_j\ast DS\to DS$ in $\schop$. Therefore, $(S, \mathrm{d}\lambda^\nu(Z)(\psi_j\ast S))\to (S,DS)$ in the graph norm, completing the proof.
\end{proof}

The key advantage of Su\'arez's definition is that it allows one to differentiate operators without relying explicitly on the 
machinery of representation theory. In particular, his approach defines the Laplacian of an operator directly 
through the Berezin transform, producing a closed operator on a larger natural domain than what arises 
from the derived representation.

\subsection{Shifting the $\alpha$-Berezin transform of an operator}
Recall the differential operator $D_\alpha$ in Section \ref{sec:Berezin_shifts}. Here, we note that $D_{\nu+\alpha}$ shifts the $\alpha$-Berezin transform of operators, analogous to Theorem \ref{theo:Berezin shift}, stated for functions.

\begin{theorem}\label{theo:Berezin shift op}
  Assume that $\bsd$ be a bounded symmetric domain and consider $\Berg=\cA^2_\nu$ with $\nu>\gen-1$. Let $\alpha\geq 0$, and let $S\in \bdd$. Then $\Phi^\nu_\alpha\in \mathrm{Dom}_{\mathrm{op}}^1(D_{\nu+\alpha})$ and
   $$D_{\nu+\alpha}\Phi^\nu_\alpha=\Phi^\nu_{\alpha+1},$$ 
   and the $\alpha$-Berezin transform satisfies:
   \[
      B^\nu_{\alpha+1}(S) = D_{\nu+\alpha} (B^\nu_\alpha (S))
   \]
   Furthermore, if $\bsd$ is classical domain and if $S\in \mathrm{Dom}_{\mathrm{op}}(D_{\nu+\alpha})$, then
   \[
      B^\nu_{\alpha+1}(S) = B^\nu_\alpha( D_{\nu+\alpha} S ).
   \]
\end{theorem}

\begin{proof}
    Note that from the construction of the $\alpha$-Berezin transform, we obtain that $B^\nu_\alpha(S)=S\ast\Phi^\nu_\alpha$ and $B^\nu(\Phi^\nu_\alpha)=\frac{1}{c_\nu}\varphi_{\nu+\alpha}$. Then it follows that
    $$D_{\nu+\alpha}
    B^\nu(\Phi^\nu_\alpha)=\frac{1}{c_\nu}D_{\nu+\alpha}
    \varphi_{\nu+\alpha}=\frac{1}{c_\nu} \varphi_{\nu+\alpha+1}=B^\nu(\Phi^\nu_{\alpha+1}).$$
    Hence, we obtain that $\Phi^\nu_\alpha\in \mathrm{Dom}_{\mathrm{op}}(D_{\nu+\alpha}
    )$ and that $D_{\nu+\alpha}
    \Phi^\nu_\alpha=\Phi^\nu_{\alpha+1}$. Then, by Proposition \ref{prop:diffops_convol_dom}, we obtain
    $$D_{\nu+\alpha} (B^\nu_\alpha (S))=D_{\nu+\alpha} (S\ast \Phi^\nu_\alpha)=S\ast (D_{\nu+\alpha}\Phi^\nu_\alpha)=S\ast \Phi^\nu_{\alpha+1}=B^\nu_{\alpha+1}(S).$$
    To prove the last statement, note that by Theorem \ref{theo:center and difops} and by the fact that $\bsd$ is classical, there is $Z\in Z(U(\gg^\C))$ s.t. $D_{\nu+\alpha}=\mathrm{d}\ell(Z)$ on smooth vectors. Therefore by Proposition \ref{prop:diffops_convol_dom}, 
    $$D_{\nu+\alpha} (B^\nu_\alpha (S))=D_{\nu+\alpha} (S\ast \Phi^\nu_\alpha)= (D_{\nu+\alpha} S)\ast \Phi^\nu_\alpha.$$
\end{proof}


\section{A QHA heat equation and a characterization of the radial Toeplitz algebra}\label{sec:heateq}
When the domain $\bsd=\B^n$, the unit ball in $\C^n$, function and operator convolutions with the approximate identity $\{\varphi_\alpha\}$ satisfy an equation reminiscent of the heat equation, with derivative replaced by a difference:
\begin{equation}\label{eq:heateq_rankone}
    \varphi_{\alpha}\ast S-\varphi_{\alpha+1}\ast S = \frac{1}{\alpha(\alpha-n)}\varphi_\alpha\ast \Tilde{\Delta} S,
\end{equation}
where $\widetilde{\Delta}$ is the invariant Laplacian (see \cite[Proposition 1.17]{Z05} for the formula).
This equation plays an interesting role in the study of Toeplitz operators. In \cite{S08}, Suarez used it to show that radial 
Toeplitz operators are dense in the Toeplitz algebra, a result later applied to the Gelfand theory of the radial Toeplitz algebra. 
Similar applications of the classical heat equation to the Toeplitz algebra on the Fock space were observed in \cite{DM24} as well. 
Our aim here is to investigate this phenomenon for Bergman spaces over classical bounded symmetric domains $\bsd$. While 
the invariant Laplacian alone suffices in the rank-one case, general domains require $r$ differential operators for this purpose, 
namely the generators of the algebra of bi-invariant differential operators $\D(G/K)$.

Let $\Delta_k$ denote the invariant differential operator corresponding to the elementary symmetric polynomial $\delta_k$. Let $\mathrm{Dom}_{\mathrm{op}}(\Delta_k)$ be the space of operators in the domain of $\Delta_k$. 
Let $\domop=\cap_{k=1}^r\mathrm{Dom}_{\mathrm{op}}(\Delta_k)$.
Recall that 
$$\varphi_{\alpha+1}=D_\alpha \varphi_\alpha=\frac{1}{d(\alpha)} \sum_{k=0}^r (\alpha - \frac{\gen-1}{2})^{2(r-k)} \Delta_k\varphi_\alpha.$$
Also, for $S\in \domop$,
\begin{align*}
    \varphi_{\alpha+1}\ast S-\varphi_{\alpha}\ast S &= (D_\alpha\varphi_{\alpha})\ast S- \varphi_{\alpha}\ast S\\
    &= D_\alpha(\varphi_\alpha\ast S)- \varphi_\alpha\ast S \quad \text{( by Proposition \ref{prop:diffops_convol})}\\
    &=\varphi_{\alpha}\ast D_{\alpha}S-\varphi_{\alpha}\ast S \quad \text{( by Proposition \ref{prop:diffops_convol_dom})}\\
    &=\varphi_\alpha \ast (D_\alpha-\id)S.
\end{align*}
Then it follows that
\begin{equation}\label{eq:heat eq}
     \varphi_{\alpha+1}\ast S-\varphi_{\alpha}\ast S = \varphi_\alpha\ast \bigg(\Big(\frac{(\alpha - \frac{\gen-1}{2})^{2r}}{d(\alpha)}-1\Big)\id+\frac{1}{d(\alpha)} \sum_{k=1}^r (\alpha - \frac{\gen-1}{2})^{2(r-k)} \Delta_k \bigg)S
\end{equation}

\begin{example}
    In the case of the unit ball $\B^n$, we get $p=n+1$, $r=1$, $a=2$, $b=n-1$. Then $\Delta_1=-\Tilde{\Delta}-\rho_1^2$ (see \cite{U94}), where $\Tilde{\Delta}$ denotes the invariant Laplacian, and moreover, 
    $$d(\alpha)=(\alpha-n/2)^2-\rho_1^2=(\alpha-n/2)^2-n^2/4=\alpha(\alpha-n).$$
    Therefore, we obtain the well-known identity \eqref{eq:heateq_rankone}:
\begin{align*}
    \varphi_{\alpha+1}\ast S-\varphi_{\alpha}\ast S&= \varphi_\alpha \ast \bigg(\Big(\frac{(\alpha - \frac{n}{2})^{2}}{d(\alpha)}-1\Big)\id+\frac{1}{d(\alpha)}   \Delta_1 \bigg)S\\
    &= \frac{-1}{d(\alpha)} \varphi_\alpha\ast \Tilde{\Delta}S\\
    &=\frac{-1}{\alpha(\alpha-n)}\varphi_\alpha\ast \Tilde{\Delta} S.
\end{align*}
\end{example}

\begin{lemma}\label{lem:asymptotics}
    Let $S\in \domop$. Then there is a constant $A$ such that 
    $$\|\varphi_{\alpha+1}\ast S-\varphi_{\alpha}\ast S\|\leq \frac{A}{\alpha^2}.$$
\end{lemma}
\begin{proof}
    First note that by \eqref{eq:heat eq} and the QHA Young's inequality, we get that
    \begin{align*}
        \|\varphi_{\alpha+1}\ast S-\varphi_{\alpha}\ast S\|\leq \Big|\frac{(\alpha - \frac{\gen-1}{2})^{2r}}{d(\alpha)}-1\Big|\|S\|+ \sum_{k=1}^r \frac{|\alpha - \frac{\gen-1}{2}|^{2(r-k)}}{d(\alpha)} \|\Delta_k S\|.
    \end{align*}
    Using the formula for $d(\alpha)$ \eqref{eq:dnu}, we obtain that
    $$\frac{(\alpha - \frac{\gen-1}{2})^{2r}}{d(\alpha)}=\frac{1}{\prod_{j=1}^r 1 - \frac{\rho_j^2}{(\alpha-\frac{\gen-1}{2} )^2} } \sim 1\pm \frac{\sum_{j=1}^r\rho_j^2}{(\alpha-\frac{\gen-1}{2} )^2}\pm O(1/\alpha^4).$$
    Therefore,
    $$\Big|\frac{(\alpha - \frac{\gen-1}{2})^{2r}}{d(\alpha)}-1\Big|\sim \frac{\sum_{j=1}^r\rho_j^2}{\alpha^2}.$$
    It is also clear from \eqref{eq:dnu} that
    $$\frac{|\alpha - \frac{\gen-1}{2}|^{2(r-k)}}{d(\alpha)} = O(1/\alpha^2)$$
    for $k=1,\dots r$. Now by choosing $A$ to be a sufficiently large constant based on the above asymptotics and the norms $\|\Delta_k S\|$, the result follows.
\end{proof}

\begin{theorem}\label{theo:main theorem two}
    Consider $\Berg=\cA^2_\nu$, where $\bsd$ is a classical bounded symmetric domain, and $\nu>\gen-1$. Then the following holds:
    $$(\overline{\bdd^\infty}^{\|\cdot\|})^K=(\overline{\domop}^{\|\cdot\|})^K=\cbuopL^K=\toepalg^K.$$
\end{theorem}

\begin{proof}
    First we show that $(\domop)^K\subseteq \cbuopL^K$. Let $S\in \domop$. It is enough to show that $\varphi_\alpha\ast S\to S$ in operator norm. Note that for $\alpha,\beta\in \N$ such that $\alpha>\beta$, we have
    \begin{align*}
        \|\varphi_{\alpha}\ast S -\varphi_\beta\ast S\|&=   \Big\|\sum_{k=\beta}^{\alpha-1} \varphi_{k+1}\ast S-\varphi_{k}\ast S \Big\|\\
        &\leq  \sum_{k=\beta}^{\alpha-1} \|\varphi_{k+1}\ast S-\varphi_{k}\ast S\|.
    \end{align*}
    Hence, by Lemma \ref{lem:asymptotics}, the sequence
    $\{\varphi_\alpha\ast S\}$ is Cauchy, and converges in norm to some operator $T\in \bdd$. Now we prove that $T=S$. First, note that $\varphi_\alpha\ast (S\ast \Phi^\nu)\to T\ast \Phi^\nu$ in $\|\cdot\|_\infty$ by the Young's inequality and convergence $\varphi_\alpha\ast S\to T$. But also, since $S\ast \Phi^\nu \in \cbufR$ , we obtain that $(S\ast \Phi^\nu)\ast \varphi_\alpha \to S\ast \Phi^\nu$. But, since $S$ is radial, $(S\ast \Phi^\nu)\ast \varphi_\alpha=\varphi_\alpha\ast (S\ast \Phi^\nu)$.  Then by the uniqueness of the limit, $T\ast \Phi^\nu=S\ast \Phi^\nu$ a.e. and hence, $S=T$ by the injectivity of the Berezin transform. Thus $S\in \cbuopL$.
    
    Now the density follows from the fact that
    $$(\bdd^\infty)^K\subseteq (\domop)^K\subseteq \cbuopL^K$$
    and $(\bdd^\infty)^K$ is dense in $\cbuopL^K$. We also recall that for a subset $V\subseteq\bdd$, $\overline{V^K}^{\|\cdot\|}=(\overline{V}^{\|\cdot\|})^K$, as an application of the $K$-average map (radialization). Furthermore, $\cbuopL^K=\toepalg^K$ by Theorem \ref{theo: main theorem}.
\end{proof}

\section*{Appendix: Review of the Classical Domains}\label{sec:classical_dom}

In this section, we remind the reader of the definitions of the classical bounded symmetric domains and their symmetry groups. We also provide a table which includes the ranks, root multiplicities, and genus for each of these domains.  All of the material in this section is very well known.  We start with some matrix calculations for the group $\SU (p,q)$ related to the
{\it Harish-Chandra} realization of bounded symmetric domains.

\subsection{The group $\SU (p,q)$}
For $p,q\in \N$, $n=p+q$ let $I_{p,q}=\mathrm{diag} (I_p,-I_q)$ where $\mathrm{diag}$ stands for {\it diagonal matrix} and
$I_m$ is the $m\times m$-identity matrix $I_m = \mathrm{diag}(1,1,\ldots , 1)$. Define
a map $\theta : \GL(n,\C)\to \GL(n,\C)$ by $ \theta (x) = (x^*)^{-1}$. On the Lie algebra level, we let $\dot \theta (X) = - X^*$.
Finally we let $\tau (x) = I_{p,q} \theta (x) I_{p,q}$
and similarly $\dot\tau (X) = -I_{p,q} X^*I_{p,q}$. In this notation we have
\[\SU (p,q) = \{x\in \SL(n,\C)\mid \tau (x) =x\}\]
with Lie algebra
\[\su (p,q ) = \{X\in \sl (n,\C)\mid I_{p,q}X +X^*I_{p,q} =0\} = \{X\in
\sl (n,\C) \mid \dot\tau (X)=X\}. \]
We let
\begin{align*} K &=\SU(p,q)^\theta=\left\{\left. k(A,D)=\begin{pmatrix} A & 0 \\ 0 & D\end{pmatrix}
\, \right|\, A \in \rU (p), D\in \rU (q)\text{ and } \det A \det D = 1\right\}\\
&= \rS (\rU (p)\times \rU(q)).
\end{align*}
Then $G/K$ is a bounded symmetric domain. 

The group $G$ has a complexification $G_\C = \SL (n,\C)$. The group $K$
has also a complexification 
\begin{align*}
K_\C &= \{k(A,D)\mid A\in\GL(p,\C), D\in \GL(q,\C), \text{ and } \det A \det D = 1\}\\
&=\rS(\GL(p,\C)\times \GL (q,\C)) .
\end{align*}

We need the following subgroup of $G_\C$:
\begin{align*}
P^+ &= \left\{\left. p^+(Z) = \begin{pmatrix} I_p & Z\\ 0 & I_q\end{pmatrix}\,
\right| \, Z\in M_{p,q}(\C)\right\}\\
P^+ &= \left\{\left. p^-(W)=\begin{pmatrix} I_p & 0\\ W & I_q\end{pmatrix}\,
\right| \, W\in M_{q,p}(\C)\right\} .
\end{align*}

A simple matrix multiplication shows that
\[\begin{pmatrix}  A & B\\ C & D\end{pmatrix}
= p^+(Z) k(\alpha ,\beta ) p^- (W) =\begin{pmatrix} x + ZyW & Zy \\
yW & y\end{pmatrix}  \]
and this can be solved uniquely if and only if $\det D \not=0$, which
is always the case if the left-hand side is in $\SU (p,q)$, so 
$G \subset P^+K_\C P^-$. This also shows that  $P^+ K_\C P^-$ is open and dense in $G_\C$ and if $\begin{pmatrix}  A & B\\ C & D\end{pmatrix}
=P^+(Z) k(x,y) p^- (W)\in P^-K_\C P^+$ then
\begin{equation}\label{eq:decomp}
 y = D,\quad Z= BD^{-1},\quad W= D^{-1}C\quad\text{and}\quad
x= A - BD^{-1}C .
\end{equation}

Let $x= \begin{pmatrix}  A & B\\ C & D\end{pmatrix}$ and $Z\in M_{p,q}(\C)$ such that
\[\begin{pmatrix}  A & B\\ C & D\end{pmatrix}p^+(Z)  
=\begin{pmatrix} A & AZ+B\\ C & CZ + D\end{pmatrix} =:
p^{+} (x\cdot Z) k(x(x,Z),y(x,Z)) p^-
\in P^- K_\C P^+\]
then
\smallskip

\noindent
{\rm 1)} $x\cdot Z = (AZ + B)(CZ+D)^{-1}$.
\smallskip

\noindent
{\rm 2)} $y(x,Z) = CZ +D$
which shows how some of the formulas in the following subsections
are derived. Furthermore we have
\[p^-{(- W)}^*p^+(Z) = \begin{pmatrix} I_p &Z \\ -W^+ & I_q- W^*Z\end{pmatrix}. \]
Hence the lower corner in the $K_\C$ part is $I_q -W^*Z$. Finally we note that
all the homomorphism $\chi : K_\C\to \C$ are given by
\[\chi_\lambda \left(\begin{pmatrix} A & 0 \\ 0 & D\end{pmatrix}\right)
= (\det D)^\lambda = (\det A)^{-\lambda} \quad\text{for some } \lambda \in \C .\]

\subsection{Cartan Domains of Type I}
Suppose that $p,q\in\N$. Then the domain $D^\mathrm{I}_{p,q}$ consists of all complex $p\times q$ matrices with operator norm less than one.  That is: 
\begin{align*}
   D^\text{I}_{p,q} = \{A\in M(p,q,\C) \mid \|A\|_\mathrm{op} < 1\} 
                     = \{A\in M(p,q,\C) \mid AA^* < I \},
\end{align*}
where $M(p,q,\C)$ denotes the space of all $p\times q$ matrices over $\C$, and where $AA^* < I$ means that $I-AA^*$ is a strictly positive matrix.  The group $G=\SU(p,q)$ acts by biholomorphisms on this domain as follows:
\[
    \left( \begin{matrix}
       A & B \\
       C & D
 \end{matrix} \right) \cdot Z := (AZ + B)(CZ+D)^{-1}
\]
for all $\left( \begin{matrix}
       A & B \\
       C & D
 \end{matrix} \right) \in \SU(p,q)$ and all $Z\in D^\mathrm{I}_{p,q}$, where we have written matrices in $\SU(p,q)$ block form.

The Jordan triple determinant is given by:
\[
   h(Z,W) = \det(I_n - W^*Z)
\]
for all $Z,W\in D^\mathrm{I}_{m}$,  
and the complex Jacobian of the action of $\SU(m,n)$ is given by:
\[
    j( \left(\begin{matrix}
                A & B \\
                C & D 
             \end{matrix}\right), Z) = \det(CZ+D)^{-\gen},
\]
and so one can obviously choose the holomorphic $\gen$th root 
\[
    j( \left(\begin{matrix}
                A & B \\
                C & D 
             \end{matrix}\right), Z)^{1/\gen} = \det(CZ+D)^{-1}.
\]

\subsection{Cartan Domains of Type II}
Suppose that $m\in\N$.  Then the domain $D^\mathrm{II}_{m}$ consists of all complex, symmetric $m\times m$ complex matrices of operator norm less than one:
\begin{align*}
   D^\text{II}_{m} = \{A\in M(m,\C) \mid A = A^T, \|A\|_\mathrm{op} < 1\} 
                    = \{A\in M(m,\C) \mid A = A^T, AA^* < I \}.
\end{align*}
where $M(m,\C)$ denotes the space of all $m\times m$ matrices over $\C$.

 The group 
\[
   \Sp(m,\R) := \left\{ g \in \SU(m,m) \left|\, g
   \left( \begin{matrix}
      0 & I \\
      -I & 0 
   \end{matrix} \right) g^T = \left( \begin{matrix}
      0 & I \\
      -I & 0
   \end{matrix} \right) \right. \right\}
\]
acts by biholomorphisms on this domain as follows:
\[
   \left( \begin{matrix}
       A & B \\
       C & D
   \end{matrix} \right) \cdot Z := (AZ + B)(CZ+D)^{-1}
\]
for all $\left( \begin{matrix}
       A & B \\
       C & D
\end{matrix} \right) \in \Sp(m,\R)$ and all $Z\in D^\mathrm{III}_{p,q}$, where we have written matrices in $\Sp(m,\R)$ with blocks of size $m\times m$.  

Note that the definition here is not the usual realization of the group $\Sp(m,\R)$, but it is isomorphic to it.

The Jordan triple determinant is given by:
\[
   h(Z,W) = \det(I_n - W^*Z)
\]
for all $Z,W\in D^\mathrm{II}_{m}$,  
and the complex Jacobian of the action of $\SO^*(2m)$ is given by:
\[
    j( \left(\begin{matrix}
                A & B \\
                C & D 
             \end{matrix}\right), Z) = \det(CZ+D)^{-\gen},
\]
and so one can obviously choose the holomorphic $\gen$th root 
\[
    j( \left(\begin{matrix}
                A & B \\
                C & D 
             \end{matrix}\right), Z)^{1/\gen} = \det(CZ+D)^{-1}.
\]

\subsection{Cartan Domains of Type III}
Suppose that $m\in\N$. Then the domain $D^\mathrm{III}_{m}$ consists of all complex, antisymmetric $m\times m$ complex matrices of operator norm less than one: 
\begin{align*}
   D^\text{III}_{m} = \{A\in M(m,\C) \mid A = -A^T, \|A\|_\mathrm{op} < 1\} 
                    = \{A\in M(m,\C) \mid A = -A^T, AA^* < I \}.
\end{align*}
where $M(m,\C)$ denotes the space of all $m\times m$ matrices over $\C$. 
 The group 
\[
   \SO^*(2m) := \left\{ g \in \SU(m,m) \left|\, g
   \left( \begin{matrix}
      0 & I \\
      I & 0
   \end{matrix} \right) g^T = \left( \begin{matrix}
      0 & I \\
      I & 0
   \end{matrix} \right) \right. \right\}
\]
acts by biholomorphisms on this domain as follows:
\[
   \left( \begin{matrix}
      A & B \\
      C & D
   \end{matrix} \right) \cdot Z := (AZ + B)(CZ+D)^{-1}
\]
for all $\left( \begin{matrix}
       A & B \\
       C & D
 \end{matrix} \right) \in \SO^*(2m)$ and all $Z\in D^\mathrm{II}_{m}$, where we have written matrices in $\SO*(2m)$ with blocks of size $m\times m$.

The Jordan triple determinant is given by:
\[
   h(Z,W) = \det(I_n - W^*Z)
\]
for all $Z,W\in D^\mathrm{III}_{m}$,  

and the complex Jacobian of the action of $\Sp(m,\R)$ is given by:
\[
    j( \left(\begin{matrix}
                A & B \\
                C & D 
             \end{matrix}\right), Z) = \det(CZ+D)^{-\gen},
\]
and so one can obviously choose the holomorphic $\gen$th root 
\[
    j( \left(\begin{matrix}
                A & B \\
                C & D 
             \end{matrix}\right), Z)^{1/\gen} = \det(CZ+D)^{-1}.
\]

\subsection{Cartan Domains of Type IV}
We will only discuss the last
classical domain, the   {\it Lie ball} $D^{IV}_m$, $m\geq 3$, very briefly.  

The Lie ball is the symmetric domain
\[
D^{IV}_m
=
\left\{
z\in\mathbb C^m :
\|z\|^2<1,\quad
1-2\|z\|^2+|z^Tz|^2>0
\right\}.
\]

This seems not to fit so well into the above discussion, but denoting by $\sigma$
the complex conjugation on $\SU (2,m)$ and the domain $D^I_{2,m}$ we have
\[\SU (2,m)^\sigma_e = \SO (2,m)_e \quad\text{and}\quad
(D^I_{2,m})^\sigma = \{Z\in M_{2.m} \mid I_m - Z^t Z >0\}.\]
This is a real bounded domain, but by writing 
\[Z = \begin{pmatrix} x \\ y\end{pmatrix}, \quad x,y\in \R^m\]
we can map the domain into $\C^m$ by $Z \mapsto x+ iy$. The image
is then the Lie ball.

The group $G=SO_0(2,m)$, preserving $\operatorname{diag}(I_2,-I_m)$,
acts transitively by biholomorphisms, with stabilizer
$K=SO(2)\times SO(m)$.

\begin{figure}
   \caption{Data for the classical bounded symmetric domains.}
   \begin{tabular}{|c||c|c|c|c|c|c|c|}
      \hline
      Domain & $G$ & $K$ & $r$ & $a$ & $b$ & $g$ \\
      \hline
      $D^\text{I}_{p,q}$ & $\SU(p,q)$ & $S(\U(p)\times\U(q))$ & $\min\{p,q\}$ & 2 & $|p-q|$ & $p+q$\\
      \hline
      $D^\text{II}_{2m}$ & $\SO^*(4m,\R)$ & $\U(2m)$ & $2m$ & $4$ & $0$ & $8m-2$\\
      \hline
      $D^\text{II}_{2m+1}$ & $\SO^*(2(2m+1),\R)$ & $\U(2m+1)$ & $2m+1$ & $4$ & $2$ & $8m+4$\\
      \hline
      $D^\text{III}_m$ & $\Sp(m,\R)$ & $\U(m)$ & $m$ & $1$ & $0$ & $m+1$ \\
      \hline
      $D^\text{IV}_m$ & $\SO_0(2,m)$ & $\SO(2)\times \SO(m)$ & $2$ & $m-2$ & $0$ & $m$\\
      \hline
   \end{tabular}
\end{figure}

\bibliographystyle{alpha}

\begin{thebibliography}{BHYnV14b}

\bibitem[BF20]{BF20}
W.~Bauer and R.~Fulsche. Berger-{C}oburn theorem, localized operators, and the {T}oeplitz
  algebra. 
  {\em Oper. Theory Adv. Appl.},
 {\bf 279}   53--77.
  Birkh\"auser/Springer  2020.

\bibitem[BFT25]{BFT25}
W. Bauer, R. Fulsche, and J. Toft.
 Convolutions of Orlicz spaces and Orlicz Schatten classes, with
  applications to Toeplitz operators,  arXiv:2505.01707, 2025.

\bibitem[BHYV14a]{BHV14b}
W.~Bauer, C.~Herrera Ya\~nez, and N.~Vasilevski.
Eigenvalue characterization of radial operators on weighted {B}ergman
  spaces over the unit ball.
  {\em Integral Equations Operator Theory}, {\bf 78} (2) (2014), 271--300.

\bibitem[BHYV14b]{BHV14}
W.~Bauer, C.~Herrera Ya\~nez, and N.~Vasilevski.
\newblock {$(m,\lambda)$}-{B}erezin transform and approximation of operators on
  weighted {B}ergman spaces over the unit ball.
{\em Oper. Theory Adv. Appl.}, {\bf 240} (2014), 45--68.
  Birkh\"auser/Springer, Cham, 2014.

\bibitem[BL76]{BerghLofstrom76}
J.  Bergh and J. L{\"o}fstr{\"o}m.
 {\em Interpolation Spaces: An Introduction}, {\em
  Grundlehren der mathematischen Wissenschaften}, {\bf 223}
Springer-Verlag, Berlin-Heidelberg-New York, 1976.

\bibitem[DDMO25]{DDMO24}
M. Dawson, V. Dewage, M. Mitkovski, and G.  \'Olafsson.
Quantum harmonic analysis on the unweighted {B}ergman space of the
  unit ball.
{\em Integral Equations Operator Theory}, {\bf 97}(3), (2025), Paper No. 23,  51 pp.

\bibitem[Dew26]{D25}
V. Dewage.
Toeplitz algebra of bounded symmetric domains: A quantum harmonic
  analysis approach via localization.
 {\em J. Funct. Anal.}, {\bf 290} (2) (2026), paper no. 111211, 32 pp.

\bibitem[DFO25]{DFO25}
V. Dewage, R. Fulsche, and G.  \'Olafsson.
The function-operator convolution algebra over the Bergman space of
  the ball and its Gelfand theory, 2025. arXiv:2512.24746. To appear in
  {\em Contemporary Math.}.

\bibitem[DM25a]{DM24}
V. Dewage and M. Mitkovski.
 The {L}aplacian of an operator and the radial {T}oeplitz algebra.
 {\em J. Math. Anal. Appl.}, {\bf 547}(1) (2025), Paper No. 129245, 19 pp. 

\bibitem[DM25b]{DM24a}
V. Dewage and M.  Mitkovski.
  A quantum harmonic analysis approach to the {B}erger-{C}oburn
  theorem.
 {\em New York J. Math.}, {\bf 31} (2025), 887--901.

\bibitem[D17]{D17}
A.~Dieudonne.
{$m$}-{B}erezin transform and approximation of operators on the
  {B}ergman space over bounded symmetric domains.
\newblock {\em Complex Anal. Oper. Theory} {\bf 11}(3) (2017), 651--674.

\bibitem[DLMS24]{DLMS24}
M. D{\"o}rfler, F. Luef, H. McNulty, and E. Skrettingland.
Time-frequency analysis and coorbit spaces of operators.
 {\em J. Math. Anal. Appl.}, {\bf 534}(2) (2024), 33 pages. 
 

\bibitem[DM78]{DixmierMalliavin78}
J. Dixmier and P. Malliavin.
 Factorisations de fonctions et de vecteurs ind\'efiniment
  diff\'erentiables.
  {\em Bull. Sci. Math. (2)}, {\bf 102}(4) (1978), 307--330.


\bibitem[DOQB15]{DOQ15}
M.~Dawson, G.~\'{O}lafsson, and R.~Quiroga-Barranco. Commuting {T}oeplitz operators on bounded symmetric domains and
 {\em J. Funct. Anal.}, {\bf 268}(7) (2015), 1711--1732.

\bibitem[DOQB18]{DOQ18}
M.~Dawson, G.~\'{O}lafsson, and R.~Quiroga-Barranco.
 The restriction principle and commuting families of {T}oeplitz
  operators on the unit ball.
 {\em S\~{a}o Paulo J. Math. Sci.}, {\bf 12}(2) (2018), 196--226.

\bibitem[Eng99]{E99}
M. Engli\v{s}.
 Compact {T}oeplitz operators via the {B}erezin transform on bounded
  symmetric domains.
  {\em Integral Equations Operator Theory}, (1999), {\bf 33}(4), 426--455.

\bibitem[FK90]{FK90}
J.~Faraut and A.~Kor\'anyi.
\newblock Function spaces and reproducing kernels on bounded symmetric domains.
\newblock {\em J. Funct. Anal.} {\bf 88}(1) (1990), 64--89.


\bibitem[Ful20]{F19}
R.~Fulsche. Correspondence theory on {$p$}-{F}ock spaces with applications to
  {T}oeplitz algebras.
 {\em J. Funct. Anal.}, {\bf 279}(7) (2020), Paper No. 108661, 41 pp.


\bibitem[FG23]{FG23}
R.~Fulsche and N.~Galke.
 Quantum harmonic analysis on locally compact abelian groups. {\em J. Fourier Anal. Appl.} {\bf 31} (2025), Paper No. 13, 
 58 pp. 
 
\bibitem[FH23]{FH23}
R.~Fulsche and R.~Hagger.
\newblock Quantum harmonic analysis for polyanalytic {F}ock space.
{\em J. Fourier Anal. Appl.} {\bf 30} (2024), Paper No. 63, 42 pp. 


\bibitem[FLW26]{FLW26}
R.  Fulsche, F. Luef, and R. F. Werner.  Wiener's {T}auberian theorem in classical and quantum harmonic
  analysis.
  {\em J. Funct. Anal.}, {\bf 290}(4) (2026), Paper No. 111265, 48 pp.


\bibitem[FvL25]{FL25}
R. Fulsche and L. van Luijk.
 Heisenberg-smooth operators from the phase-space perspective.
 {\em Mathematische Nachrichten}, {\bf 298}(8) (2025), 2845--2866.

\bibitem[FRR23]{FR23}
R.~Fulsche and M.~A. Rodriguez~Rodriguez.
  Commutative {G}-invariant Toeplitz ${C}^*$-algebras on the {F}ock
  space and their {G}elfand theory through quantum harmonic analysis, {\em J. Operator Theory} {\bf 93} (2025), 593--620.


\bibitem[G64]{G64}
S. G. Gindikin.
 Analysis in homogeneous domains.
 {\em Russian Math Surveys}, {\bf 19}(4) (1964), 1--90.

\bibitem[G69]{Goodman69}
R. W. Goodman.
 Analytic and entire vectors for representations of {L}ie groups.
  {\em Trans. Amer. Math. Soc.}, {\bf 143} (1969), 55--76.

\bibitem[GMV13]{GMV13}
S.~M. Grudsky, E.~A. Maximenko, and N.~L. Vasilevski.
 Radial {T}oeplitz operators on the unit ball and slowly oscillating
  sequences.
 k {\em Commun. Math. Anal.}, {\bf 14}(2) (2013), 77--94.


\bibitem[G47]{Garding47}
Lars G{\aa}rding.  Note on continuous representations of {L}ie groups.  {\em Proc. Nat. Acad. Sci. U.S.A.} {\bf 33} (1947),
331--332.

\bibitem[Ha17]{H17}
R.~Hagger.  The essential spectrum of {T}oeplitz operators on the unit ball.  {\em Integral Equations Operator Theory}, {\bf 89}(4) (2017), 519--556.

\bibitem[Ha19]{H19}
R.~Hagger. Limit operators, compactness and essential spectra on bounded
  symmetric domains.
 {\em J. Math. Anal. Appl.}, {\bf 470}(1) (2019), 470--499.

\bibitem[Hal23]{H23}
S.~Halvdansson.  Quantum harmonic analysis on locally compact groups. 
 {\em Journal of Functional Analysis}, {\bf 285}(8) (2023), Paper No. 110096.

\bibitem[HC54]{H54}
Harish-Chandra.  Representations of semisimple lie groups: VI  {\em Proc.\ Nat.\ Acad.\ Sci.}, {\bf 40}(11) (1954), 
1078--1080.

\bibitem[Hel78]{H78}
S.~Helgason.
  {\em Differential Geometry, {L}ie groups, and Symmetric Spaces}, {\em Pure Appl. Math.}
{\bf 80}.  Academic Press, Inc.  New
  York-London, 1978.

\bibitem[Hel84]{H84}
S.~Helgason. {\em Groups and Geometric Analysis: Integral Geometry, Invariant
  Differential Operators, and Spherical Functions}, {\em Pure Appl. Math.} {\bf 113}
\newblock Academic Press, Inc. [Harcourt Brace Jovanovich, Publishers], New
  York-London, 1984.

\bibitem[Hua63]{H63}
L.~K. Hua.  {\em Harmonic Analysis of Functions of Several Complex Variables in
  the Classical Domains}.  American Mathematical Society, Providence, RI, 1963.
  

\bibitem[K69]{K69}
M.~Koecher.  {\em An Elementary Approach to Bounded Symmetric Domains}. Rice University, Houston, TX, 1969.

\bibitem[Laf24]{L24}
L. Lafleche. On quantum Sobolev inequalities.  {\em Journal of Functional Analysis} {\bf 286}(10), Paper No. 110400,
40 pp.

\bibitem[Leb72]{Lebedev72}
N.~N. Lebedev.  {\em Special Functions and their Applications}. Dover Publications, Inc., New York, revised edition, 1972. 

\bibitem[Loo69]{L69}
O.~Loos.  {\em Symmetric Spaces. {I}: {G}eneral theory}.  W. A. Benjamin, Inc., New York-Amsterdam, 1969.

\bibitem[Loo77]{L77}
O.~Loos.  {\em Symmetric Spaces and {Jordan Pairs}}.
  Mathematical Lectures, 1977.


\bibitem[LM24]{LM24}
F.  Luef and H, McNulty.
 Quantum {Time}-{Frequency} {Analysis} and {Pseudodifferential}
  {Operators}. {\em Indiana University. Math. J.} {\bf 75}(2) (2026), 587--623. 

\bibitem[LS25]{LS25}
F. Luef and H. J. Samuelsen.  Fourier restriction for {S}chatten class operators and functions on
  phase space. {\em Int. Math. Res. Not. IMRN}, (2):Paper No. rnae291, 22, 2025.
  
\bibitem[LS18]{LS18}
F.~Luef and E.~Skrettingland. Convolutions for localization operators.  {\em J. Math. Pures Appl. (9)}, {\bf 118} (2018), 288--316.


\bibitem[LS21]{LS20}
F.~Luef and E.~Skrettingland. A {W}iener {T}auberian theorem for operators and functions.  {\em J. Funct. Anal.}, 
{\bf 280}(6) (2021), Paper No. 108883, 44 pp. 

\bibitem[MO26]{MO26}
M. Mitkovski, and Y. Ouyang. Local dyadic fractional Sobolev spaces: paraproducts, commutators, and the algebra property.
arXiv:2603.26565.
  
\bibitem[Nel59]{Nelson59}
E. Nelson. Analytic vectors.  {\em Ann. of Math. (2)}, {\bf 70} (1959), 572--615.

\bibitem[P72]{Poulsen72}
N. S. Poulsen.  On {$C\sp{\infty }$}-vectors and intertwining bilinear forms for
  representations of {L}ie groups. {\em J. Functional Analysis}, {\bf 9} (1972), 87--120.

\bibitem[S80]{S80}
I.~Satake.  {\em Algebraic Structures of Symmetric Domains},  {\em
  Kan\^o{} Memorial Lectures} {\bf 4}. Iwanami Shoten, Tokyo; Princeton University Press, Princeton, NJ,
  1980.

\bibitem[Su{\'a}04]{S04}
D.~Su{\'a}rez.  Approximation and symbolic calculus for {T}oeplitz algebras on the
  {B}ergman space.  {\em Rev. Mat. Iberoamericana}, {\bf 20}(2) (2004), 563--610.

\bibitem[Su{\'a}05]{S05}
D.~Su{\'a}rez. Approximation and the {$n$}-{B}erezin transform of operators on the
  {B}ergman space. {\em J. Reine Angew. Math.}, {\bf 581} (2005), 175--192.

\bibitem[Su{\'a}08]{S08}
D.~Su{\'a}rez.  The eigenvalues of limits of radial {T}oeplitz operators.
 {\em Bull. Lond. Math. Soc.}, {\bf 40}(4) (2008), 631--641.

\bibitem[Upm83]{U83}
H.~Upmeier.  Toeplitz operators on bounded symmetric domains.  {\em Trans. Amer. Math. Soc.},
{\bf 280}(1) (1983), 221--237.

\bibitem[Upm21]{U21}
H. Upmeier.  Eigenvalues of {$K$}-invariant {T}oeplitz operators on bounded
  symmetric domains.  {\em Integral Equations Operator Theory}, {\bf 93}(3) (2021), Paper No. 27, 21 pp. 

\bibitem[UU94]{U94}
A.~Unterberger and H.~Upmeier. The Berezin transform and invariant differential operators.  
{\em Comm. in Math. Phys.}, {\bf 164}(3) (1994), 563--597.

\bibitem[Var74]{Varadarajan74}
V.~S. Varadarajan.  {\em Lie groups, {L}ie algebras, and their Representations}.
 Prentice-Hall Series in Modern Analysis. Prentice-Hall, Inc.,
  Englewood Cliffs, NJ, 1974.

\bibitem[W84]{W84}
R.~Werner.  Quantum harmonic analysis on phase space. 
{\em J. Math. Phys.}, {\bf 25}(5) (1984), 1404--1411.

\bibitem[Xia15]{X15}
J.~Xia.  Localization and the {T}oeplitz algebra on the {B}ergman space.  {\em J. Funct. Anal.}, 
{\bf 269}(3) (2015), 781--814.

\bibitem[Zhu05]{Z05}
K.~Zhu.  {\em Spaces of Holomorphic Functions in the Unit Ball},   {\em Graduate Texts in Mathematics} {\bf 226} 
 Springer-Verlag, New York, 2005.

\end{thebibliography}

\end{document}